\documentclass[12pt,a4paper,twoside,reqno]{amsart}
\title[Canonical metrics on holomorphic Courant algebroids]{Canonical metrics on \\holomorphic Courant algebroids}
\author[M. Garcia-Fernandez]{Mario Garcia-Fernandez}
\author[R. Rubio]{Roberto Rubio}
\author[C. S. Shahbazi]{\\ C. S. Shahbazi}
\author[C. Tipler]{Carl Tipler}

\address{Dep. Matem\'aticas, Universidad Aut\'onoma de Madrid, and Instituto de Ciencias Matem\'aticas (CSIC-UAM-UC3M-UCM), Cantoblanco, 28049 Madrid, Spain}
\email{mario.garcia@icmat.es}

\address{Weizmann Institute of Science, 234 Herzl St, Rehovot 76100, Israel}
\curraddr{Universitat de Barcelona, 08007 Barcelona, Spain}
\email{roberto.rubio@ub.edu}
\address{Department of Mathematics, University of Hamburg, Bundesstra{\ss}e 55, 20146 Germany}
\email{carlos.shahbazi@uni-hamburg.de}
\address{LMBA, UMR CNRS 6205; D\'epartement de
  Math\'ematiques, Universit\'e de Bretagne Occidentale, 6, avenue
  Victor Le Gorgeu, 29238 Brest Cedex 3 France}
\email{carl.tipler@univ-brest.fr}

\thanks{This project has received funding from the European Union's Horizon 2020 research and innovation programme under the Marie Sklodowska-Curie grant agreement No 655162. MGF was partially supported by a Marie Sklodowska-Curie grant, the Universidad Aut\'onoma de Madrid and CSIC, under grant i-Link$1107$. RR was supported by ISF grant 687/13, a grant from the Minerva foundation and ERC StG grant 637912, and was initially supported by IMPA. CS is supported by a Humboldt Research Fellowship from the Alexander Von Humboldt Foundation. CT is supported by the French government ``Investissements d'Avenir'' program ANR--11--LABX--0020--01 and ANR project EMARKS No ANR--14--CE25--0010.}

\usepackage{hyperref,url}
\usepackage{amssymb,latexsym}
\usepackage{amsmath,amsthm}
\usepackage[latin1]{inputenc}
\usepackage{enumerate}
\usepackage[all]{xy} \CompileMatrices
\usepackage{amscd}

\usepackage{slashed} 

\SelectTips{cm}{12} 
\theoremstyle{plain}
\newtheorem{theorem}{Theorem}[section]
\newtheorem{lemma}[theorem]{Lemma}
\newtheorem{corollary}[theorem]{Corollary}
\newtheorem{proposition}[theorem]{Proposition}

\theoremstyle{definition}
\newtheorem{definition}[theorem]{Definition}
\newtheorem{definition-theorem}[theorem]{Definition-Theorem}
\newtheorem{example}[theorem]{Example}

\newtheorem{innercustomthm}{Theorem}
\newenvironment{customthm}[1]
{\renewcommand\theinnercustomthm{#1}\innercustomthm}
{\endinnercustomthm}
\newtheorem{innercustomprop}{Proposition}
\newenvironment{customprop}[1]
{\renewcommand\theinnercustomprop{#1}\innercustomprop}
{\endinnercustomprop}

\theoremstyle{remark}
\newtheorem{remark}[theorem]{Remark}

\newtheorem{question}[theorem]{Question}

\numberwithin{equation}{section} 

\newcommand{\tr}{\operatorname{tr}}
\newcommand{\Id}{\operatorname{Id}}

\newcommand{\End}{\operatorname{End}}

\newcommand{\Ker}{\operatorname{Ker}}

\newcommand{\ad}{\operatorname{ad}}

\newcommand{\Aut}{\operatorname{Aut}}
\newcommand{\dbar}{\bar{\partial}}

\newcommand{\CC}{{\mathbb C}}

\newcommand{\RR}{{\mathbb R}}
\newcommand{\ZZ}{{\mathbb Z}}

\renewcommand{\(}{\left(}
\renewcommand{\)}{\right)}

\newcommand{\surj}{\to\kern-1.8ex\to}

\newcommand{\lra}[1]{\stackrel{#1}{\longrightarrow}}

\newcommand{\cA}{\mathcal{A}}

\newcommand{\cK}{\mathcal{K}}

\newcommand{\cF}{\mathcal{F}}
\newcommand{\cV}{\mathcal{V}}
\newcommand{\cW}{\mathcal{W}}

\newcommand{\cG}{\mathcal{G}}
\newcommand{\cL}{\mathcal{L}}
\newcommand{\cO}{\mathcal{O}}

\newcommand{\cS}{\mathcal{S}}

\newcommand{\Lie}{\operatorname{Lie}}

\newcommand{\cH}{\mathcal{H}} 

\def\om{\omega}
\def\Om{\Omega}

\def\Lie{\mathrm{Lie}}

\def\Im{\mathrm{Im}}
\def\Id{\mathrm{Id}}

\def\cA{\mathcal{A}}

\def\cF{\mathcal{F}}
\def\cG{\mathcal{G}}
\def\cH{\mathcal{H}}

\def\cU{\mathcal{U}}

\def\del{\partial}
\def\delb{\overline\partial}

\newcommand{\st}{\;|\;}

\newcommand{\R}{{\mathbb{R}}}
\newcommand{\C}{{\mathbb{C}}}

\newcommand{\la}{\langle}
\newcommand{\ra}{\rangle}

\newcommand{\SU}{\mathrm{SU}}
\newcommand{\U}{\mathrm{U}}
\newcommand{\GL}{\mathrm{GL}}
\newcommand{\SL}{\mathrm{SL}}

\renewcommand{\Im}{\mathrm{Im}}

\begin{document}

\maketitle

\begin{abstract}
The solution of the Calabi Conjecture by Yau implies that every K\"ahler Calabi-Yau manifold $X$ admits a metric with holonomy contained in $\SU(n)$, and that these metrics are parametrized by the positive cone in $H^{1,1}(X,\mathbb{R})$. In this work we give evidence of an extension of Yau's theorem to non-K\"ahler manifolds, where $X$ is replaced by
a compact complex manifold with vanishing first Chern class endowed with a holomorphic Courant algebroid $Q$ of Bott-Chern type. The equations that define our notion of \emph{best metric} correspond to a mild generalization of the Hull-Strominger system, whereas the role of $H^{1,1}(X,\mathbb{R})$ is played by an affine space of `Aeppli classes' naturally associated to $Q$ via Bott-Chern secondary characteristic classes.

\end{abstract}


\section{Introduction}
\label{sec:intro}

The Calabi Conjecture, made by E. Calabi in 1954 \cite{Calabi}, asserts that given a smooth volume form $\mu$ on a compact K\"ahler manifold $X$ there exists a K\"ahler metric $\omega$ on $X$ such that $\omega^{n}/n! = \mu$. From the work of Calabi and Yau, we know that such a metric exists and is unique on each positive class $[\omega] \in H^{1,1}(X,\RR)$ satisfying $[\omega]^n/n! = \int_X \mu$. In the particular case of a Calabi-Yau manifold, Yau's theorem \cite{Yau0} implies that $X$ admits a metric with holonomy contained in $\SU(n)$, and that these metrics are parametrized by the K\"ahler cone of $X$. The initial step of the proof is to fix the class $[\omega]$, whereby the problem is reduced to a PDE for a smooth function on $X$, namely, the complex Monge-Amp\`ere equation, amenable to the application of analytical techniques.

Following the recent advances in K\"ahler geometry there has been a renewed interest on extending Yau's theorem to the case of non-K\"ahler compact complex manifolds \cite{Yau2}. As a natural generalization of the Calabi problem, and motivated by string theory, Yau has proposed to study the Hull-Strominger system, which couples a Hermite-Einstein metric on a bundle with a balanced metric on a Calabi-Yau manifold, possibly of non-K\"ahler type. The construction of compact solutions for these equations was pioneered by Fu, Li and Yau \cite{FuYau,LiYau}, and has been an active topic of research in mathematics in the last ten years (see \cite{Fei,GF2,PPZ2} for recent reviews covering this topic).

In contrast to the existence problem for the Hull-Strominger system, the uniqueness problem for these equations has never been systematically addressed in the literature. Given a holomorphic bundle over a Calabi-Yau manifold, one can infer from Yau's theorem that the key to parametrize the solutions of the system should be some generalization of the K\"ahler cone, like, for instance, the balanced cone, as implicitly suggested by the approach in \cite{Phong}. Even for compact complex surfaces, where the existence of solutions is well understood thanks to the work of Strominger \cite{Strom}, the uniqueness problem is still open.

The main aim of this work is to take a step forward 
towards an answer to the uniqueness question. To do this, we propose to combine the Aeppli cohomology of the complex manifold with some unexplored geometric structures, known as holomorphic Courant algebroids. Despite their very rich properties, these objects have only received some attention in work by Gualtieri \cite{G2}, Bressler \cite{Bressler}, Gr\"utzmann-Sti\'enon \cite{GrSt} and Pym-Safronov \cite{Pym}. This paper builds on the idea that the existence and uniqueness problem for the Hull-Strominger system should be better understood as the problem of finding `the best metric' in a holomorphic Courant algebroid $Q$ with fixed `Aeppli class'. Compelling evidence for this proposal is given by the results described below.

Our approach to the Hull-Strominger system is gauge-theoretical in nature, and very close in spirit to the moment-map interpretation of the Calabi conjecture by Fine \cite{Fine}. 
Regarding a holomorphic Courant algebroid as the Atiyah Lie algebroid of a holomorphic principal bundle for the (complexified) string group \cite{grt2}, the present work brings up new tools--such as the \emph{dilaton functional}--that may help to elucidate an analogue of the Donaldson-Uhlembeck-Yau Theorem \cite{Don,UY} in the realm of \emph{higher gauge theory}.

As we had just completed the present work, two papers about the Fu-Yau equations appeared \cite{CHZ,PPZ}, which in particular imply a uniqueness result for the Hull-Strominger system in the special case of Goldstein-Prokushkin threefolds. While the main focus and methods are very different, it would be interesting to investigate a connection with the general approach to the uniqueness question proposed in this work.

\subsection{Summary of results}

Our first result is concerned with a mild generalization of the Hull-Strominger system. Let $G$ be a reductive complex Lie group. Let $P$ be a holomophic principal $G$-bundle over a compact complex manifold $X$ with $c_1(X) = 0$. We say that a triple $(\Psi,\omega,h)$, given by an $\SU(n)$-structure $(\Psi,\omega)$ on $X$ and a reduction $h$ of $P$ to a maximal compact subgroup $K \subset G$, is a solution of the \emph{twisted Hull-Strominger system} if
\begin{equation}\label{eq:twistedStromholintro}
\begin{split}
      F_h \wedge \omega^{n-1}  & = 0,\\
      d \Psi - \theta_\omega \wedge \Psi & = 0,\\
      d \theta_\omega & = 0,\\
      dd^c \omega - c(F_h \wedge F_h) & = 0.
    \end{split}
\end{equation}
Here $\theta_\omega := J d^*\omega$ is the Lee form of $\omega$ and $c$ is a bi-invariant symmetric bilinear form on the Lie algebra of $G$. The existence of solutions implies that the associated first Pontryagin class vanishes in Bott-Chern cohomology
$$
p_1(P) = 0 \in H^{2,2}_{BC}(X).
$$
Motivation for the twisted Hull-Strominger system \eqref{eq:twistedStromholintro} comes from generalized geometry \cite{grt} (see Section \ref{subsec:canonical}), supergravity \cite{grst}, and mirror symmetry. Actually, the study of this system of equations has very recently led to the first non-K\"ahler examples of $(0,2)$ mirror symmetry \cite{ALM}.

The Hull-Strominger system is recovered from \eqref{eq:twistedStromholintro} when the cohomology class of the Lee form $[\theta_\omega] \in H^1(X,\RR)$ vanishes (see Proposition \ref{prop:Strominger}). Key to our development is that, unlike the Hull-Strominger system, when $[\theta_\omega] \neq 0$ the equations \eqref{eq:twistedStromholintro} have solutions on compact complex manifolds that are not balanced. To see this, in Proposition \ref{prop:twistedStromexact} and Proposition \ref{prop:conformal} we classify the compact complex surfaces that admit a solution, and in Section \ref{subsec:ex3fold} we find examples on compact complex threefolds with infinitely many topological types and $[\theta_\omega] \neq 0$ (cf. \cite{FHZ}).

Our first objective is to understand the uniqueness question for \eqref{eq:twistedStromholintro} in the simplest possible non-trivial situation, namely, when $G = \{1\}$ and $X$ is a complex surface. The equations \eqref{eq:twistedStromholintro} admit non-K\"ahler solutions even in this case, such as quaternionic Hopf surfaces (see Proposition \ref{prop:twistedStromexact}). In the sequel, in the case $G= \{1\}$ we will refer to \eqref{eq:twistedStromholintro} as the \emph{twisted Calabi-Yau equation} (see Definition \ref{def:twistedStromexact}).  Given a solution $(\Psi,\omega)$ of the twisted Calabi-Yau equation, the hermitian form $\omega$ is pluriclosed, that is, $dd^c\omega = 0$, and it has an associated positive class in Aeppli cohomology
$$
[\omega] \in H^{1,1}_A(X,\RR).
$$
The next result illustrates the necessity of using Aeppli classes in the uniqueness problem.

\begin{customthm}{\ref{th:Hopf}}\label{th:Hopfintro}
If a compact complex surface admits a solution of the twisted Calabi-Yau equation \eqref{eq:twistedStromexact}, then it admits a unique solution $(\Psi,\omega)$ on each positive Aeppli class, up to rescaling of $\Psi$ by a unitary complex number.
\end{customthm}

In order to generalize Theorem \ref{th:Hopfintro} to the cases of higher dimensional manifolds or non-trivial $G$, we find two main obstacles. Firstly, the proof is based on a classification result for the twisted Calabi-Yau equation in Proposition \ref{prop:twistedStromexact}, combining some known facts about Einstein-Weyl manifolds \cite{GauIvanov} and quaternionic manifolds \cite{Kato} in real dimension $4$. Thus, even when $G= \{1\}$, our methods do not apply in three complex dimensions or higher (see Section \ref{subsec:ex3fold}). Secondly, when $G \neq \{1\}$ a solution of the (twisted) Hull-Strominger system does not define an Aeppli class on $X$. Therefore, a priori, it is unclear how to formulate and analogue of Theorem \ref{th:Hopfintro} in this case.

To overcome these difficulties, we change our perspective on the equations and propose to consider them in a higher version of the Atiyah algebroid of $P$. For this, inspired by 
Bott-Chern theory \cite{BottChern}, in Section \ref{sec:stringholCourant} we introduce a special class of holomorphic Courant
algebroids $Q$, which we call Bott-Chern algebroids (see Definition \ref{def:BCtype}), and study hermitian metrics on them (see Definition \ref{def:metricQ}). Upon a choice of holomophic bundle $P$ over $X$, the Bott-Chern algebroids are classified (see Proposition \ref{prop:BCclassification}) by the image of a linear map
\begin{equation}\label{eq:partialmapintro}
\partial \colon H^{1,1}_A(X,\mathbb{R}) \to H^1(\Omega^{\leqslant \bullet})/\Im \; \sigma_P
\end{equation}
induced by the $\partial$-operator on $(1,1)$-forms. Here, $H^1(\Omega^{\leqslant \bullet})$ denotes the space of isomorphism classes of exact holomorphic Courant algebroids  \cite{G2}, which we quotient by a natural isotropy action of the holomorphic gauge group of $P$.

In Proposition \ref{prop:stromholcour} we show that any solution of the twisted Hull-Strominger system \eqref{eq:twistedStromholintro} determines a Bott-Chern algebroid $Q$ endowed with a hermitian metric satisfying natural equations (see \eqref{eq:twistedStromholredux}). This leads us to define an affine space of (real) Aeppli classes $\Sigma_Q(\RR)$ on $Q$ in Section \ref{subsec:Aeppli}, modelled on the kernel of \eqref{eq:partialmapintro}
$$
\Sigma_Q(\RR) \cong \Ker \partial,
$$
and to associate an Aeppli class to any hermitian metric. 
A crucial ingredient in our construction is the cocycle property of the Bott-Chern
secondary characteristic class~\cite{BottChern} (see also \cite{BGS,Don})
$$
R(h_1,h_0) \in \Omega^{1,1}/\operatorname{Im}(\partial \oplus \dbar)
$$
for pairs of reductions $h_1,h_0$ on a bundle.

With this new framework at hand, in Section \ref{subsec:variational} we introduce new tools to address the existence and uniqueness problem for the Hull-Strominger system, which is recovered from \eqref{eq:twistedStromholintro} when $[\theta_\omega] = 0$. Analogously to the relation between Calabi-Yau metrics and solutions to the complex Monge-Amp\`ere equation $\omega^n/n! = \mu$ in K\"ahler geometry, we observe that the Hull-Strominger system is a particular example of a family of more flexible equations, formulated on a compact complex manifold $X$ endowed with a smooth volume form $\mu$ and a Bott-Chern algebroid $Q$.

Following a variational principle, we define a functional $M$ on the space of metrics on $Q$, which we call the \emph{dilaton functional}.
Upon restriction to the space $B_\sigma^+$ of metrics on $Q$ with fixed Aeppli class $\sigma \in \Sigma_Q(\RR)$, the Euler-Lagrange equations for $M$ define a system of partial differential equations \eqref{eq:stromabstractexp} which we call the \emph{Calabi System}. When $X$ admits a holomorphic volume form, 
the Calabi system is equivalent to the Hull-Strominger system (see Proposition \ref{prop:relationStr}). In a sense, the Calabi system 
gives an extension of the complex Monge-Amp\`ere equation to holomorphic Courant algebroids (see Proposition \ref{prop:extremaexactcase}), which we expect will provide new insights on the existence problem for the Hull-Strominger system.

The dilaton functional $M$ has some remarkable properties. It is bounded from below, and concave along suitable paths on the space of metrics on $Q$ with fixed Aeppli class (see Corollary \ref{cor:concave}). These special paths, given by \eqref{eq:concavepath}, are reminiscent of the geodesics in the space of metrics in a fixed K\"ahler class,
which play an important role in the constant scalar curvature problem in
K\"ahler geometry (see e.g. \cite{D6}). Using the concavity properties of the functional we prove the following (see also Remark \ref{rem:uniqueKahlercase}).

\begin{customprop}{\ref{prop:uniqueness}}\label{prop:uniqueness0}
If $(\omega_0,h_0)$ and $(\omega_1,h_1)$ are two solutions of the Calabi system \eqref{eq:stromabstractredux} with Aeppli class $\sigma \in \Sigma_Q(\RR)$ that can be joined by a smooth solution $(\omega_t,h_t)$ of \eqref{eq:concavepath}, then $\omega_1 = k \omega_0$ for some constant $k$, and $h_1$ is related to $h_0$ by an element in the holomorphic gauge group $\cG_P$ of $P$. Furthermore, when $d\omega_0\neq 0$, we must have $k=1$ and, consequently, $(\omega_0,h_0)$ and $(\omega_1,h_1)$ are related by an automorphism of $Q$.
\end{customprop}

It is therefore natural to expect that the dilaton functional and the Dirichlet problem for the PDE \eqref{eq:concavepath} are important gadgets in the theory for the Hull-Strominger system. Expanding upon the method of Proposition \ref{prop:uniqueness0}, using the dilaton functional and the special features of the path \eqref{eq:concavepath} in complex dimension two (see Remark \ref{rem:concavepathsurface}) we prove the following.

\begin{customthm}{\ref{th:bounded}}\label{th:boundedintro}
Let $X$ be a compact complex surface endowed with an exact Bott-Chern algebroid $Q$.
There is at most one solution of the Calabi system in a positive Aeppli class $\sigma \in \Sigma_Q(\RR)$ on $Q$.
Furthermore, if such a solution exists, the dilaton functional $M$ is bounded from above on $B_\sigma^+$.
\end{customthm}

By Definition \ref{def:stringholCour}, the hypothesis that $Q$ is exact is equivalent to the condition $G = \{1\}$ for the principal bundle $P$. In this situation, the functional $M$ can be formulated in an arbitrary \emph{strong K\"ahler with torsion} manifold and it admits a critical point if and only if the manifold is K\"ahler. More precisely, its critical points are K\"ahler solutions of the Calabi problem (see Proposition \ref{prop:extremaexactcase}). We expect that the previous result holds for the Calabi system on arbitrary Bott-Chern algebroids endowed with a hermitian metric. It is interesting to notice that, in the setup of holomorphic Courant algebroids, Theorem \ref{th:Hopfintro} can be stated as follows: if $Q$ is exact, there is at most one solution of the twisted Hull-Strominger system on each positive Aeppli class. We also expect that this uniqueness result holds for the equations \eqref{eq:twistedStromholintro} on arbitrary Bott-Chern algebroids.

Section \ref{sec:linear} is devoted to study the linear theory for the twisted Hull-Strominger system and the Calabi system on a Bott-Chern algebroid $Q$, showing that the linearization of the equations restricted to an Aeppli class \eqref{eq:twistedStromholredexp} induces a Fredholm operator. For the case of the Calabi system we prove that the corresponding operator has index zero and provide a Fredholm alternative: either it has a non-trivial finite-dimensional kernel, or it is invertible. As an application, we study the existence of solutions under deformations of $(X,Q)$. To state our next result, we use that there is an inclusion
\begin{equation}\label{eq:LieinKerintro}
\Ker d \subset \Ker \cL,
\end{equation}
where $\cL$ (see \eqref{eq:FredholmLsigma0}) denotes the linearization of the Calabi system, and $d \colon \Omega^1 \to \Omega^2$ is the exterior differential acting on forms. 

\begin{customthm}{\ref{theo:stability under deformations}}\label{theo:stability under deformationsintro}
Assume that $(X,Q)$ admits a solution of the Calabi system with Aeppli class $\sigma$, such that $\Ker d = \Ker \cL$. Let $(X_t,Q_t)_{t\in B}$ be a Bott-Chern deformation of $(X,Q)$ such that $h^{1,1}_A(X_t)$ and $h^{2,2}_{BC}(X_t)$ are constant. Then, for any $t$ small enough, $(X_t,Q_t)$ admits a differentiable family of solutions, parametrized by an open set in $\Sigma_{Q_t}(\RR)$.
\end{customthm}

In particular, when the deformation is trivial, we prove that nearby solutions are parametrized by a small neighbourhood $U \subset \Sigma_Q(\RR)$ of $\sigma$ (see Remark \ref{rem:uniquenessdef}). This gives another evidence for the proposed extension of Yau's theorem for pairs $(X,Q)$ using our notion of Aeppli classes. It is interesting to notice that, if we fix $(X,P)$ and let $Q$ and $\sigma$ vary, the expected overall dimension of the space of nearby solutions is
\begin{equation}\label{eq:predictionintro}
\dim \; \Im \; \partial + \dim \ker \partial = \dim H^{1,1}_A(X,\RR),
\end{equation}
where $\partial$ is as in \eqref{eq:partialmapintro}. A precursor of this observation can be found in \cite{TsengYau0}. The first contribution in \eqref{eq:predictionintro} has to be understood as the number of deformations of $Q$, while the latter corresponds to the dimension of $\Sigma_Q(\RR)$. Remarkably, \eqref{eq:predictionintro} matches the number of solutions of the Hull-Strominger system in the nilmanifold $\mathfrak{h}_3$ found in \cite{FIVU}, regarded as solutions of the twisted Hull-Strominger system (see Remark \ref{rem:prediction}).
If $X$ is a $\partial\dbar$-manifold, $h^{1,1}_A(X_t)$ and $h^{2,2}_{BC}(X_t)$ are constant and any small complex deformation of $(X,P)$ induces a unique Bott-Chern deformation of $(X,Q)$ (see Lemma \ref{lem:BCdeformationddbar}) admitting a real family of solutions of dimension $h^{1,1}_A(X)$.

To conclude, in Section \ref{subsec:deformstemb} we study sufficient conditions for the identity $\Ker d =  \Ker \cL$. We use this analysis in Corollary \ref{cor:stability under deformations} to provide a large class of solutions of the Calabi system on K\"ahler manifolds via deformation. Even though we have not been able to prove it in general, we expect that the orthogonal complement of $\Ker d \subset \Ker \cL$
is related to a suitable Lie subalgebra of the infinitesimal automorphisms of $Q$ (see Remark \ref{rem:LieinKer}). A confirmation of this expectation would reduce the problem of deformation of solutions of the Calabi system to algebraic geometry.


\medskip

\textbf{Acknowledgments:} The authors would like to thank L. \'Alvarez-C\'onsul, D. Angella, V. Apostolov, N. Hitchin, A. Fino, A. Moroianu, L. Ornea, B. Pym, L. Ugarte, L. Vezzoni, and V. Vuletescu for helpful discussions. We thank L. Vezzoni for pointing out reference \cite{Bismutflat}. We are grateful to the anonymous referees for important corrections and suggestions. Part of this work was undertaken during visits of MGF to the Yau Mathematical Sciences Center, LMBA and IMPA, of RR to ICMAT, of CS to ICMAT and IFT, and of CT to IMPA, CIRGET, ICMAT and Weizmann Institute. We would like to thank these very welcoming institutions for providing a nice and stimulating working environment.

\section{The twisted Hull-Strominger system}

\subsection{Definition and basic properties}\label{subsec:twistedStr}

Let $X$ be a compact complex manifold of dimension $n$ with vanishing first Chern class
$$
c_1(X) = 0 \in H^2(X,\mathbb{Z}).
$$
Given a hermitian form $\omega$ on $X$ we denote by $g = \omega(\cdot,J\cdot)$ the induced riemannian metric and by
\begin{equation}\label{eq:Leeformdef0}
\theta_\omega = J d^*\omega
\end{equation}
the associated Lee form, where $J$ is the integrable almost complex structure on $X$. Alternatively, the Lee form is the unique real one-form $\theta_\omega$ on $X$ satisfying
\begin{equation}\label{eq:Leeformdef}
d \omega^{n-1} = \theta_\omega \wedge \omega^{n-1}.
\end{equation}
Associated to $\omega$ there is a contraction operator $\Lambda_\omega$, given by the adjoint of the Lefschetz operator $\alpha \mapsto \alpha \wedge \omega$ acting on differential forms. In particular, for $\alpha \in \Omega^2$, one has
$$
(\Lambda_\omega \alpha) \omega^{n} = n \alpha \wedge \omega^{n-1}.
$$
Let $\Psi$ be a smooth section of the canonical bundle $K_X$ and consider the smooth function $\|\Psi\|_\omega$ on $X$ given by the point-wise norm of $\Psi$, defined by
\begin{equation*}\label{eq:norm}
\|\Psi\|_\omega^2\frac{\omega^n}{n!} = (-1)^{\frac{n(n-1)}{2}}i^n \Psi \wedge \overline{\Psi}.
\end{equation*}
An $\SU(n)$-structure on $X$ is given by a pair $(\Psi,\omega)$ as before, such that
\begin{equation}\label{eq:normalizationnorm}
\|\Psi\|_\omega = 1.
\end{equation}

Let $G$ be a reductive complex Lie group with Lie algebra $\mathfrak{g}$. Let $p \colon P \to X$ be a holomorphic principal $G$-bundle. Given a maximal compact subgroup $K \subset G$, a reduction $h \in \Omega^0(P/K)$ of the structure group of $P$ determines a Chern connection $\theta^h$, with curvature $F_h := F_{\theta^h}$ satisfying
$$
F_h^{0,2} = 0.
$$
We fix a non-degenerate bi-invariant symmetric bilinear form
\begin{equation}\label{eq:ccomplex}
c \colon \mathfrak{g} \otimes \mathfrak{g} \to \mathbb{C}
\end{equation}
such that 
the induced bilinear form on the Lie algebra $\mathfrak{k}$ of $K$ is real valued, that is,
\begin{equation}\label{eq:creal-notused}
c(\mathfrak{k} \otimes \mathfrak{k}) \subset \mathbb{R}
\end{equation}
(see \eqref{eq:cphysics} below for a concrete example).
By Chern-Weil theory, the form $c$ defines a Pontryagin class in the real Bott-Chern cohomology of $X$ given by
$$
p_1(P) = [c(F_h \wedge F_h)] \in H^{2,2}_{BC}(X,\RR)
$$
for any choice of reduction $h$. Here, the Bott-Chern cohomology groups of $X$ are defined by
\begin{equation}\label{eq:BC-cohomology}
H^{p,q}_{BC}(X) = \frac{\Ker(d\colon \Omega^{p,q} \to \Omega^{p+1,q}\oplus \Omega^{p,q+1})}{\Im(dd^c \colon \Omega^{p-1,q-1} \to \Omega^{p,q})}.
\end{equation}
Note that $H^{p,p}_{BC}(X)$ has a natural real structure. We will assume that
\begin{equation}\label{eq:pic0}
p_1(P) = 0  \in H^{2,2}_{BC}(X,\RR).
\end{equation}

\begin{definition}\label{def:twistedStrom}
We say that a triple $(\Psi,\omega,h)$, given by an $\SU(n)$-structure $(\Psi,\omega)$ on $X$ and a reduction $h$ of the structure group of $P$ to $K$, is a solution of the \emph{twisted Hull-Strominger system} if
\begin{equation}\label{eq:twistedStromhol}
\begin{split}
      F_h \wedge \omega^{n-1}  & = 0,\\
      d \Psi - \theta_\omega \wedge \Psi & = 0,\\
      d \theta_\omega & = 0,\\
      dd^c \omega - c(F_h \wedge F_h) & = 0.
    \end{split}
\end{equation}
\end{definition}

Motivation for this definition comes from generalized geometry \cite{Hit1}. A solution of the last equation in \eqref{eq:twistedStromhol} determines a smooth (string) Courant algebroid $E$ and, in this setup, the twisted Hull-Strominger system corresponds to a special class of solutions of the \emph{Killing spinor equations} in \cite{GF3,grt} (see \cite{grst}). 
The main focus of this paper is on the interplay between the existence and uniqueness problem for the twisted Hull-Strominger system \eqref{eq:twistedStromhol} and a holomorphic version of the \emph{string algebroid} $E$ 
introduced in \cite{grt2} (see Definition \ref{def:stringholCour}).

We start our study of \eqref{eq:twistedStromhol} discussing some obstructions to the existence of solutions. 
By the Buchdahl-Li-Yau Theorem \cite{Buchdahl,LiYauHYM} for the Hermite-Einstein equation (corresponding to the first equation in \eqref{eq:twistedStromhol}), if $(X,P)$ admits a solution then the holomorphic bundle $P$ must be polystable with respect to the unique Gauduchon metric $\tilde \omega$ in the conformal class of $\omega$. 
In addition, the third equation in \eqref{eq:twistedStromhol} combined with \eqref{eq:Leeformdef} implies that $X$ is a locally conformally balanced manifold. Alternatively, the existence of solutions of \eqref{eq:twistedStromhol} implies that $X$ must admit a Gauduchon metric with harmonic Lee form. The associated cohomology class
$$
[\theta_\omega] \in H^1(X,\RR)
$$
will play an important role in our study.

In the following result we analyze the conditions 
\begin{equation}\label{eq:Leeformeq}
d \Psi - \theta_\omega \wedge \Psi = 0, \qquad d \theta_\omega = 0.
\end{equation}
in equation \eqref{eq:twistedStromhol} in terms of the holonomy of the Bismut connection
$$
\nabla^+ = \nabla^g - \frac{1}{2}g^{-1}d^c\omega
$$
of the hermitian form $\omega$. It will be clear from the proof that the normalization \eqref{eq:normalizationnorm} is crucial here.


\begin{lemma}\label{lemma:holSUn}
If an $\SU(n)$-structure $(\Psi,\omega)$ on $X$ satisfies \eqref{eq:Leeformeq}, then the holonomy of the Bismut connection $\nabla^+$ is contained in $\SU(n)$.
\end{lemma}
\begin{proof}
By the holonomy principle 
it is enough to prove that $\nabla^+ \Psi = 0$. Using that $\theta_\omega$ is closed, given any point $x \in X$ there exists a smooth local function $\phi$ such that
$$
\theta_\omega = d\phi
$$
around $x$. Then, by the first equation in \eqref{eq:Leeformeq} $\Omega = e^{-\phi}\Psi$ is closed, and hence it provides a holomorphic trivialization of $K_X$ around $x$. In this trivialization, the Chern connection $\nabla^C$ on $K_X$ induced by $\omega$ is given by (see \eqref{eq:normalizationnorm})
$$
\nabla^C = d + 2\partial \log \|\Omega\|_\omega = d - 2 \partial \phi.
$$
The proof follows using Gauduchon's formula \cite[Eq. (2.7.6)]{Gau} relating $\nabla^C$ with the connection induced by $\nabla^+$ on the canonical bundle
\begin{equation*}\label{eq:nablaC}
  \nabla^C \Psi = \nabla^+ \Psi + i d^*\omega \otimes \Psi,
\end{equation*}
which implies $\nabla^+ \Psi = 0$ around $x$.
\end{proof}


To finish this section, we introduce a definition corresponding to the twisted Hull-Strominger system \eqref{eq:twistedStromhol} in the case $G = \{1\}$. This toy model situation will play an important role in Section \ref{subsec:1stexamples}.

\begin{definition}\label{def:twistedStromexact}
We say that an $\SU(n)$-structure $(\Psi,\omega)$ on $X$ is a solution of the \emph{twisted Calabi-Yau equation} if
\begin{equation}\label{eq:twistedStromexact}
\begin{split}
      d \Psi - \theta_\omega \wedge \Psi& = 0,\\
      d\theta_\omega & = 0,\\
      dd^c \omega & = 0.
    \end{split}
\end{equation}
\end{definition}

\subsection{The case $[\theta_\omega] = 0$ and the uniqueness question}\label{subsec:sigmaX0}

An important motivation for the study of \eqref{eq:twistedStromhol} is provided by the Hull-Strominger system of partial differential equations \cite{HullTurin,Strom} (see \cite{Fei,GF2,PPZ2} for recent reviews covering this topic). As we will see next, 
we recover the Hull-Strominger system from \eqref{eq:twistedStromhol} when the cohomology class of the Lee form $[\theta_\omega] \in H^1(X,\RR)$ vanishes. 

Let $(X,\Omega)$ be a compact Calabi-Yau manifold of dimension $n$, that is, a compact complex manifold $X$ endowed with a 
holomorphic volume form $\Omega$. Let $P$ be a holomorphic principal $G$-bundle on $X$ satisfying \eqref{eq:pic0}. As in the previous section, $K \subset G$ denotes a maximal compact subgroup.

\begin{definition}\label{def:Strom}
We say that $(\omega,h)$, given by a hermitian form $\omega$ on $(X,\Omega)$ and a reduction $h$ of $P$ to $K$, is a solution of the \emph{Hull-Strominger system} if
\begin{equation}\label{eq:Stromhol}
\begin{split}
      F_h \wedge \omega^{n-1}  & = 0,\\
      d^* \omega - d^c\log \|\Omega\|_\omega & = 0,\\
      dd^c \omega - c(F_h \wedge F_h) & = 0.
    \end{split}
\end{equation}
\end{definition}

In the original formulation of the Hull-Strominger system in physics \cite{HullTurin,Strom}, the Calabi-Yau manifold $(X,\Omega)$ has complex dimension $3$, it is endowed with a rank $r$ holomorphic vector bundle $V$ with trivial determinant, and $P$ is the fibred product of the bundles of holomorphic frames of $TX$ and $V$, so that
$$
G = \SL(3,\CC) \times \SL(r,\CC).
$$
The bi-invariant symmetric bilinear form $c$ in $\mathfrak{g}$ is then given by
\begin{equation}\label{eq:cphysics}
c = \alpha \tr_{\mathfrak{sl}(3,\CC)} - \alpha \tr_{\mathfrak{sl}(r,\CC)},
\end{equation}
for a positive real constant $\alpha > 0$. In this setup, a difference between Definition \ref{def:Strom} and other definitions in the literature (see e.g. \cite{LiYau}) is that we require the connection $\nabla$ in the tangent bundle (so that $\theta^h = \nabla \times A$ in \eqref{eq:Stromhol}) to solve the Hermite-Einstein equation (see \cite{GF2} for a lengthy discussion about this condition). 
Nonetheless, we expect that most of our methods extend to the situation considered in \cite{LiYau}.


For the next result, we do not assume that $X$ is Calabi-Yau.

\begin{proposition}\label{prop:Strominger}
Let $X$ be a compact complex manifold with $c_1(X) = 0$. Suppose that $(X,P)$ admits a solution $(\Psi,\omega,h)$ of \eqref{eq:twistedStromhol} and that $[\theta_\omega] = 0$. Then, if $\theta_\omega = d\phi$, we have that $\Omega = e^{-\phi} \Psi$ is a holomorphic volume form on $X$ and $(\omega,h)$ is a solution of the Hull-Strominger system \eqref{eq:Stromhol}. Conversely, if $\Omega$ is a holomorphic volume form on $X$ and $(\omega,h)$ is a solution of \eqref{eq:Stromhol} on $(X,P)$, then $(\|\Omega\|_\omega^{-1}\Omega,\omega,h)$ is a solution of \eqref{eq:twistedStromhol} with $[\theta_\omega] = 0$.
\end{proposition}

\begin{proof}
For the `if part', note first that $\Omega = e^{-\phi} \Psi$ is closed, and hence defines a holomorphic volume form on $X$:
$$
d \Omega = - d\phi \wedge \Omega + e^{-\phi} d\Psi = (-d\phi + \theta_\omega)\wedge \Omega = 0.
$$
Using now equation \eqref{eq:normalizationnorm}, we have
$$
1 = \|\Psi\|_\omega = e^{\phi}\|\Omega\|_\omega,
$$
and therefore $d^*\omega = - J \theta_\omega = d^c\log \|\Omega\|_\omega$ as required. The converse follows by taking the exterior differential of $\|\Omega\|_\omega^{-1}\Omega$, combined with the second equation in \eqref{eq:Stromhol}, which is equivalent to $\theta_\omega = - d \log \|\Omega\|_\omega$.
\end{proof}

By the previous result, the vanishing of the class $[\theta_\omega] \in H^1(X, \RR)$ implies that the complex manifold $X$ has holomorphically trivial canonical bundle $K_X \cong \mathcal{O}_X$ and that it is balanced \cite{Michel}. Recall that $X$ is called balanced if there exists a hermitian form $\tilde \omega$ on $X$ such that $d \tilde \omega^{n-1} = 0$. The associated class in real Bott-Chern cohomology \cite{FuXiao}
\begin{equation}\label{eq:balancedclass}
\mathfrak{b} = [\tilde \omega^{n-1}] \in H^{n-1,n-1}_{BC}(X,\RR)
\end{equation}
is called the \emph{balanced class} of $\tilde \omega$. For a solution of the Hull-Strominger system the balanced hermitian form is
$$
\tilde \omega = \|\Omega\|_\omega^{\frac{1}{n-1}} \omega,
$$
and the Buchdahl-Li-Yau Theorem \cite{Buchdahl,LiYauHYM} for the Hermite-Einstein equation states in this case that the bundle $P$ must be polystable with respect to the balanced class $\mathfrak{b}$ \cite{lt}.

We can use Proposition \ref{prop:Strominger} to find some first interesting families of solutions of \eqref{eq:twistedStromhol} by application of existence results for the Hull-Strominger system. When $G = \{1\}$, the system \eqref{eq:Stromhol} reduces to
\begin{equation}\label{eq:Stromexact}
\begin{split}
      d^* \omega - d^c \log \|\Omega\|_\omega & = 0,\\
      dd^c \omega & = 0,
    \end{split}
\end{equation}
which is equivalent to the metric $g = \omega(\cdot,J\cdot)$ being Calabi-Yau, that is, with holonomy of the Levi-Civita connection contained in $\SU(n)$ \cite[Cor. 4.7]{IvanovPapadopoulos} (see also \cite{grt}). Thus, in this case $X$ must be K\"ahler and, by Yau's solution of the Calabi Conjecture \cite{Yau0}, the solutions of \eqref{eq:Stromexact} are parametrized by the cone of K\"ahler classes in $H^{1,1}(X,\RR)$.

When $G \neq \{1\}$, following \cite{GF2} the solutions have a very different flavour depending on whether the complex dimension $n$ of the Calabi-Yau is one, two, or higher. For $n =1$, 
any hermitian metric is K\"ahler, and the solutions of \eqref{eq:Stromhol} are parametrized by a K\"ahler class on an elliptic curve $X$---corresponding to a flat metric on $X$---and a polystable bundle over $X$. For $n =2$, the solutions of \eqref{eq:Stromhol} are given, up to conformal rescaling of the hermitian form $\omega$, by a Calabi-Yau metric $\tilde g$ on $X$ and a holomorphic bundle over $X$ satisfying \eqref{eq:pic0}, 
which is polystable with respect to the K\"ahler class of $\tilde g$ (see Section \ref{subsec:1stexamples}). In complex dimension three or higher the theory goes beyond the realm of K\"ahler geometry and the examples become more scarce. 
Despite the fact that there are recent new constructions of solutions of the Hull-Strominger system, to the present day the existence problem in the critical dimension $n = 3$ is widely open. We refer to \cite{Fei,GF2,GF4,PPZ2} for a detailed discussion about this case.

\begin{remark}\label{rem:stembedding}
A Calabi-Yau metric $g$ can be regarded as a solution of \eqref{eq:Stromhol} with $G = \SL(n,\CC) \times \SL(n,\CC)$ by considering $\theta^h$ to be the product of two copies of the Chern connection of $g$, on the fibre product of two copies of the bundle of holomorphic frames of $X$. These examples, known as \emph{standard embedding solutions} in the literature, will be studied in more detail in Section \ref{subsec:deformstemb}.
\end{remark}




In contrast to the existence problem for the Hull-Strominger system, which has become recently an active topic of research, the uniqueness problem for these equations has never been explored. 
In the light of Yau's solution of the Calabi Conjecture \cite{Yau0}, it is natural to ask the following question for the more general twisted Hull-Strominger system \eqref{eq:twistedStromhol}.

\begin{question}\label{question}
If $(X,P)$ admits a solution of the twisted Hull-Strominger system \eqref{eq:twistedStromhol}, which cohomological quantities parametrize the possible solutions? 
\end{question}

A complete answer in the case $G = \{1\}$ and $[\theta_\omega] = 0$ is given by Yau's Theorem \cite{Yau0}, which states that any K\"ahler class on a Calabi-Yau manifold admits a unique K\"ahler Ricci-flat metric. When $G \neq \{1\}$, a classical approach to Question \ref{question} in the case $[\theta_\omega] = 0$, that is, for the Hull-Strominger system \eqref{eq:Stromhol}, is to consider the balanced class of the solution \eqref{eq:balancedclass} as the relevant cohomological quantity.  
Thus, the expected answer in this approach would be the elements of the balanced cone of $X$ \cite{FuXiao}. This is, for instance, the path followed in \cite{Phong} using geometric flows.
In the present paper we propose a radically different answer to this question for the more general equations \eqref{eq:twistedStromhol}, combining the Aeppli cohomology of $X$ with holomorphic Courant algebroids.

\subsection{Complex surfaces and Aeppli classes}\label{subsec:1stexamples}

In this section we give evidence of an extension of Yau's Theorem for Calabi-Yau metrics \cite{Yau0} to the twisted Hull-Strominger system \eqref{eq:twistedStromhol} in the case of complex surfaces, where the role of K\"ahler classes is played by Aeppli cohomology classes. 
Even though this case is rather special, it provides the starting point of our approach to Question \ref{question} in higher dimensions. 

Let $X$ be a compact complex surface with $c_1(X) = 0$. We consider first the case of the twisted Calabi-Yau equation \eqref{eq:twistedStromexact} (corresponding to \eqref{eq:twistedStromhol} with $G = \{1\}$).
Our first goal is to provide a classification of the solutions of \eqref{eq:twistedStromexact}, combining some known facts about Einstein-Weyl manifolds \cite{GauIvanov} and quaternionic manifolds \cite{Kato} in real dimension $4$. We start by showing that any solution of \eqref{eq:twistedStromexact} is Einstein-Weyl. Recall that a Weyl structure with respect to a conformal class $[g]$ on a smooth manifold $M$ is defined as a torsion-free connection on $TM$, preserving $[g]$. A Weyl structure is said to be Einstein if the associated Ricci tensor is a multiple of any metric in $[g]$.

\begin{lemma}\label{lem:EinsteinWeyl}
If $(\Psi,\omega)$ is a solution of the twisted Calabi-Yau equation \eqref{eq:twistedStromexact} on a complex surface then $g = \omega(\cdot,J\cdot)$ is Einstein-Weyl.
\end{lemma}
\begin{proof}
Consider the universal cover $\tilde X$ of $X$ and the pull-back solution $(\tilde \Psi,\tilde \omega)$ of \eqref{eq:twistedStromexact}. On $\tilde X$ we have that $\theta_\omega = d\phi$ for a globally defined function and therefore, by the proof of Lemma \ref{lemma:holSUn} and \eqref{eq:Leeformdef}, it follows that $e^{-\phi}\tilde g$ has holonomy contained in $\SU(2)$ (as it is K\"ahler, and the Levi-Civita connection preserves $e^{-\phi}\Psi$). In particular, $e^{-\phi}\tilde g$ is Ricci-flat, and therefore $g$ is Einstein-Weyl.
\end{proof}

\begin{remark}\label{remark:lck}
The proof uses crucially that a solution of \eqref{eq:twistedStromexact} on a complex surface is locally conformally K\"ahler, that is,
\begin{equation*}\label{eq:lcK}
d\theta_\omega = 0, \qquad d \omega =  \theta_\omega \wedge \omega.
\end{equation*}
In higher dimensions, the same argument shows that a solution of 
\eqref{eq:twistedStromexact} which is locally conformally K\"ahler, is necessarily Einstein-Weyl. This class of manifolds provides an interesting class of candidates for solutions of \eqref{eq:twistedStromexact}.
\end{remark}

We are interested in compact Einstein-Weyl four-manifolds $(M,g)$ which admit a compatible (integrable) complex structure $J$. These metrics are classified in \cite[Thm. 3]{GauIvanov}: either $(M,g,J)$ is a flat torus or a $K3$ surface with a K\"ahler Ricci-flat metric, or $(M,J)$ is a Hopf surface and $g$ is locally isometric up to homothety to $S^3 \times \RR$ (with the standard product metric). Note that in the last case the metric is Vaisman \cite{GauIvanov}, that is, locally conformally K\"ahler and with Lee form parallel with respect to the Levi-Civita connection.

To state our classification result for solutions of the twisted Calabi-Yau equation \eqref{eq:twistedStromexact}, we recall some background. A Hopf surface is a compact complex surface whose universal covering is $\mathbb{C}^{2}\backslash \left\{ 0\right\}$. The fundamental group $\Gamma$ of a Hopf surface $X$ which admits a Vaisman metric is of the form \cite{Belgun}
\begin{equation}\label{eq:pi1}
\Gamma = \langle \gamma \rangle \ltimes H,
\end{equation}
where $H$ is a finite subgroup of $\U(2)$ and $\langle \gamma \rangle$ is an infinite cyclic group generated by a holomorphic contraction which, in suitable coordinates in the universal covering, takes the form
\begin{equation}\label{eq:cyclicgroup}
\gamma(z_1,z_2) = (\alpha z_1,\beta z_2),
\end{equation}
where $\alpha, \beta$ are complex numbers such that $1 < |\beta| \leq |\alpha|$. By the classification in \cite{Kato}, a quaternionic Hopf surface is a Hopf surface $(\mathbb{C}^{2}\backslash \left\{ 0\right\})/\Gamma$ with fundamental group $\Gamma = \langle \gamma \rangle \ltimes H$ conjugated to a subgroup of $\SU(2) \times \RR^*$. Equivalently, in suitable coordinates in $\mathbb{C}^{2}\backslash \left\{ 0\right\}$, $H \subset \SU(2)$ and
\begin{equation}\label{eq:alphabeta}
1 < |\alpha| = |\beta|, \qquad \textrm{and} \qquad \alpha \beta \in \RR.
\end{equation}

\begin{proposition}\label{prop:twistedStromexact}
Let $X$ be a compact complex surface. If $(\Psi,\omega)$ is a solution of the twisted Calabi-Yau equation \eqref{eq:twistedStromexact} on $X$ with $g = \omega(\cdot,J\cdot)$, then one of the following holds:
\begin{enumerate}[i)]
\item $(X,g)$ is a flat torus or a $K3$ surface with a K\"ahler Ricci-flat metric. In this case $[\theta_\omega] = 0$ and any such $(X,g)$ provides a solution.

\item $X$ is a quaternionic Hopf surface and there exist coordinates $(z_1,z_2)$ in the universal covering $\mathbb{C}^{2}\backslash \left\{ 0\right\}$ such that the pull-backs of $\omega$ and $\Psi$ are
\begin{equation}\label{eq:seedgeneral}
\tilde \omega = a i\frac{dz_{1}\wedge d\bar{z}_{1} + dz_{2}\wedge d\bar{z}_{2}}{|z|^2}, \qquad \tilde \Psi = \lambda \frac{dz_{1}\wedge dz_{2}}{|z|^2},
\end{equation}
respectively, for $|z|^2 = |z_1|^2 + |z_2|^2$ and suitable $a \in \RR_{> 0}$ and $\lambda \in \CC^*$. In this case $[\theta_\omega] \neq 0$ and any $(\tilde \Psi,\tilde \omega)$ as in \eqref{eq:seedgeneral} induces a solution.
\end{enumerate}
\end{proposition}

\begin{proof}
If $[\theta_\omega] = 0$ then $\omega$ is K\"ahler (see Section \ref{subsec:sigmaX0}), and therefore $g = \omega(\cdot,J\cdot)$ is a Calabi-Yau metric by Lemma \ref{lemma:holSUn}. Thus, $X$ must be a torus or a $K3$ surface. Conversely, any K\"ahler Ricci-flat metric on a torus or a $K3$ surface provides a solution of \eqref{eq:twistedStromexact}, and in this case $[\theta_\omega] = 0$.

By the classification in \cite[Thm. 3]{GauIvanov} and Lemma \ref{lem:EinsteinWeyl}, it remains to understand the case when $X$ is a Hopf surface. The second Betti number of a Hopf surface vanishes, and therefore $[\theta_\omega] \neq 0$, as $X$ does not admit any K\"ahler metric. Since $g$ is locally isometric to $S^3 \times \RR$, one requires \cite[Lemme 11]{GauHopf}
\begin{equation*}\label{eq:normalphabeta}
1 < |\beta| = |\alpha|,
\end{equation*}
and in this case $\tilde \omega$ is homothetic to the hermitian form
\begin{equation*}\label{eq:seedomega}
\omega_{0} = i\frac{dz^{1}\wedge d\bar{z}^{1} + dz^{2}\wedge d\bar{z}^{2}}{|z|^2},
\end{equation*}
with Lee form
$$
\theta_{0} = d \log |z|^2 = \frac{z^{1}d\bar{z}^{1} + z^{2} d\bar{z}^{2} + c.c.}{|z|^2},
$$
where $c.c.$ stands for complex conjugate.
Define
\begin{equation*}\label{eq:seedPsi0}
\Psi_{0} = \frac{dz^{1}\wedge dz^{2}}{|z|^2}.
\end{equation*}
It is straightforward to check that $(\Psi_0,\omega_0)$ provides a solution of \eqref{eq:twistedStromexact} on $\mathbb{C}^{2}\backslash \left\{ 0\right\}$ and thus, by Lemma \ref{lemma:holSUn}, $\tilde \Psi$ and $\Psi_0$ are both parallel with respect to the Bismut connection of $\tilde \omega$. This implies that $(\tilde \Psi,\tilde \omega)$ is of the form \eqref{eq:seedgeneral} and, since $g$ is Vaisman, the fundamental group $\Gamma$ of $X$ is as in \eqref{eq:pi1}. Using that $\Gamma$ preserves $\eqref{eq:seedgeneral}$ we necessarily have that \eqref{eq:alphabeta} holds and furthermore $H \subset \SU(2)$. Thus, we conclude that $X$ is quaternionic. For the converse, we simply note that $(\Psi_0,\omega_0)$ is preserved by $\Gamma$ when $X$ is a quaternionic Hopf surface.
\end{proof}


\begin{remark}\label{rem:quaternionic}
A quaternionic $4$-manifold is a smooth manifold of real dimension $4$ with an atlas formed by  quaternionic maps with respect to the standard quaternionic structure on $\mathbb{H} \cong \RR^4$. Compact quaternionic $4$-manifolds were classified by Kato \cite{Kato}, and they are given by complex analytic tori or quaternionic Hopf surfaces. The possible finite subgroups $H \subset \SU(2)$ which appear in the fundamental group (see \eqref{eq:pi1}) of a quaternionic Hopf surface are listed in \cite[Prop. 8]{Kato} (see also \cite{Katoerratum}). In particular, any quaternionic $4$-manifold is hypercomplex.
\end{remark}

\begin{remark}\label{rem:Hopfmoduli}
Hopf surfaces have a characteristic `jumping behaviour' under deformations of the complex structure, which rules out even a non-separable moduli space for the class of all Hopf surfaces \cite[Sec. 6]{Dabrowski}. A primary Hopf surface--that is, with fundamental group $\Gamma \cong \ZZ$ given by \eqref{eq:cyclicgroup}--which admits a solution of \eqref{eq:twistedStromexact} has necessarily $deg(X) = 0$ \cite{Dabrowski}, and hence the existence of solutions obstructs the possible jumps. 
This remarkable property of \eqref{eq:twistedStromexact} is a characteristic feature of partial differential equations with a moment map interpretation (see e.g. \cite{lt}), and it would be interesting to see if this system allows for such an interpretation (for the Hull-Strominger system, see \cite{grt3}).
\end{remark}

\begin{remark}\label{rem:sigmaprimaryI}
When $X$ is a primary Hopf surface of class I and $\alpha \beta = | \alpha \beta|$ is satisfied, the Vaisman metric $\omega_{\alpha,\beta}$ constructed in \cite[Sec. 2]{GauOrnea} jointly with the $(2,0)$-form $\Psi_{\alpha,\beta} = \Phi_{\alpha,\beta}^{-1} dz_1 \wedge dz_2$ provide a solution of \eqref{eq:twistedStromexact}. However, $\|\Psi_{\alpha,\beta}\|_{\omega_{\alpha,\beta}}$ is not constant unless condition \eqref{eq:alphabeta} holds, and therefore in general 
Lemma \ref{lemma:holSUn} does not apply.
\end{remark}

The previous result can be used now as a guide to address Question \ref{question}. Recall that a Hopf surface does not admit any K\"ahler metric. 
Consequently, Proposition \ref{prop:twistedStromexact} shows that, already in the case of complex surfaces, 
the balanced cone of $X$ cannot be used to parametrize the solutions of \eqref{eq:twistedStromhol} (note that K\"ahler and balanced are equivalent conditions in this case).
Furthermore, 
by Proposition \ref{prop:twistedStromexact} a solution of \eqref{eq:twistedStromexact} with $[\theta_\omega] \neq 0$ is Vaisman \cite{GauIvanov}, and hence all the Morse-Novikov cohomology groups $H^k_{\theta_\omega}(X)$ vanish \cite{LLMP}. Therefore, $H^2_{\theta_\omega}(X)$ (and its Bott-Chern analogue \cite{OrneaVer}) is also ruled out as a a potential answer to Question \ref{question} (see Remark \ref{remark:lck}).
On the other hand, given a solution $(\Psi,\omega)$ of \eqref{eq:twistedStromexact} the hermitian form $\omega$ is pluriclosed, that is, $dd^c\omega = 0$, and it has an associated real class in Aeppli cohomology
$$
[\omega] \in H^{1,1}_A(X,\RR),
$$
where the Aeppli cohomology groups of $X$ are defined by (note that $H^{p,p}_A(X)$ has a natural real structure)
$$
H^{p,q}_A(X) = \frac{\Ker(dd^c\colon \Omega^{p,q} \to \Omega^{p+1,q+1})}{\Im(\partial \oplus \dbar \colon \Omega^{p,q-1}\oplus \Omega^{p-1,q} \to \Omega^{p,q})}.
$$
Motivated by Proposition \ref{prop:twistedStromexact}, we propose the following specialization of Question \ref{question} for the system \eqref{eq:twistedStromexact}. Recall that a real Aeppli class of bidegree $(1,1)$ is called positive if it is represented by a pluriclosed hermitian form.

\begin{question}\label{questionexact}
Let $X$ be a compact complex manifold with $c_1(X) = 0$. If $X$ admits a solution of the twisted Calabi-Yau equation \eqref{eq:twistedStromexact}, is there a unique solution for each positive Aeppli class in $H^{1,1}_A(X,\RR)$?
\end{question}

We are now ready to prove the main result of this section, which provides an affirmative answer to Question \ref{questionexact} in complex dimension $2$.

\begin{theorem}\label{th:Hopf}
If a compact complex surface $X$ admits a solution of the twisted Calabi-Yau equation \eqref{eq:twistedStromexact}, then it admits a unique solution $(\Psi,\omega)$ on each positive Aeppli class, up to rescaling of $\Psi$ by a unitary complex number.
\end{theorem}
\begin{proof}
When $[\theta_\omega] = 0$, $X$ is K\"ahler and therefore $H^{1,1}_A(X) \cong H^{1,1}(X)$, where $H^{1,1}(X)$ is the $(1,1)$ Dolbeault cohomology group of $X$. Thus, the statement in this case follows by Yau's Theorem for Calabi-Yau metrics \cite{Yau0}. When $[\theta_\omega] \neq 0$, Proposition \ref{prop:twistedStromexact} implies that $X$ is 
of class VII \cite{Kato}, and therefore $H^{1,1}_A(X) \cong \CC$ by \cite[Thm. 1.2]{AngellaBC}. Thus, the statement in this case follows by the explicit form of the solutions in the universal covering of $X$, given by \eqref{eq:seedgeneral}.
\end{proof}


To finish this section, we consider the case with arbitrary complex Lie group $G$. For our analysis we use a special feature of \eqref{eq:twistedStromhol} in complex dimension two, namely, that the first three equations of the system are conformally invariant. This follows easily from the behaviour of the Lee form $\theta_\omega$ and the norm $\|\Psi\|_\omega$ under conformal rescaling, that is, if $\omega' = e^f \omega$ for some smooth function $f$ on $X$, then
$$
\theta_{\omega'} = \theta_\omega + df, \qquad \|\Psi\|_{\omega'} = e^{-f} \|\Psi\|_\omega.
$$
\begin{proposition}\label{prop:conformal}
Let $X$ be a compact complex surface with $c_1(X) = 0$, endowed with a holomorphic principal $G$-bundle $P$ satisfying \eqref{eq:pic0}. Then,  $(X,P)$ admits a solution of the twisted Hull-Strominger system \eqref{eq:twistedStromhol} if and only if $X$ admits a solution $(\Psi,\omega)$ of \eqref{eq:twistedStromexact} such that $P$ is polystable with respect to $\omega$.
\end{proposition}
\begin{proof}
For the `if part', note that $P$ admits a reduction $h$ satisfying the Hermite-Einstein equation $F_h \wedge \omega = 0$. By the conformal invariance of the first three equations of the system it is enough to find a smooth real function $f$ on $X$ such that
$$
dd^c(e^f \omega) = c(F_h\wedge F_h).
$$
To prove the existence of this function, consider the differential operator
\begin{align*}
Q \colon C^\infty(X) &\to C^\infty(X)\\
\phi & \mapsto Q(\phi) := - * dd^c (\phi \omega)
\end{align*}
acting on smooth real functions on $X$. By \cite[Lem. 7.2.4]{lt}, we have $Q = P^*$, where $P^*$ is the adjoint of
$$
P(\phi) = 2i \Lambda_\omega \dbar \partial \phi.
$$
Applying \cite[Lem. 7.2.7]{lt} we have $\operatorname{coKer} Q = \Ker P = \mathbb{R}$, and therefore condition \eqref{eq:pic0} ensures the existence of a function $\phi$ such that $dd^c(\phi \omega) = c(F_h\wedge F_h)$. Using that $\omega$ is Gauduchon, we can add a constant $C > 0$ to $\phi$ so that $\phi + C > 0$ on $X$ and still solves the equation. Hence, by setting $e^f:= \phi + C$, the claim follows. 

Similarly, the `only if part' follows from the existence of a Gauduchon metric in the conformal class of the hermitian form solving \eqref{eq:twistedStromhol} \cite{Gau84}.
\end{proof}

\begin{example}\label{rem:bundles}
For a quaternionic diagonal Hopf surface--that is, with fundamental group generated by \eqref{eq:cyclicgroup}, with $\alpha = \beta$ and $\alpha \beta \in \RR$--, the moduli space of stable $\SL(2,\CC) \times \SL(2,\CC)$-bundles with second Chern class $c_2 = n_1 + n_2$ is non-empty and has complex dimension $4(n_1 + n_2)$ \cite{MorVer} (note that a quaternionic Hopf surface is hypercomplex, by Remark \ref{rem:quaternionic}). Taking the bilinear form \eqref{eq:ccomplex} on $\mathfrak{g} = \mathfrak{sl}(2,\CC) \oplus \mathfrak{sl}(2,\CC)$ to be $c = - n_2 \tr + n_1 \tr$, for $- \tr$ the Killing form on $\mathfrak{sl}(2,\CC)$, any point in the moduli space provides a polystable bundle satisfying \eqref{eq:pic0}, and hence fulfilling the hypothesis of Proposition \ref{prop:conformal}.
\end{example}

Combined with Proposition \ref{prop:twistedStromexact}, the previous result provides a complete characterization of the complex surfaces which may admit a solution. In a sense, Proposition \ref{prop:conformal} can be regarded as an analogue for the equations \eqref{eq:twistedStromhol} (in complex dimension two) of Yau's Conjecture for the Hull-Strominger system in \cite{Yau2}. On the contrary, Question \ref{question} seems to be more delicate. In the proof of the previous result we can choose a normalization for the Gauduchon metric $\tilde \omega$ in the conformal class of a solution $(\Psi,\omega,h)$ of \eqref{eq:twistedStromhol} by fixing the volume
\begin{equation}\label{eq:normalizationmap}
\int_X \tilde \omega^2 = \int_X \omega^2,
\end{equation}
in order to associate an Aeppli class $[\tilde \omega] \in H^{1,1}_A(X,\RR)$ to a given solution. However, the Aeppli classes which are achieved via this procedure seem to depend in a subtle way on the holomorphic bundle $P$.

\subsection{Examples on compact threefolds}\label{subsec:ex3fold}

This section is devoted to discuss some non-trivial examples of solutions of the twisted Hull-Strominger system with $[\theta_\omega] \neq 0$ on compact threefolds. Our analysis further illustrates the fact that the theory for the equations \eqref{eq:twistedStromhol} is richer than the theory for the Hull-Strominger system. In complex dimension three or higher the conformal invariance of the first three equations in \eqref{eq:twistedStromhol} is lost, and there seems to be no analogue of the proof of Proposition \ref{prop:conformal}. Motivated by this, in Section \ref{sec:stringholCourant} we will introduce a notion of Aeppli class associated to a general solution of the twisted Hull-Strominger system, by means of a special class of holomorphic Courant algebroids introduced in \cite{grt2}.

Our first family of examples builds on a construction in \cite{Bismutflat}. It corresponds to solutions of the twisted Calabi-Yau equation \eqref{eq:twistedStromexact} on a class of non-K\"ahler complex manifolds diffeomorphic to $S^3 \times T^3$. Consider the non-compact threefold
$$
\tilde X = (\CC^2 \backslash \{0\}) \times \CC \cong \SU(2) \times \RR^3,
$$
endowed with the solution of \eqref{eq:twistedStromexact} given by the product of \eqref{eq:seedgeneral} with the flat Calabi-Yau metric and holomorphic volume form on $\CC$. Consider the homomorphism $\rho \colon \mathbb{Z}^3 \to \SU(2)$ defined by
\begin{align*}
\rho(1,0,0) & {} = A, \\
\rho(0,1,0) & {} = (\cos \alpha) \Id + (\sin \alpha) A,\\
\rho(0,0,1) & {} = (\cos \beta) \Id + (\sin \beta) A,
\end{align*}
where $\alpha, \beta$ are real numbers and
$$ A = \frac{1}{\sqrt{2}}\left( \begin{array}{cc}
i & 1 \\
-1 & - i \end{array} \right).
$$
Using $\rho$, the abelian group $\mathbb{Z}^3$ acts by left translations on $\tilde X \cong \SU(2) \times \RR^3$ preserving the solution of \eqref{eq:twistedStromexact} (see Section \ref{subsec:Calabi}), and therefore the compact threefold $X = \tilde X/\mathbb{Z}^3$ also carries a solution. For generic choices of $(\alpha,\beta)$, any finite unbranched cover of $X$ cannot be the product of a Hopf surface and an elliptic curve. In particular, $X$ is neither an elliptic fibration over a Hopf surface nor a Hopf surface fibration over an elliptic curve. Since $X$ is a non-K\"ahler manifold, we necessarily have $[\theta_\omega] \neq 0$ (see Section \ref{subsec:sigmaX0}).

Our second family of examples is based on Sasaki-Einstein geometry, and relies on a construction in \cite{IvanovIvanov}. Let $M^5$ be a Sasaki-Einstein manifold of dimension five and define
$$
\tilde M^6 = M^5 \times \RR.
$$
Denote by $t$ a choice of coordinate on $\RR$. Given a Tanno deformation of the Sasaki-Einstein structure, $\tilde M^6$ carries a natural $\SU(3)$ structure $(\Psi,\omega)$--as in \cite[(7.54)]{IvanovIvanov}--, which is locally conformally K\"ahler, whose Lee form is given by $dt$, and whose Bismut connection $\nabla^+$ on $T\tilde M^6$ satisfies the Hermite-Yang-Mills equation
$$
F_{\nabla^+} \wedge \omega^2 = 0, \qquad F_{\nabla^+}^{0,2} = 0,
$$
and also $\nabla^+ \Psi = 0$ (see \cite[Section 7]{IvanovIvanov}). Furthermore, one has
$$
dd^c \omega = \frac{3}{8}\tr F_{\nabla^+} \wedge F_{\nabla^+}.
$$
By Lemma \ref{lemma:holSUn}, we conclude that $(\Psi,\omega,\nabla^+)$ provides a solution of \eqref{eq:twistedStromhol} on $\tilde M^6$ with $G = \SL(3,\CC)$ and $P$ given by the bundle of complex frames of $(\tilde M^6,\Psi)$ endowed with the non-standard holomorphic structure determined by $(\nabla^+)^{0,1}$.  This solution is invariant under the natural $\mathbb{Z}$-action on $\tilde M^6$, and therefore it descends to a solution on $M^6 = M^5 \times S^1$. Notice that $[\theta_\omega] \neq 0$ by construction, as it corresponds to the cohomology class induced by the Lee form on the $S^1$-factor. By well-known results in Sasaki-Einstein geometry, the topological type of $M^5$ can be taken to be $\sharp_{k} (S^3 \times S^2)$ with $k \geqslant 1$, simply connected rational homology spheres, including $S^5$, and connected sums of mixed type (cf. \cite{FHZ}).

\section{Bott-Chern algebroids: metrics and Aeppli classes.}\label{sec:stringholCourant}

\subsection{Holomorphic string algebroids}\label{subsec:courant}

In this section we show that any solution of the twisted Hull-Strominger system \eqref{eq:twistedStromhol} determines a holomorphic Courant algebroid of string type, as defined and classified in \cite{grt2}.
We start by recalling the basic definitions. Let $X$ be a complex manifold of dimension $n$. We denote by $\cO_X$ and $\underline{\CC}$ the sheaves of holomorphic functions and $\CC$-valued locally constant functions on $X$, respectively.

\begin{definition}\label{def:Courant}
A holomorphic Courant algebroid $(Q,\la\cdot,\cdot\ra,[\cdot,\cdot],\pi)$ over $X$ consists of a holomorphic vector bundle $Q \to X$, with sheaf of sections denoted also by $Q$, together with a holomorphic non-degenerate symmetric bilinear form $\la\cdot,\cdot\ra$, a holomorphic vector bundle morphism $\pi:Q\to TX$, and an homomorphism of sheaves of $\underline{\CC}$-modules
$$
[ \cdot,\cdot ] \colon Q \otimes_{\underline{\CC}} Q \to Q,
$$
satisfying the identities, for $e,e',e''\in Q$ and $\phi\in \cO_X$,
  \begin{itemize}
  \item $[e,[e',e'']] = [[e,e'],e''] + [e',[e,e'']]$,
  \item $\pi([e,e'])=[\pi(e),\pi(e')]$,
  \item $[e,\phi e'] = \pi(e)(\phi) e' + \phi[e,e']$,
  \item $\pi(e)\la e', e'' \ra = \la [e,e'], e'' \ra + \la e',
    [e,e''] \ra$,
  \item $[e,e']+[e',e]= \mathcal{D}\la e,e'\ra$,
  \end{itemize}
where $\mathcal{D} \colon \mathcal{O}_X \to Q$ denotes the composition of the exterior differential, the natural map $\pi^* \colon T^*X \to Q^*$, and the isomorphism $Q^* \to Q$ provided by $\la\cdot,\cdot\ra$.
\end{definition}


A holomorphic Courant algebroid is called transitive when the anchor map $\pi \colon Q \to TX$ is surjective. In this case, 
the quotient
$$
A_Q = Q/(\Ker \pi)^\perp
$$
is a vector bundle which inherits a holomorphic Lie algebroid structure. Furthermore, the holomorphic subbundle
$$
\ad_Q = \Ker \pi/(\Ker \pi)^\perp \subset A_Q
$$
inherits the structure of holomorphic bundle of quadratic Lie algebras.

We are interested in a particular class of transitive holomorphic Courant algebroids introduced in \cite{grt2}. 
Let $G$ be a complex Lie group  endowed with an invariant symmetric bilinear form $c \colon \mathfrak{g} \otimes \mathfrak{g} \to \CC$ on its Lie algebra, and let $p \colon P \to X$ be a holomorphic principal $G$-bundle over $X$. The holomorphic Atiyah Lie algebroid $A_{P}$ of $P$ has underlying holomorphic bundle $TP/G$,
whose local sections are given by $G$-invariant holomorphic vector fields on $P$, anchor map $dp \colon TP/G \to TX$,
and bracket induced by the Lie bracket on $TP$. The holomorphic bundle of Lie algebras
$
\Ker dp \subset A_{P}
$
corresponds to the adjoint bundle induced by the adjoint representation of $G$,
$$
\Ker dp \cong \ad P = P \times_{G} \mathfrak{g},
$$
and we have the short exact sequence of holomorphic Lie algebroids
$$
0 \to \ad P \to A_{P} \to TX \to 0.
$$

\begin{definition}[\cite{grt2}]
\label{def:stringholCour}
A string algebroid over $X$ with structure group $G$ and pairing $c$ is a tuple $(Q,P,\rho)$ such that $Q$ is a transitive holomorphic Courant algebroid $Q$, $P$ is a holomorphic principal $G$-bundle, and $\rho$ is a bracket-preserving morphism fitting into a short exact sequence
	\begin{equation}\label{eq:defstring}
	\xymatrix{
		0 \ar[r] & T^*X \ar[r] & Q \ar[r]^\rho & A_P \ar[r] & 0,
	}
	\end{equation}
such that the induced map of holomorphic Lie algebroids $\rho \colon A_Q \to A_P$ is an isomorphism restricting to an isomorphism $\ad_Q \cong (\ad P,c)$.
\end{definition}


An isomorphism of string algebroids is given by a commutative diagram
\begin{equation}\label{eq:defstringiso}
    \xymatrix{
      0 \ar[r] & T^*X \ar[r] \ar[d]^{id} & Q \ar[r]^\rho \ar[d]^{\varphi} & A_P \ar[r] \ar[d]^{g} & 0\\
      0 \ar[r] & T^*X \ar[r] & Q' \ar[r]^{\rho'} & A_{P'} \ar[r] & 0
    }
\end{equation}
where $\varphi \colon Q \to Q'$ is an isomorphism of holomorphic Courant algebroids and $g \colon A_P \to A_{P'}$ is induced by an isomorphism of holomorphic principal bundles covering the identity on $X$.

Given a string algebroid $(Q,P,\rho)$ we will say that $P$ is the underlying principal bundle of $Q$. In the sequel, we will abuse of the notation and denote a string algebroid simply by $Q$. Notice that, when $G = \{1\}$, a string algebroid corresponds to an exact holomorphic Courant algebroid
\begin{equation*}
0 \to T^*X \to Q \to TX \to 0.
\end{equation*}

For later use, we recall next that a suitable choice of a three-form $H$ and a connection $\theta$ produces a string algebroid (see \cite[Prop. 2.4]{grt2}).

\begin{example}\label{ex:algebroidH}
Let $P$ be a holomorphic principal $G$-bundle over $X$. We denote by $\underline{P}$ the underlying smooth bundle, and consider the Atiyah Lie algebroid $A^{1,0} := T\underline{P}^{1,0}/G$ with anchor $\pi_A:A^{1,0}\to T^{1,0}X$. Let $\theta$ be a connection on $\underline{P}$ such that $F_{\theta}^{0,2}=0$ and that $\theta^{0,1}$ induces $P$. Assume that there exists $H\in \Omega^{3,0}\oplus \Omega^{2,1}$ such that $dH + c(F_\theta\wedge F_\theta) = 0$. Consider the smooth bundle
$$
\underline{Q} = A^{1,0}\oplus (T^{1,0}X)^*,
$$
whose sections are denoted by $V + \xi$ and $W+\eta$. We endow $\underline{Q}$ with the pairing
\begin{equation}\label{eq:smcxpairing}
\langle V + \xi , V + \xi \rangle = \xi(\pi_A V) + c(\theta V,\theta V),
\end{equation}
where $\theta V$ denotes vertical projection, the anchor map $\pi_{\underline{Q}}(V+\xi) = \pi_A V$, and the bracket
\begin{equation*}\label{eq:bracket}
\begin{split}
[V+ \xi,W + \eta]  = {} & [V,W] + \partial i_{\pi_A V} \eta + i_{\pi_A V}\partial \eta - i_{\pi_A W}\partial \xi + i_{\pi_A V}i_{\pi_A W}H^{3,0} \\
& + 2c(\partial^\theta(\theta V),\theta W) + 2c(i_{\pi_A V} F^{2,0}_\theta, \theta W)-2c(i_{\pi_A W} F^{2,0}_\theta,\theta V).
\end{split}
\end{equation*}
Finally, using the holomorphic structures on $X$ and $A^{1,0}$, define the $\bar{\partial}$-operator
$$
\dbar_{\underline{Q}} (V + \xi) = \dbar V +  \dbar \xi + i_{\pi_A V}H^{2,1} + 2c(F^{1,1}_\theta,\theta V).
$$
Then, $\dbar_{\underline{Q}}^2 = 0$ and the previous construction endows the sheaf of holomorphic sections of $(\underline{Q},\dbar_{\underline{Q}})$ with the structure of a string algebroid. Note that, in particular, an element $r + \xi \in \Ker \pi_{\underline{Q}}$ is holomorphic if and only if
\begin{equation}\label{eq:holomorphicsection}
\dbar_{\underline{Q}} (r + \xi) = \dbar r + \dbar \xi + 2c(F^{1,1}_\theta,r) = 0.
\end{equation}
\end{example}

We recall next the classification of string algebroids obtained in \cite{grt2}. In the sequel, we will assume that $G$ is a reductive complex Lie group with fixed bilinear form $c$. Consider the elliptic complex $(\Omega^{\leqslant\bullet},d)$ defined by
$$
\Omega^{\leqslant k} = \oplus_{j\leqslant k} \Omega^{j+ 2,k-j},
$$
with the convention that $\Omega^{p,q} = 0$ if $p < 0$ or $q <0$, and the usual exterior de Rham differential
\begin{equation*}
  \xymatrix@R-2pc{
\ldots \ar[r]^d & \Omega^{\leqslant k} \ar[r]^d & \Omega^{\leqslant k+1} \ar[r]^d & \ldots
}
\end{equation*}
Explicitly, for $0 \leqslant k \leqslant 2$,
\begin{align*}
\Omega^{\leqslant 0} & = \Omega^{2,0},\\
\Omega^{\leqslant 1} & = \Omega^{3,0} \oplus \Omega^{2,1},\\
\Omega^{\leqslant 2} & = \Omega^{4,0} \oplus \Omega^{3,1} \oplus \Omega^{2,2}.
\end{align*}
Then, \cite[Cor. 3.7]{grt2} states that there is an exact sequence of pointed sets
\begin{equation}\label{eq:lescEind2cx}
  \xymatrix{
0 \ar[r] & H^1(\Omega^{\leqslant \bullet}) \ar[r]^{\quad \iota} & H^1(\cS) \ar[r]^\jmath & H^1(\cO_{G}) \ar[r]^{^{p_1}\quad } & H^2(\Omega^{\leqslant \bullet}),
  }
\end{equation}
where $H^1(\cS)$ denotes the set of isomorphism classes of string algebroids with structure group $G$ 
and $H^1(\cO_{G})$ is the set of isomorphism classes of holomorphic principal $G$-bundles on $X$. The map $p_1$ is given by
$$
p_1([P]) = [c(F_\theta \wedge F_\theta)] \in H^2(\Omega^{\leqslant \bullet})
$$
for any choice of connection $\theta$ on the smooth principal $G$-bundle $\underline{P}$ underlying $P$, such that $F_\theta^{0,2} = 0$ and whose $(0,1)$-part induces $P$.

\begin{definition}[\cite{grt2}]\label{def:holstringclass}
Let $P$ be a holomorphic principal $G$-bundle over $X$ with $p_1([P]) = 0 \in H^2(\Omega^{\leqslant \bullet})$. The set of holomorphic string classes on $P$ is 
$\jmath^{-1}([P])$.
\end{definition}

A holomorphic string class on $P$ is therefore an isomorphism class of string algebroids with underlying bundle isomorphic to $P$. Recall that the vector space $H^1(\Omega^{\leqslant \bullet})$ in \eqref{eq:lescEind2cx} classifies exact Courant algebroids on $X$ up to isomorphism \cite{G2}. Regarded as an additive group, it acts transitively on $\jmath^{-1}([P])$. 
The next result from \cite{grt2} provides a classification \emph{\`a la de Rham} of the set of holomorphic string classes on $P$ and characterizes the isotropy of the $H^1(\Omega^{\leqslant \bullet})$-action.

To state the result, denote by $\mathcal{A}_{P}$ the space of connections on the underlying smooth bundle $\underline{P}$ such that $F_\theta^{0,2} = 0$ and whose $(0,1)$-part induces $P$. Denote by $\cG_P$ the \emph{holomorphic gauge group} of $P$, given by automorphisms of $P$ which project to the identity on $X$. Recall from \cite[Lem. 2.11]{grt2} that there is a group homomorphism
\begin{equation}\label{eq:sigmaP}
\sigma_P \colon \cG_P \to H^1(\Omega^{\leqslant \bullet})
\end{equation}
defined by
\begin{equation}\label{eq:sigmaP-formula}
\sigma_P(g) = [CS(g\theta) - CS(\theta) - d\la g\theta \wedge \theta^h \ra ] \in  H^1(\Omega^{\leqslant \bullet}),
\end{equation}
for any choice of connection $\theta \in \mathcal{A}_{P}$. Here, $CS(\theta) \in \Omega^3(\underline{P})$ denotes the Chern-Simons three-form of $\theta$, which satisfies
$$
d CS(\theta) = c(F_\theta \wedge F_\theta).
$$
As we will see next, the quotient
$$
 H^1(\Omega^{\leqslant \bullet})/\Im \; \sigma_P
$$
can be identified with the set of isomorphism classes of string algebroids with underlying holomorphic principal $G$-bundle $P$ (see \cite[Prop. 3.11]{grt2}).

\begin{theorem}[\cite{grt2}]\label{lemma:deRhamC}
There is a natural bijection
\begin{equation*}\label{eq:lescEind3}
\jmath^{-1}([P]) \cong \{(H,\theta) \in \Omega^{\leqslant 1} \times \cA_{P} \; | \; dH + c(F_\theta \wedge F_\theta) = 0\}/ \sim,
\end{equation*}
where $(H,\theta)\sim (H',\theta')$ if, for some $B\in \Omega^{2,0}$ and $g\in \cG_P$
\begin{equation}\label{eq:anomaly}
H'  = H + CS(g\theta) - CS(\theta') - dc(g\theta \wedge \theta') + dB.
\end{equation}
Furthermore, the isotropy of $\jmath^{-1}([P])$ for the transitive $H^1(\Omega^{\leqslant \bullet})$-action is $\Im\; \sigma_P$.
\end{theorem}

A choice of pair $(H,\theta)$ in the equivalence class determines a particular presentation of the Courant algebroid $Q$ as in Example \ref{ex:algebroidH}.

The upshot of the previous classification is the relation between the twisted Hull-Strominger system \eqref{eq:twistedStromhol} and the string algebroids in the next result. The proof is a straightforward consequence of Theorem \ref{lemma:deRhamC}.

\begin{proposition}\label{prop:stromholcour}
A solution $(\Psi,\omega,h)$ of the twisted Hull-Strominger system \eqref{eq:twistedStromhol} on $(X,P)$ determines a string algebroid $Q$ over $X$ with underlying principal bundle $P$, given by $(2i\partial \omega,\theta^h)$ as in Example \ref{ex:algebroidH}. 
\end{proposition}

\begin{remark}
Note that, even though $Q$ in Proposition \ref{prop:stromholcour} does not depend on the complex $(n,0)$-form $\Psi$, the fibre of the forgetful map $(\Psi,\omega,h) \mapsto (\omega,h)$ on the space of solutions of \eqref{eq:twistedStromhol} is a circle, since by Lemma \ref{lemma:holSUn} the form $\Psi$ is parallel with respect to the Bismut connection of $\omega$ and  satisfies \eqref{eq:normalizationnorm}.
\end{remark}

\subsection{Bott-Chern algebroids and hermitian metrics}\label{subsec:holCourhermit}

Relying on Proposition \ref{prop:stromholcour}, this paper builds on the idea that the existence and uniqueness problem for the twisted Hull-Strominger system \eqref{eq:twistedStromhol} (and hence for \eqref{eq:Stromhol}) should be better understood as the problem of finding `the best metric' in a string algebroid $Q$.
Expanding upon this idea, in this section we introduce and study a particular class of  string algebroids and a natural notion of `hermitian metric' on them.


Let $X$ be a complex manifold of dimension $n$. Let $G$ be a reductive complex Lie group, with bi-invariant symmetric bilinear form $c : \mathfrak{g} \otimes \mathfrak{g} \to \mathbb{C}$. We assume that $c$ satisfies the reality condition
\begin{equation}\label{eq:creal}
c(\mathfrak{k}\otimes \mathfrak{k})\subset \R
\end{equation}
for the Lie algebra $\mathfrak{k}$ of any maximal compact subgroup $K \subset G$. Let $P$ be a holomorphic principal $G$-bundle over $X$ such that
$$
p_1([P]) = 0 \in H^2(\Omega^{\leqslant \bullet}).
$$
By \cite[Cor. 3.6]{grt2}, when $X$ is a $\partial \dbar$-manifold the previous condition is equivalent to the vanishing of $p_1([P])$ in de Rham cohomology $H^4(X,\CC)$. Throughout this section, we consider string algebroids $Q$ with fixed underlying principal bundle $P$, as described in Theorem \ref{lemma:deRhamC}. 

Motivation for the next definition comes from the fact that the string algebroids associated to solutions of \eqref{eq:twistedStromhol} are special, due to condition \eqref{eq:pic0}. 
In particular, the representative $(H,\theta)$ of the holomorphic string class in Theorem \ref{lemma:deRhamC} 
can be taken to be $H \in \Omega^{2,1}$, that is, with no component in $\Omega^{3,0}$.


\begin{definition}\label{def:BCtype}
A Bott-Chern algebroid is a string algebroid $Q$ whose class is given by $[(2i\partial \tau,\theta^h)]$ for a real $(1,1)$-form $\tau \in \Omega^{1,1}$ and $\theta^h$ the Chern connection of a reduction $h \in \Omega^0(P/K)$ of $P$ to a maximal compact subgroup $K$.
\end{definition}

From the previous definition and Theorem \ref{lemma:deRhamC} it follows that
\begin{equation}\label{eq:misteriouseqb}
dd^c \tau = c(F_h \wedge F_h),
\end{equation}
where $F_h : = F_{\theta^h}$, and therefore a necessary condition for $Q$ to be Bott-Chern is that $p_1([P])$ vanishes in Bott-Chern cohomology (cf. \eqref{eq:pic0})
$$
p_1(P) = 0 \in H^{2,2}_{BC}(X).
$$

We next state a parametrization of the Bott-Chern algebroids for fixed $X$ and $P$. For this, note that there is a well-defined linear map induced by the $\partial$-operator (see \eqref{eq:sigmaP})
\begin{equation}\label{eq:partialmap}
\partial \colon H^{1,1}_A(X,\mathbb{R}) \to H^1(\Omega^{\leqslant \bullet})/\Im \; \sigma_P.
\end{equation}
The proof requires some additional tools, and it is postponed until Section \ref{subsec:Aeppli}.

\begin{proposition}\label{prop:BCclassification}
The set of equivalence classes of Bott-Chern algebroids over $X$ with fixed principal bundle $P$ is an affine space for the vector space given by the image of \eqref{eq:partialmap}. Furthermore, if $X$ is a $\partial\bar{\partial}$-manifold then there is only one equivalence class.
\end{proposition}


\begin{example}\label{example:ddbarBC}
	Let $X$ be a $\partial\dbar$-manifold (see e.g. \cite{AngellaBC}) and assume $G = \{1\}$. Then, since a string algebroid $Q$ with trivial structure group is exact, we have $H^1(\cS) = H^1(\Omega^{\leqslant \bullet})$ and there is a short exact sequence \cite{G2} (cf. \eqref{eq:lescEind2cx})
	$$
	0 \to H^{3,0}_{\dbar}(X) \to H^1(\Omega^{\leqslant \bullet}) \to H^{2,1}_{\dbar}(X) \to 0.
	$$
	If $Q$ is Bott-Chern, as there is only one equivalence class, it must be isomorphic to $TX \oplus T^*X$. 
\end{example}

\begin{example}\label{example:HopfQBC}
Let $X$ be a compact complex surface of class VII with $b_2  = 0$. When $G = \{1\}$ we have
$$
H^1(\cS) = H^1(\Omega^{\leqslant \bullet}) \cong H^{2,1}_{\dbar}(X),
$$
where the last equality follows by dimensional reasons. Using that $H^{2,1}_{\dbar}(X)$ and $H^{1,1}_A(X)$ are both one-dimensional, the $\CC$-linear map
\begin{equation}\label{eq:isomorphismcohomo}
\begin{split}
H^{1,1}_A(X) & \to H^{2,1}_{\dbar}(X) \\
[\tau] &\mapsto [2i\partial \tau]
\end{split}
\end{equation}
is an isomorphism \cite{AngellaBC}. Consequently, an exact holomorphic Courant algebroid is Bott-Chern if and only if its holomorphic string class is in the image via \eqref{eq:isomorphismcohomo} of the real subspace $H^{1,1}_A(X,\mathbb{R}) \subset H^{1,1}_A(X)$.
\end{example}

The last example suggests that a string algebroid $Q$ can be twisted by $dd^c$-closed $(1,1)$-form, in a way that it preserves the underlying holomorphic principal bundle and the Bott-Chern property. This construction is similar to the natural twist of a holomorphic vector bundle by a line bundle. Note that the underlying smooth bundle $\underline{Q}$ of a holomorphic Courant algebroid $Q$ is endowed with a bracket and an anchor map, similarly as in Example \ref{ex:algebroidH}.

\begin{definition}\label{def:twist}
We define $Q \otimes \beta$, the twist of $Q$ by a $dd^c$-closed form $\beta \in \Omega^{1,1}$, as the same underlying smooth Courant algebroid with the new holomorphic structure
$$\dbar_{Q\otimes \beta}=\dbar_Q + \dbar \beta,$$
and whose bracket and anchor maps are given by restriction of the ones in $\underline{Q}$ to the sheaf of holomophic sections.
\end{definition}

In particular, if $Q$ is given as in Example \ref{ex:algebroidH} by $(H,\theta)$, the string algebroid $Q\otimes \beta$ corresponds to $(H + 2i\partial \beta,\theta)$. At the level of classes, a holomorphic string class $[Q]$ is twisted by an Aeppli class $\alpha \in H_A^{1,1}(X)$ by $[Q]\otimes \alpha:=[Q\otimes \beta]$, for any $\beta$ such that $[\beta] = \alpha$. When $\beta$ is real, twisting by $\beta$ preserves the Bott-Chern property.


When $X$ satisfies the $\partial\dbar$-Lemma, it is easy to see that the twists by $dd^c$-closed $(1,1)$-forms correspond to automorphisms of the Courant algebroid $Q$.

\begin{lemma}\label{lem:ddbarisoQ}
Let $X$ be a $\partial\dbar$-manifold. Then, for any $dd^c$-closed form $\beta \in \Omega^{1,1}$, we have that $[Q \otimes \beta]=[Q]$.
\end{lemma}

We are now ready to introduce our notion of hermitian metric on $Q$, assuming the Bott-Chern property.

\begin{definition}\label{def:metricQ}
Let $Q$ be a Bott-Chern algebroid.
\begin{enumerate}[i)]

\item A hermitian metric on $Q$ is a pair $(\omega,h)$, where $\omega$ is a positive $(1,1)$-form on $X$, and $h \in \Omega^0(P/K)$ is a reduction of $P$ to a maximal compact subgroup $K$, such that $[Q]=[(2i\partial \omega,\theta^h)].$

\item We say that $Q$ is \emph{positive} if it admits a hermitian metric.
\end{enumerate}
\end{definition}

Notice that the (possibly empty) set of hermitian metrics on a given Bott-Chern algebroid $Q$ only depends on the holomorphic string class $[Q]$. A priori, there is no obvious way to check the positivity of $Q$. In fact, since a hermitian metric $(\omega,h)$ on $Q$ satisfies \eqref{eq:misteriouseqb} with $\tau = \omega$, this question can be regarded as a strong version of \cite[Question 5.11]{GF2}. Nonetheless, under the assumption that $X$ admits a positive class in $H^{1,1}_A(X,\RR)$---that is, represented by pluriclosed hermitian form $\omega_0$ on $X$---the twisted Courant algebroid
$$
Q \otimes k\omega_0
$$
is positive for large enough $k \gg 0$. This type of manifolds are known in the literature as strong K\"ahler with torsion (SKT).

\begin{proposition}\label{prop:BCpositive}
Suppose that $X$ is SKT. Then, for any Bott-Chern algebroid $Q$ there exists a sufficiently large constant $k >0$ such that $Q \otimes k\omega_0$ is positive. In addition, if $X$ satisfies the $\partial \dbar$-Lemma, any  Bott-Chern algebroid $Q$ is positive.
\end{proposition}
\begin{proof}
By the Bott-Chern property, we can choose $(H,\theta)$ such that $H = 2i\partial \tau_0$ with $\tau_0 \in \Omega^{1,1}$ and $\theta = \theta^h$ for a reduction $h$ on $P$. The first part of the proof is an immediate consequence of Definition \ref{def:twist}, choosing $\beta = k \omega_0$. The second part of the statement follows from Lemma \ref{lem:ddbarisoQ}.
\end{proof}

\begin{example}\label{example:HopfQpositive}
In Example \ref{example:HopfQBC}, the locus of isomorphism classes of positive objects is identified via \eqref{eq:isomorphismcohomo} with
$$
\RR_{> 0} \langle [\omega_0] \rangle 
\subset H^{1,1}_A(X).
$$
\end{example}

\begin{example}\label{example:Ugarte}
The hypotheses of Proposition \ref{prop:BCpositive} are sufficient but not necessary conditions for positivity. Various homogeneous solutions of the Hull-Strominger system \eqref{eq:Stromhol} found in the literature provide examples of  positive Bott-Chern algebroids with $G \neq \{1\}$ (see e.g. \cite{OUVi}). These manifolds are neither $\partial\dbar$-manifolds nor SKT, as, for instance, the nilmanifold with underlying Lie algebra $\mathfrak{h}_3$ considered in \cite{FIVU}.
\end{example}

\subsection{Aeppli classes in Bott-Chern algebroids}\label{subsec:Aeppli}

Let $Q$ be a Bott-Chern 
algebroid over a compact complex manifold $X$, with underlying holomorphic principal $G$-bundle $P$. 
The aim of this section is to introduce a notion of `real Aeppli class' in $Q$ and to associate such a class to any hermitian metric, provided that $Q$ is positive. We need the following result, which 
defines secondary characteristic classes originally introduced by Bott and Chern  \cite{BottChern}.

\begin{proposition}[\cite{BGS,Don}]\label{prop:Donaldson}
For $h_0, h_1$ two 
reductions of $P$ to a maximal compact subgroup, there exists an invariant
\begin{equation}\label{eqref:BCinvariant}
R(h_1,h_0) \in \Omega^{1,1}/\operatorname{Im}(\partial \oplus \dbar)
\end{equation}
with the following properties:
\begin{enumerate}

\item $R(h_0,h_0) = 0$, and, for any third reduction $h_2$,
$$
R(h_2,h_0) = R(h_2,h_1) + R(h_1,h_0).
$$

\item if $h$ varies in a one-parameter family $h_t$, then
\begin{equation}\label{eqref:BCinvariantder}
\frac{d}{dt}R(h_t,h_0) = -2ic(\dot h_t h^{-1}_t,F_{h_t}).
\end{equation}

\item the following identity holds
$$
dd^c R(h_1,h_0) = c(F_{h_1}\wedge F_{h_1}) - c(F_{h_0}\wedge F_{h_0}).
$$
\end{enumerate}
\end{proposition}

As observed by Donaldson in \cite[Prop. 6]{Don}, the \emph{Bott-Chern class} \eqref{eqref:BCinvariant} is defined by integration of \eqref{eqref:BCinvariantder} along a path in the space of 
reductions of $P$, using that the resulting $(1,1)$-form is independent of the chosen path up to addition of elements in $\operatorname{Im}(\partial \oplus \dbar)$. This constructive method will be important for the proof of Lemma \ref{lem:tauhparamet} below. Here, we use the polar decomposition of the group $G = \exp(i \mathfrak{k}) \cdot K$ to regard $h$ as a $K$-equivariant map $h \colon P \to \exp(i \mathfrak{k})$ and therefore one has $\frac{d}{dt}\theta^{h_t} = - 2 \partial^{h_t}(\dot h_t h^{-1}_t)$, which differs from the notation in \cite{Don}.


In order to introduce our notion of real Aeppli classes for the Bott-Chern algebroid $Q$, we consider
$$ 
B_Q := \{ (\tau,h) \in \Omega^{1,1}\times \Omega^0(P/K) \st \; \overline{\tau} = \tau, \; dd^c\tau = c(F_h\wedge F_h), \; [Q] = [(2i\del \tau,\theta^h)] \}.
$$
Note that $B_Q$ depends only on the class $[Q]$. 

\begin{proposition}\label{propo:Aepplicone}
There is a well-defined map
$$
Ap \colon B_Q \times B_Q \to H^{1,1}_A(X,\RR)
$$
defined by the formula
$$
Ap(\tau,h,\tau_0,h_0) = [\tau - \tau_0 - R(h,h_0)].
$$
Furthermore, $Ap$ satisfies the cocycle condition
\begin{equation}\label{eq:Apcocycle}
Ap(\tau_2,h_2,\tau_0,h_0) = Ap(\tau_2,h_2,\tau_1,h_1) + Ap(\tau_1,h_1,\tau_0,h_0)
\end{equation}
for any triple of elements in $B_Q$.
\end{proposition}

\begin{proof}
Condition \eqref{eq:misteriouseqb} for elements in $B_Q$, combined with property $(3)$ in Proposition \ref{prop:Donaldson}, imply that $Ap$ is well defined. Furthermore, as a consequence of $(1)$ in Proposition \ref{prop:Donaldson} the map $Ap$ satisfies the cocycle condition \eqref{eq:Apcocycle}.
\end{proof}

As a straightforward consequence of the cocycle condition \eqref{eq:Apcocycle}, we obtain that the map $Ap$ induces an equivalence relation in $B_Q$ defined by
$$
(\tau,h) \sim (\tau',h') \quad \textrm{\emph{if and only if}} \quad Ap(\tau,h,\tau',h') = 0.
$$
We hence make the following definition.

\begin{definition}\label{def:ApQ}
The set of \emph{(real) Aeppli classes} of $Q$ is the quotient 
$$
\Sigma_Q(\RR) := B_Q/\sim.
$$
\end{definition}

The space $B_Q$ decomposes as a disjoint union
\begin{equation*}\label{eq:Aepplidecomp}
B_Q = \bigsqcup_{\sigma \in \Sigma_Q(\RR)} B_\sigma,
\end{equation*}
where $B_\sigma$ denotes the set of elements of $B_Q$ with class $\sigma$. Note that a choice $(\tau_0,h_0) \in B_Q$ identifies $\Sigma_Q(\RR)$ with a subset of $H^{1,1}_A(X,\RR)$. The definition of $B_Q$ combined with the proof of Proposition \ref{prop:BCclassification}, imply that $\Sigma_Q(\RR)$ has a natural structure of affine space modelled on the kernel of the linear map \eqref{eq:partialmap}. By Lemma \ref{lem:ddbarisoQ}, when $X$ satisfies the $\partial \dbar$-Lemma $\Sigma_Q(\RR)$ is modelled on $H^{1,1}_A(X,\RR)$.
When $X$ is not a $\partial \dbar$-manifold the situation is very different, as we illustrate with the following example.

\begin{example}\label{example:HopfAeppli}
Consider a compact complex surface $X$ of class VII with $b_2  = 0$, endowed with an exact holomorphic Courant algebroid $Q$. Assume that $Q$ is Bott-Chern. By Example \ref{example:HopfQBC} the isomorphism class of $Q$ corresponds to $[\tau_0] \in H^{1,1}_A(X,\mathbb{R})$ via \eqref{eq:isomorphismcohomo},
for a real $dd^c$-closed $\tau_0 \in \Omega^{1,1}$, and we have
$$
Ap(\underline{\;\;\;\;},\tau_0) \colon B_Q \to H^{1,1}_A(X,\RR) \colon \tau \to [\tau - \tau_0].
$$
Using that \eqref{eq:isomorphismcohomo} is an isomorphism we conclude that $Ap(\underline{\;\;\;\;},\tau_0) = 0$ and therefore $\Sigma_Q(\RR)$ reduces to a point in this case.
\end{example}

We assume now that $Q$ is positive and denote by
$$
B_Q^+ \subset B_Q
$$
the subset consisting of the hermitian metrics on $Q$. Similarly, we denote by $B_\sigma^+ \subset B_\sigma$ the set of hermitian metrics with class $\sigma \in \Sigma_Q(\RR)$ (see \eqref{eq:Aepplidecomp}).

\begin{definition}
An Aeppli class on $Q$ is positive if it is represented by a hermitian metric $(\omega,h) \in B_Q^+$.
\end{definition}

To finish this section, we give an explicit parametrization of the set of hermitian metrics in a fixed positive Aeppli class, and prove the classification of Bott-Chern algebroids stated in Proposition \ref{prop:BCclassification}. These results, which are crucial for our developments in the subsequent sections, rely on the following key lemma, which upgrades the Bott-Chern class \eqref{eqref:BCinvariant} to take values in $\Omega^{1,1}$. 

\begin{lemma}\label{lem:CSRinvariant}
Let $h,h_0$ be reductions of $P$. Define
\begin{equation}\label{eq:Rhh0}
\tilde R(h,h_0) = - 2i \int_0^1 c(\dot h_t h_t^{-1},F_{h_t}) dt \in \Omega^{1,1}
\end{equation}
where $h = e^uh_0$, for $u \in \Omega^0(i \ad P_{h_0})$, and $h_t = e^{tu}h_0$. Then,
\begin{equation}\label{eq:hermitianformula-notused}
2i\partial \tilde R(h,h_0) + CS(\theta^{h}) - CS(\theta^{h_0}) - dc(\theta^{h} \wedge \theta^{h_0}) \in d \Omega^{2,0}.
\end{equation}
\end{lemma}

\begin{proof}
We set
\begin{align*}
C_t  :&= CS(\theta^{h_t}) - CS(\theta^{h_0}) - dc(\theta^{h_t} \wedge \theta^{h_0})\\
& = - 2c(a_t \wedge F_{h_t}) - c(a_t \wedge d^{h_t} a_t) - \frac{1}{3}c(a_t \wedge [a_t,a_t]),
\end{align*}
where $h_t = e^{tu}h_0$ as in \eqref{eq:Rhh0}, and $a_t = \theta^{h_0} - \theta^{h_t}$. By the fundamental theorem of calculus, it suffices to prove that
\begin{equation}\label{eq:keycalculation}
\frac{d}{dt}(2i\partial \tilde R(h_t,h_0) + C_t) \in d \Omega^{2,0}
\end{equation}
for all $t \in [0,1]$. To see this, note that
\begin{align*}
\frac{d}{dt} F_{h_t} & = \frac{d}{dt} (d \theta^{h_t} + [\theta^{h_t},\theta^{h_t}]/2) = - d\dot a_t - [\theta^{h_t},\dot a_t] = - d^{h_t}\dot a_t,\\
\frac{d}{dt} (d^{h_t} a_t) & = \frac{d}{dt}(da_t + [\theta^{h_t},a_t]) = - [\dot a_t,a_t] + d^{h_t} \dot a_t,
\end{align*}
and hence
\begin{align*}
\dot C_t := &  - 2c(\dot a_t \wedge F_{h_t}) + 2c(a_t \wedge d^{h_t}\dot a_t) \\
& - c(\dot a_t \wedge d^{h_t} a_t) + c(a_t \wedge [\dot a_t, a_t]) - c(a_t \wedge d^{h_t} \dot a_t) \\
& - \frac{1}{3}c(\dot a_t \wedge [a_t,a_t]) - \frac{1}{3}c(a_t\wedge [\dot a_t,a_t]) - \frac{1}{3}c(a_t \wedge[a_t,\dot a_t])\\
= &  - 2c(\dot a_t \wedge F_{h_t}) + c(a_t \wedge d^{h_t}\dot a_t) \\
& - c(\dot a_t \wedge d^{h_t} a_t) + c(a_t \wedge [\dot a_t, a_t]) - c(a_t \wedge [\dot a_t, a_t])\\
= & - 2c(\dot a_t \wedge F_{h_t}) - d c(a_t \wedge \dot a_t),
\end{align*}
where in the second equality we have used that
$$
c(a_t \wedge [\dot a_t, a_t]) = c(a_t \wedge [a_t, \dot a_t]) = c([a_t,a_t] \wedge \dot a_t) = c(\dot a_t \wedge [a_t,a_t]).
$$
Finally, part $(2)$ of Proposition \ref{prop:Donaldson} implies that
$$
\frac{d}{dt} 2i\partial \tilde R(h_t,h_0) = 4 \partial c(\dot h_t h_t^{-1},F_{h_t}) = 4 c(\partial^{h_t}(\dot h_t h_t^{-1}),F_{h_t}),
$$
and thus \eqref{eq:keycalculation} follows from (cf. \cite[Sec. 1]{Don})
$$
\dot a_t = - \frac{d}{dt} \theta^{h_t} = 2 \partial^{h_t} (\dot h_t h_t^{-1}).
$$
\end{proof}

Using Lemma \ref{lem:CSRinvariant}, we give next our parametrization of the set of hermitian metrics in a fixed positive Aeppli class $\sigma \in \Sigma_Q(\RR)$.

\begin{lemma}\label{lem:tauhparamet}
Let $(\omega_0,h_0) \in B^+_\sigma$ be a hermitian metric on $Q$ with Aeppli class $\sigma \in \Sigma_Q(\RR)$. Let $(\omega,h)$ be a pair given by a hermitian form $\omega \in \Omega^{1,1}$ on $X$ and a reduction $h$ of $P$ to a maximal compact subgroup. Then, if
\begin{equation}\label{eq:hermitianformula}
\omega - \omega_0 - \tilde R(h,h_0) \in \operatorname{Im}(\partial \oplus \dbar),
\end{equation}
the pair $(\omega,h)$ defines a hermitian metric on $Q$ with Aeppli class $\sigma$. Conversely, if $(\omega,h) \in B^+_\sigma$, then $\omega$ satisfies \eqref{eq:hermitianformula}.
\end{lemma}

\begin{proof}
If $(\omega,h)$ is a hermitian metric on $Q$ satisfying \eqref{eq:hermitianformula} then it follows by the definition of $\Sigma_Q(\RR)$ that the class of $(\omega,h)$ is $\sigma$. Thus, following Definition \ref{def:metricQ}, for the first part it suffices to prove that $\eqref{eq:hermitianformula}$ implies $(2i\partial \omega,\theta^h) \sim (2i\partial \omega_0,\theta^{h_0})$. This is a straightforward consequence of Lemma \ref{lem:CSRinvariant}, as we have
\begin{align*}
2i\partial \omega & =  2i\partial \omega_0 + 2i\partial \tilde R(h,h_0) - d(2i\partial \eta^{1,0})\\
& = 2i\partial \omega_0 + CS(\theta^{h_0}) - CS(\theta^{h}) - dc(\theta^{h_0} \wedge \theta^{h}) + dB
\end{align*}
for some $\eta^{1,0} \in \Omega^{1,0}$ and $B \in \Omega^{2,0}$. For the converse, if $(\omega,h)$ is a hermitian metric on $Q$ with class $\sigma$ then
$$
\omega - \omega_0 - R'(h,h_0) \in \operatorname{Im} (\partial \oplus \dbar ),
$$
where
$$
R'(h,h_0) = -2i \int_0^1 c( \dot h_s h_s^{-1},F_{h_s}) ds
$$
for a suitable path $h_s$ joining $h_0$ and $h$. Now, by the proof of \cite[Prop. 6]{Don}, $R'(h,h_0)$ and $\tilde R(h,h_0)$ differ by an element in $\operatorname{Im}(\partial \oplus \dbar)$ and the result follows.
\end{proof}

By Lemma \ref{lem:tauhparamet}, upon fixing $(\omega_0,h_0) \in B^+_\sigma$, the hermitian metrics $(\omega,h)$ on $Q$ with Aeppli class $\sigma$ are parametrized by pairs $(\xi,s)$, given by a real $1$-form $\xi$ on $X$ and a smooth section $s \in \Omega^0(\ad P_{h_0})$, where
\begin{equation}\label{eq:hermitianformula2}
h = e^{is}h_0, \qquad \omega = \omega_0 + 2(d\xi)^{1,1} + \tilde R(h,h_0).
\end{equation}
In particular, an infinitesimal change in the metric is given by
$$
\delta h = is, \qquad \delta \omega = 2(d \xi)^{1,1}  + 2c(s,F_{h_0}).
$$
Notice that the pair $(\xi,s)$ corresponds to a smooth global section of the kernel of the anchor map of the algebroid $\underline{Q}$ in Example \ref{ex:algebroidH}. This gives a geometric interpretation of the degrees of freedom of a hermitian metric on a Bott-Chern algebroid.

We address next the classification of Bott-Chern algebroids in Proposition \ref{prop:BCclassification}, which follows again by application of Lemma \ref{lem:CSRinvariant}.

\begin{proof}[Proof of Proposition \ref{prop:BCclassification}]

We fix $(\tau_0,h_0)$ as in Definition \ref{def:BCtype} with associated holomorphic string class $[(2i\partial \tau_0,\theta^{h_0})] \in \jmath^{-1}([P])$. If $(2i\partial \tau,\theta^{h})$ represents another holomorphic string class in $\jmath^{-1}([P])$ for a suitable $(\tau,h)$, we associate an element in the image of \eqref{eq:partialmap} given by
\begin{equation}\label{eq:map-classifying-Bott-Chern}
[2i\partial(\tau - \tau_0 - \tilde R(h,h_0))] \in H^1(\Omega^{\leqslant \bullet})/\Im \; \sigma_P.
\end{equation}
If $[(2i\partial \tau,\theta^{h})] = [(2i\partial \tau',\theta^{h'})]$ for another pair $(\tau',h')$, applying Theorem \ref{lemma:deRhamC} and Lemma \ref{lem:CSRinvariant} there exist $B,B' \in \Omega^{2,0}$ and $g \in \cG_P$ such that
\begin{align*}
2i\partial\tau' & = 2i \partial \tau + CS(\theta^{gh}) - CS(\theta^{h'}) - dc(\theta^{gh} \wedge \theta^{h'}) + dB\\
& = 2i \partial (\tau - \tilde R(gh,h')) + dB'
\end{align*}
where have used that $g \theta^h = \theta^{gh}$, since $g$ is holomorphic. Taking now classes in $H^1(\Omega^{\leqslant \bullet})$ it follows from Proposition \ref{prop:Donaldson}, Lemma \ref{lem:CSRinvariant}, and \eqref{eq:sigmaP-formula} that
$$
2i\partial(\tau' - \tau_0 - \tilde R(h',h_0)) = 2i\partial(\tau - \tau_0 - \tilde R(h,h_0)) + \sigma_P(g) \in H^1(\Omega^{\leqslant \bullet}),
$$
and therefore the map given by \eqref{eq:map-classifying-Bott-Chern} is well defined.

If the class associated to $(\tau,h)$ above is zero, Lemma \ref{lem:CSRinvariant} implies that
$$
2i\partial\tau = 2i \partial \tau_0 + 2i\partial \tilde R(h,h_0) - 2i\partial \tilde R(gh_0,h_0) + dB
$$
for $B \in \Omega^{2,0}$ and $g \in \cG_P$, and applying Proposition \ref{prop:Donaldson} and Lemma \ref{lem:CSRinvariant} again, we have
$$
2i\partial\tau = 2i \partial \tau_0 + CS(g\theta^{h_0}) - CS(\theta^{h}) - dc(g\theta^{h_0} \wedge \theta^{h}) + dB'
$$
for other $B' \in \Omega^{2,0}$, and therefore, again by  \eqref{eq:sigmaP-formula}, the map is injective. Surjectivity follows by twisting (see Definition \ref{def:twist}). The last part of the statement is a straightforward consequence of the $\partial \dbar$-Lemma.
\end{proof}

\section{Canonical metrics and variational approach}\label{sec:Aepplivariational}

\subsection{Canonical metrics}\label{subsec:canonical}

Let $Q$ be a positive Bott-Chern algebroid over a compact complex manifold $X$ of dimension $n$ with $c_1(X) = 0$. Following the previous section, we introduce now a notion of `best metric' in relation to the twisted Hull-Strominger system \eqref{eq:twistedStromhol} and, for a given positive Aeppli class, we use Lemma \ref{lem:tauhparamet} to reduce the system to an explicit PDE.

\begin{definition}\label{def:twistedStromQ}
We say that a tuple $(\Psi,\omega,h)$, given by a hermitian metric $(\omega,h)$ on $Q$ and a smooth section $\Psi$ of the canonical bundle $K_X$, satisfies the \it{twisted Hull-Strominger system on $Q$} if
\begin{equation}\label{eq:twistedStromholredux}
\begin{split}
      F_h \wedge \omega^{n-1}  & = 0,\\
      d \Psi - \theta_\omega \wedge \Psi & = 0,\\
      d \theta_\omega & = 0,\\
      \|\Psi\|_\omega & = 1.
    \end{split}
\end{equation}
\end{definition}

By Definition \ref{def:metricQ}, the existence of solutions of \eqref{eq:twistedStromholredux} only depends on the isomorphism class $[Q]$, and therefore it defines a natural system of equations for the Bott-Chern algebroid. From the point of view of the holomorphic principal bundle $P$ underlying $Q$, a solution of \eqref{eq:twistedStromholredux} 
provides  
a solution of the system
\begin{equation}\label{eq:twistedstromabstract}
\begin{split}
F_h\wedge \omega^{n-1} & = 0  ,\\
d \Psi - \theta_\omega \wedge \Psi & = 0,\\
d \theta_\omega & = 0,\\
dd^c \om - c(F_h\wedge F_h) & = 0,
\end{split}
\end{equation}
for an $\SU(n)$-structure $(\Psi,\omega)$ 
and a reduction $h$, and, therefore, any solution of \eqref{eq:twistedStromholredux} determines a solution of the twisted Hull-Strominger system \eqref{eq:twistedStromhol}. Conversely, by Proposition \ref{prop:stromholcour} any solution of \eqref{eq:twistedStromhol} determines a Bott-Chern algebroid $Q$ endowed with a hermitian metric and a smooth section $\Psi$ solving \eqref{eq:twistedStromholredux}. The reason we include the normalization $\|\Psi\|_\omega = 1$ (see \eqref{eq:normalizationnorm}) as part of the system \eqref{eq:twistedStromholredux} is that, as we will see next, due to our Definition \ref{def:metricQ} of hermitian metric, the condition $\|\Psi\|_\omega = 1$  is secretly a partial differential equation. 

Let $\sigma \in \Sigma_Q(\RR)$ be a positive Aeppli class in $Q$. We apply Lemma \ref{lem:tauhparamet} in order to transform \eqref{eq:twistedStromholredux} into an explicit PDE. Let us fix a hermitian metric $(\omega_0,h_0)$ on $Q$ with Aeppli class $\sigma$.
Relying on Lemma \ref{lem:tauhparamet}, finding a solution $(\Psi,\omega,h)$ of the twisted Hull-Strominger system on $Q$ with Aeppli class $\sigma$ is equivalent to find a tuple $(\Psi,\xi^{0,1},h)$, given by a smooth section $\Psi$ of $K_X$, a $(0,1)$-form $\xi^{0,1} \in \Omega^{0,1}$, and a reduction $h$ on $P$, such that $(\Psi,\omega,h)$ solves \eqref{eq:twistedStromholredux}, where
$$
\omega = \omega_0 + \tilde R(h,h_0) + \partial \xi^{0,1} + \overline{\partial \xi^{0,1}}
$$
is a positive hermitian form on $X$. More explicitly, we have the system of equations

\begin{align}\label{eq:twistedStromholredexp}
      F_h \wedge (\omega_0 + \tilde R(h,h_0) + \partial \xi^{0,1} + \overline{\partial \xi^{0,1}})^{n-1}  & = 0, \nonumber\\
      d \Psi - \theta_\omega \wedge \Psi & = 0,\\
      d \theta_\omega & = 0,  \nonumber\\
       d (\omega_0 + \tilde R(h,h_0) + \partial \xi^{0,1} + \overline{\partial \xi^{0,1}})^{n-1} & = \theta_\omega \wedge (\omega_0 + \tilde R(h,h_0) + \partial \xi^{0,1} + \overline{\partial \xi^{0,1}})^{n-1}, \nonumber \\
       (\omega_0 + \tilde R(h,h_0) + \partial \xi^{0,1} + \overline{\partial \xi^{0,1}})^n & = n!(-1)^{\frac{n(n-1)}{2}}i^n \Psi \wedge \overline{\Psi}. \nonumber
    \end{align}


With the equations \eqref{eq:twistedStromholredexp} at hand, let us take another look at the twisted Calabi-Yau equation \eqref{eq:twistedStromexact} on complex surfaces, from the point of view of holomorphic Courant algebroids. In this setup, $Q$ is a positive exact holomorphic Courant algebroid and the equations \eqref{eq:twistedStromholredux} reduce to
\begin{equation}\label{eq:twistedStromholreduxexact}
d \Psi - \theta_\omega \wedge \Psi = 0, \qquad d \theta_\omega = 0, \qquad \|\Psi\|_\omega = 1,
\end{equation}
for a pluriclosed hermitian metric $\omega$ such that $[2i\partial \omega] = [Q]$. 
Furthermore, fixing the Aeppli class on $Q$ is equivalent to fixing the Aeppli class $[\omega] \in H^{1,1}_A(X,\RR)$. In this framework, Theorem \ref{th:Hopf} is reinterpreted, by using Example \ref{example:HopfAeppli}, as the following result.

\begin{theorem}\label{th:HopfCourant}
Let $X$ be a compact complex surface with $c_1(X) = 0$, endowed with an exact Bott-Chern algebroid $Q$. If $(X,Q)$ admits a solution of \eqref{eq:twistedStromholreduxexact}, then there exists a unique solution $(\Psi,\omega)$ on each positive Aeppli class $\sigma \in \Sigma_Q(\RR)$, up to rescaling $\Psi$ by a unitary complex number. 
\end{theorem}

Theorem \ref{th:HopfCourant} motivates the following specialization of Question \ref{question} (cf. Question \ref{questionexact}).

\begin{question}\label{questioncourant}
Let $X$ be a compact complex manifold with $c_1(X) = 0$ endowed with a positive Bott-Chern algebroid $Q$ 
that admits a solution of \eqref{eq:twistedStromholredux}. Given a positive Aeppli class $\sigma \in \Sigma_Q(\RR)$, is there a unique solution of the twisted Hull-Strominger system \eqref{eq:twistedStromholredux} with Aeppli class $\sigma$, up to rescaling of $\Psi$? 
\end{question}

Given a holomorphic principal $G$-bundle $P$ over $X$ carrying a solution of the twisted Hull-Strominger system \eqref{eq:twistedStromhol}, we do not expect in general that any positive Bott-Chern algebroid $Q$ with underlying bundle $P$ admits a solution of \eqref{eq:twistedStromholredux} (even though this is the case in all the examples we have). Thus, Question \ref{questioncourant} seems to be the most sensible question to address regarding the uniqueness problem for the twisted Hull-Strominger system \eqref{eq:twistedStromhol}. The characterization of the pairs $(X,Q)$ carrying a solution of \eqref{eq:twistedStromholredux} seems to be a very subtle problem, which we speculate may be related to Geometric Invariant Theory (see Remark \ref{rem:Hopfmoduli} and Section \ref{subsec:Calabi}).


\subsection{Variational interpretation of the Hull-Strominger system}\label{subsec:variational}

Building on Proposition \ref{prop:stromholcour}, in this section we introduce new tools to address the existence and uniqueness problem for \eqref{eq:twistedStromhol} 
in the case $[\theta_\omega] = 0$, that is, for the
Hull-Strominger system \eqref{eq:Stromhol} (see Proposition \ref{prop:Strominger}). Motivated by Question \ref{questioncourant}, we draw a parallel with the Calabi problem for compact K\"ahler manifolds \cite{Yau0}, and study the critical points of a functional for hermitian metrics on a Bott-Chern algebroid with fixed positive Aeppli class. 
To establish a clear analogy with the classical situation, 
in this section we work in the generality of a compact complex manifold $X$ endowed with a smooth volume form and a positive 
Bott-Chern algebroid $Q$.


We fix a smooth volume form $\mu$ on $X$ compatible with the complex structure and, for any hermitian metric $\omega$ on $X$, we define a function $f_\omega$ by
\begin{equation}\label{eq:dilaton}
\frac{\omega^n}{n!} = e^{2f_\omega} \mu.
\end{equation}
We will call $f_\omega$ the \emph{dilaton function} of the hermitian metric $\omega$ with respect to $\mu$. Note that $e^{-2f_\omega}$ is the point-wise norm of $\mu$ with respect to $\omega$.

\begin{definition}\label{def:dilatonfunctional}
The \emph{dilaton functional} in the space of hermitian metrics $B_Q^+$ on $Q$ is defined by
\begin{equation}\label{eq:Dilfunctional}
M(\omega,h) = \int_X e^{-f_\omega} \frac{\omega^n}{n!}.
\end{equation}
\end{definition}

Our next goal is to study the restriction of the dilaton functional to a positive Aeppli class $\sigma \in \Sigma_Q(\RR)$, and to relate the critical points with the Hull-Strominger system \eqref{eq:Stromhol}. Denote by $B_\sigma^+ \subset B^+_Q$ the space of hermitian metrics with fixed Aeppli class $\sigma$ (see Proposition \ref{propo:Aepplicone}). The first variation of $M$ restricted to $B_\sigma^+$ is the content of our next result.

\begin{lemma}\label{lem:1stvariation}
If $(\delta \omega,\delta h)$ is an infinitesimal variation of $(\omega,h) \in B_\sigma^+$ with
$$
\delta \omega = -2ic(\delta h\, h^{-1},F_h) + \partial \xi^{0,1} + \overline{\partial \xi^{0,1}},
$$
then
\begin{align*}
\delta_{\omega,h}M & = \frac{1}{2(n-1)!}\int_X (-2 ic(\delta h\, h^{-1}, F_h) + \partial \xi^{0,1} + \overline{\partial \xi^{0,1}})\wedge e^{-f_\omega} \omega^{n-1}.
\end{align*}
\end{lemma}

\begin{proof}
Note that
$$
M(\omega,h) = \int_X e^{f_\omega} \mu
$$
and therefore
$$
\delta_{\omega,h}M = \int_X (\delta f_\omega) e^{f_\omega} \mu = \frac{1}{2}\int_X \Lambda_\omega(\delta \omega) e^{-f_\omega} \omega^n/n!,
$$
where we have used that $2\delta f_\omega = \Lambda_\omega(\delta \omega)$ by definition of $f_\omega$.
\end{proof}

The critical points of the functional \eqref{eq:Dilfunctional} on $B_\sigma^+$ are therefore given by the system of equations
\begin{equation}\label{eq:stromabstractredux}
  \begin{split}
    F_h\wedge \omega^{n-1} & = 0  ,\\
    d (e^{-f_\omega}\omega^{n-1}) & = 0,
  \end{split}
\end{equation}
for a hermitian metric $(\omega,h)$ on $Q$ with Aeppli class $\sigma$. Note that the second equation is equivalent to (see e.g. \cite{GF2})
\begin{equation}\label{eq:confbalanced}
\theta_\omega = - d f_\omega
\end{equation}
and therefore the existence of solutions of \eqref{eq:stromabstractredux} implies that $X$ is conformally balanced. For reasons that will be clear in Section \ref{subsec:Calabi}, we will call the equations \eqref{eq:stromabstractredux} the \emph{Calabi system}.

\begin{remark}
Similarly as it occurs for \eqref{eq:twistedStromholredux}, from the point of view of the principal bundle $P$ underlying $Q$, a solution of the Calabi system \eqref{eq:stromabstractredux} corresponds to a solution of
\begin{equation}\label{eq:stromabstract}
\begin{split}
F_h\wedge \omega^{n-1} & = 0  ,\\
d (e^{-f_\omega}\omega^{n-1}) & = 0,\\
dd^c \om - c(F_h\wedge F_h) & = 0.
\end{split}
\end{equation}
\end{remark}

If we fix a reference metric $(\omega_0,h_0) \in B_\sigma^+$, using Lemma \ref{lem:tauhparamet} the system \eqref{eq:stromabstractredux} reduces to the following explicit partial differential equation
\begin{equation}\label{eq:stromabstractexp}
\begin{split}
      F_h \wedge (\omega_0 + \tilde R(h,h_0) + \partial \xi^{0,1} + \overline{\partial \xi^{0,1}})^{n-1}  & = 0,\\
       d\Big{(}e^{-f_\omega} (\omega_0 + \tilde R(h,h_0) + \partial \xi^{0,1} + \overline{\partial \xi^{0,1}})^{n-1}\Big{)} & = 0
\end{split}
\end{equation}
for a pair $(\xi^{0,1},h)$, where $\xi^{0,1} \in \Omega^{0,1}$ is a $(0,1)$-form such that
$$
\omega = \omega_0 + \tilde R(h,h_0) + \partial \xi^{0,1} + \overline{\partial \xi^{0,1}}
$$
is a positive hermitian form on $X$, and
$$
(\omega_0 + \tilde R(h,h_0) + \partial \xi^{0,1} + \overline{\partial \xi^{0,1}})^n = n!e^{2f_\omega} \mu.
$$

Assuming that $X$ admits a holomorphic volume form, we establish next the relation between the functional $M$ and the Hull-Strominger system \eqref{eq:Stromhol}. Let $\Omega$ be a holomorphic volume form on $X$ and set
\begin{equation}\label{eq:muOmega}
\mu = (-1)^{\frac{n(n-1)}{2}}i^n\Omega \wedge \overline{\Omega}.
\end{equation}

\begin{proposition}\label{prop:relationStr}
Let $Q$ be a positive Bott-Chern algebroid over a compact Calabi-Yau manifold $(X,\Omega)$, with fixed Aeppli class $\sigma = [(\omega_0,h_0)] \in \Sigma_Q(\RR)$. If $(\omega,h) \in B_\sigma^+$ is critical point of $M$ with $\mu$ as in \eqref{eq:muOmega}, then $(\om,h)$ is a solution of the Hull-Strominger system \eqref{eq:Stromhol} with associated Bott-Chern algebroid $Q$. Conversely, if $(\omega,h)$ is a solution of \eqref{eq:Stromhol} with Bott-Chern algebroid $Q$ 
such that $[(\omega,h)] = \sigma$, then $(\omega,h)$ is a critical point of $M$ on $B_\sigma^+$ and $(\omega,h)$ and $(\omega_0,h_0)$ are related by \eqref{eq:hermitianformula}.
\end{proposition}

\begin{proof}
By \eqref{eq:Leeformdef0} and \eqref{eq:confbalanced}, the `if part' of the proof reduces to the identity $\|\Omega\|_\omega = e^{-f_\omega}$. The `only if part' follows from Lemma \ref{lem:tauhparamet} .
\end{proof}

Our next result is the calculation of the second variation of the functional $M$ restricted $B_\sigma^+$. Let $(\omega_t,h_t) \in B_\sigma^+$ be a one-parameter family of hermitian metrics on $Q$, with
\begin{equation}\label{eq:omegat}
\omega_t  = \omega + \tilde R(h_t,h_0) + \partial \xi_t^{0,1} + \overline{\partial \xi_t^{0,1}}
\end{equation}
for $\xi_t^{0,1} \in \Omega^{0,1}$. Using the Lefschetz decomposition with respect to $\omega_t$, we can write
$$
\dot \omega_t = \sigma_t + \beta_t \frac{\omega_t}{n},
$$
where $\sigma_t$ is a primitive $(1,1)$-form and $\beta_t \in C^{\infty}(X)$ is given by
$$
\beta_t = \Lambda_{\omega_t}\(-2 ic(\dot h_t h^{-1}_t, F_{h_t}) + \partial \dot \xi_t^{0,1} + \overline{\partial \dot \xi_t^{0,1}}\).
$$

\begin{lemma}\label{lem:2ndvariation} The following identity holds:
\begin{align*}
\frac{d}{dt^2}M(\omega_t,h_t)  & = - \frac{1}{2}\int_X e^{-f_{\omega_t}} |\sigma_t|^2_{\omega_t} \frac{\omega_t^{n}}{n!}\\
& \phantom{= {} } + \frac{1}{2}\int_X e^{-f_{\omega_t}} \(\Lambda_{\omega_t}\(\partial \ddot \xi_t^{0,1} + \overline{\partial \ddot \xi_t^{0,1}}\) + \frac{n-2}{2n}\beta_t^2\) \frac{\omega_t^{n}}{n!}\\
& \phantom{= {} } + \frac{1}{2}\int_X e^{-f_{\omega_t}} \Lambda_{\omega_t} \(4ic(\dot h_t h_t^{-1},\dbar \partial^{h_t}(\dot h_t h_t^{-1})) - 2 ic(\partial_t (\dot h_t h_t^{-1}), F_{h_t})\)  \frac{\omega_t^{n}}{n!}.
\end{align*}
\end{lemma}

\begin{proof}
We denote $\partial_t^k = \frac{d}{dt^k}$. By Lemma \ref{lem:1stvariation} we have
\begin{align*}
\partial_t^2 M(\omega_t,h_t) & = \frac{1}{2(n-1)!}\int_X \ddot \omega_t \wedge e^{-f_{\omega_t}} \omega_t^{n-1} + \dot \omega_t \wedge \partial_t (e^{-f_{\omega_t}} \omega_t^{n-1})
\end{align*}
Now, by \eqref{eq:omegat},
\begin{align*}
\ddot \omega_t \wedge \omega_t^{n-1} & = 
\Lambda_{\omega_t} \(4ic(\dot h_t h_t^{-1},\dbar \partial^{h_t}(\dot h_t h_t^{-1})) - 2 ic(\partial_t (\dot h_t h_t^{-1}), F_{h_t}) + \partial \ddot \xi_t^{0,1} + \overline{\partial \ddot \xi_t^{0,1}}\)\frac{\omega_t^n}{n}
\end{align*}
and also
\begin{align*}
\dot \omega_t \wedge \partial_t (e^{-f_{\omega_t}} \omega_t^{n-1}) & = e^{-f_{\omega_t}}\dot \omega_t \wedge (-(\Lambda_{\omega_t} \dot \omega_t/2)\omega_t + (n-1)\dot \omega_t) \wedge \omega_t^{n-2}\\
& = \(- |\sigma_t|^2_{\omega_t} + \frac{n-2}{2n}|\Lambda_{\omega_t} \dot \omega_t|^2\)e^{-f_{\omega_t}}\frac{\omega_t^n}{n}.
\end{align*}
\end{proof}

Our formula for the second variation points towards special paths on $B_\sigma^+$ along which the functional $M$ is concave. Remarkably, these paths are independent of the choice of volume form in the definition of the functional \eqref{eq:Dilfunctional}.

\begin{corollary}\label{cor:concave}
The dilaton functional is concave along paths $(\omega_t,h_t) \in B_\sigma^+$ solving
\begin{align}\label{eq:concavepath}
\Lambda_{\omega_t}(\partial \ddot \xi_t^{0,1} + \overline{\partial \ddot \xi_t^{0,1}})  = {} & \frac{2-n}{2n}|\Lambda_{\omega_t}(-2ic(\dot h_t h_t^{-1}, F_{h_t}) + \partial \dot \xi_t^{0,1} + \overline{\partial \dot \xi_t^{0,1}})|^2\\
&  - \Lambda_{\omega_t} \(4ic(\dot h_t h_t^{-1},\dbar \partial^{h_t}(\dot h_t h_t^{-1})) - 2 ic(\partial_t (\dot h_t h_t^{-1}), F_{h_t})\). \nonumber
\end{align}

\end{corollary}

In general, the only source of concavity of $M$ is given by the primitive part of the variation of $\omega_t$, which motivates \eqref{eq:concavepath}. Note here that we cannot use the convexity properties of Donaldson's functional in the space of reductions of $P$ \cite{Don}, since the bilinear form $c$ on $\mathfrak{k} \subset \mathfrak{g}$ may have arbitrary signature.

Equation \eqref{eq:concavepath} is reminiscent of the geodesic equation in the space of K\"ahler metrics in a fixed K\"ahler class, which plays an important role in the constant scalar curvature problem in K\"ahler geometry (see e.g. \cite{D6}).
To see this, by setting $\xi_t^{0,1} = i\dbar \phi_t$ and $h_t = e^{tu}h_0$, we note that \eqref{eq:concavepath} reduces to the following fourth-order partial differential equations for the one-parameter family $\phi_t$ of smooth functions:
\begin{equation}\label{eq:concavepath-geodlike}
P_t \ddot \phi_t = \frac{n-2}{2n}|P_t \dot \phi_t + 2 ic(u,\Lambda_{\omega_t} F_{h_t})|^2 +  4ic(u,\Lambda_{\omega_t} \dbar \partial^{h_t}u),
\end{equation}
where $P_t$ is the second-order elliptic differential operator
$$
P_t = \Lambda_{\omega_t} 2i\dbar \partial \colon C^\infty(X) \to C^\infty(X).
$$
Note that $P_t$ coincides with the Hodge laplacian $\Delta_{\omega_t}$ precisely when $\omega_t$ is balanced \cite{Gau77}.

\begin{remark}\label{rem:concavepathsurface}
The equations \eqref{eq:concavepath} are very special on a complex surface, as in this case the first term on the right-hand side vanishes. With the ansatz $h_t = e^{tu}h_0$, they further reduce to
\begin{equation}\label{eq:concavepathsurface}
\Lambda_{\omega_t}(\partial \ddot \xi_t^{0,1} + \overline{\partial \ddot \xi_t^{0,1}}) = - 4 ic(u,\Lambda_{\omega_t}\dbar \partial^{h_t}u) = - \Lambda_{\omega_t}\frac{d}{dt^2} \tilde R(h_t,h_0).
\end{equation}
When $G = \{1\}$, the right hand side of equation \eqref{eq:concavepathsurface} vanishes and therefore it can be solved using linear paths in the Aeppli class.
\end{remark}

Similarly as in the case of K\"ahler geometry, our next result shows that the equation \eqref{eq:concavepath-geodlike} can be potentially used as an approach to the uniqueness problem for the Calabi system \eqref{eq:stromabstractredux}. We will use the natural action of the group of automorphisms of $Q$ on the space of hermitian metrics defined in Lemma \ref{lem:action}.

\begin{proposition}\label{prop:uniqueness}
If $(\omega_0,h_0)$ and $(\omega_1,h_1)$ are two solutions of the Calabi system \eqref{eq:stromabstractredux} with Aeppli class $\sigma \in \Sigma_Q(\RR)$ that can be joined by a smooth solution $(\omega_t,h_t)$ of \eqref{eq:concavepath}, then $\omega_1 = k \omega_0$ for some constant $k$, and $h_1$ is related to $h_0$ by an element in the holomorphic gauge group $\cG_P$ of $P$. Furthermore, when $d\omega_0\neq 0$, we must have $k=1$ and, consequently, $(\omega_0,h_0)$ and $(\omega_1,h_1)$ are related by an automorphism of $Q$. 
\end{proposition}

\begin{proof}
By Lemma \ref{lem:2ndvariation}, the second derivative of $M$ along the path $(\omega_t,h_t)$ is non-positive. Thus, since $(\omega_0,h_0)$ and $(\omega_1,h_1)$ are critical points of $M$, we have that $\sigma_t = 0$ for all $t \in [0,1]$ and therefore
\begin{equation}\label{eq:lemmauniqueness}
\dot \omega_t = \beta_t \omega_t/n.
\end{equation}
We claim that this implies that $\omega_0$ is conformal to $\omega_1$. To see this, consider the decomposition of $\omega_t$ into primitive and non-primitive parts with respect to $\omega_0$
\begin{equation}\label{eq:lemmauniqueness2}
\omega_t = \gamma_t + f_t \omega_0/n,
\end{equation}
where $\gamma_t$ and $f_t$ depend smoothly on $t$. Using  \eqref{eq:lemmauniqueness} and the fact that $\dot \gamma_t$ is primitive with respect to $\omega_0$ for all $t$, it follows that $\gamma_t$ satisfies the differential equation
$$
\dot \gamma_t =  \beta_t \gamma_t
$$
with initial condition $\gamma_0 = 0$. By uniqueness of ODE, we conclude that $\gamma_t = 0$ for all $t$ and hence $\omega_1 = f_1 \omega_0/n$.

Using now that $\omega_1 = e^f \omega_0$ for a smooth function $f$, the Buchdahl-Li-Yau Theorem \cite{Buchdahl,LiYauHYM} implies that $h_0 = g h_1$ for some $g \in \cG_P$, and as a consequence
$$
dd^c (\omega_1 - \omega_0) = c(F_{h_1} \wedge F_{h_1}) - c(F_{h_0}\wedge F_{h_0}) = 0.
$$
Finally, this implies that
$$
dd^c ((1-e^f)\omega_0) \wedge \omega^{n-2}_0 = 0,
$$
and therefore applying the maximum principle to the linear elliptic  operator $u \mapsto dd^c (u\omega_0) \wedge \omega^{n-2}_0$ it follows that $\omega_1 = k \omega_0$ for some constant $k$.

For the final claim, note that $dd^c\omega_0 \neq 0$ when $d\omega_0\neq 0$ (since $\omega_0$ is conformally balanced by assumption), so it follows from $dd^c\omega_0 = dd^c\omega_1$ that $k =1$. Using now that $(\omega_0,gh_1)$ and $(\omega_0,h_1)$ define the same holomorphic string class, we obtain (see Theorem \ref{lemma:deRhamC})
$$
CS(g \theta^{h_1}) - CS(\theta^{h_1}) - dc(g \theta^{h_1} \wedge \theta^{h_1}) = dB
$$
for some $B \in \Omega^{2,0}$. Consequently, $g \in \Ker \sigma_P$ and hence the short exact sequence \eqref{eq:AutW} implies that there exists $(\varphi,g) \in \Aut Q$ covering $g$ (see \eqref{eq:defstringiso}). The statement follows now from Lemma \ref{lem:action}, which gives
$$
(\varphi,g) \cdot (\omega_1,h_1) = (\omega_1,gh_1) = (\omega_0,h_0).
$$
\end{proof}

\begin{remark}\label{rem:uniqueKahlercase}
In the setup of Proposition \ref{prop:uniqueness}, assume that $d\omega_0 = 0$. Then, using that both solutions are in the same Aeppli class, it follows that
$$
(k-1)[\omega_0] = \widetilde{\sigma}_P(g) \in H^{1,1}_A(X,\mathbb{R})
$$
where $g \in \cG_P$ is such that $h_0 = g h_1$ (see Lemma \ref{lemma:tildesigmaP}). Provided that $\mathcal{G}_P$ is connected, it follows from Proposition \ref{prop:Futaki} that $\widetilde{\sigma}_P(g) \cdot [\omega^{n-1}_0] = 0$ and hence $k=1$ and the two solutions are related by an automorphism of $Q$. The hypothesis that $\mathcal{G}_P$ is connected holds, for example, when $G$ is a product of general linear groups.
\end{remark}

Proposition \ref{prop:uniqueness} suggests the following parallel with the theory of constant scalar curvature K\"ahler metrics \cite{D6}: 
the existence of enough `weak solutions' of \eqref{eq:concavepath} should imply uniqueness of solutions of the Calabi system \eqref{eq:stromabstractredux} with Aeppli class $\sigma$ modulo automorphisms of $Q$. Furthermore, provided that $B_\sigma$ contains a solution of the system, the dilaton functional \eqref{eq:Dilfunctional} restricted to $B_\sigma$ should be bounded from above. Thus, one can potentially try to define an obstruction to the existence of solutions of the Calabi system using the asymptotics of $dM$ along infinite paths on $B_\sigma$ solving \eqref{eq:concavepath}. The resulting `stability condition' would be for a pair $(Q,\sigma)$ given by a Bott-Chern algebroid and a positive Aeppli class, and should in particular imply that the underlying complex manifold is balanced. 
It is therefore natural to expect that the existence problem for the PDE \eqref{eq:concavepath} is an important tool in the theory for the Hull-Strominger system.


\subsection{Relation with the Calabi problem}\label{subsec:Calabi}

Expanding upon the method of Proposition \ref{prop:uniqueness}, in this section we give evidence of the previous conjectural picture by studying the simple case of exact Courant algebroids over compact complex surfaces. Consider the functional \eqref{eq:Dilfunctional} in the case $G = \{1\}$, that is, for an exact positive Bott-Chern algebroid
$$
0 \to T^*X \to Q \to TX \to 0,
$$
defined by a closed $H \in \Omega^{2,1}$. By assumption, there is a positive hermitian form $\omega_0 \in \Omega^{1,1}$ 
such that
$$
H = 2i\partial \omega_0 + dB,
$$
for $B \in \Omega^{2,0}$. Therefore, the space of real Aeppli classes on $Q$ is identified with $\Ker \partial \subset H^{1,1}_A(X,\RR)$ in \eqref{eq:partialmap} via the map
\begin{align*}
Ap(\underline{\;\;\;\;},\omega_0) \colon  B_Q &\to H^{1,1}_A(X) \\
 \tau &\mapsto [\tau - \omega_0].
\end{align*}

Given a positive Aeppli class $\sigma \in \Sigma_Q(\RR)$, which we assume e.g. to be $ \sigma = [\omega_0]$, the critical points of the functional \eqref{eq:Dilfunctional} on $B_\sigma$ are given by a hermitian form $\omega$ on $X$ 
such that
\begin{equation}\label{eq:stromabstractexact}
  \begin{split}
    d (e^{-f_\omega}\omega^{n-1}) & = 0,\\
    dd^c \om & = 0,
  \end{split}
\end{equation}
and satisfying 
$[\omega] = [\omega_0] \in H^{1,1}_A(X,\RR)$. Now, a hermitian metric which is pluriclosed and conformally balanced is necessarily K\"ahler \cite[Thm. 1.3]{IvanovPapado}, and we have the following result.

\begin{proposition}\label{prop:extremaexactcase}
Let $Q$ be an exact positive Bott-Chern algebroid. If 
$\omega \in B_\sigma$ is a critical point of \eqref{eq:Dilfunctional}, then $d\omega = 0$ and $df_\omega = 0$. Therefore, $X$ is K\"ahler, $Q$ is isomorphic to $TX \oplus T^*X$ and there exists a constant $\ell > 0$ such that
\begin{equation}\label{eq:volume}
\frac{\omega^n}{n!} = \ell \mu.
\end{equation}
\end{proposition}
On a K\"ahler manifold $X$ there is a natural isomorphism $H^{1,1}_A(X)  \cong H^{1,1}(X)$, which identifies the Aeppli cone with the K\"ahler cone in $H^2(X,\RR)$. Hence, by Proposition \ref{prop:extremaexactcase}, the existence problem for \eqref{eq:stromabstractexact} reduces to the Calabi problem, that is, to prescribe the volume of a K\"ahler metric in a fixed K\"ahler class. As in the classical problem of Calabi, the most interesting situation arises when we consider a Calabi-Yau manifold $(X,\Omega)$ and choose $\mu$ as in \eqref{eq:muOmega}, as in this case \eqref{eq:volume} implies that the holonomy of the K\"ahler metric is contained in $\SU(n)$.


The following result follows from Corollary \ref{cor:concave} and the special features of the path \eqref{eq:concavepath} in complex dimension two (see Remark \ref{rem:concavepathsurface}).

\begin{theorem}\label{th:bounded}
Let $X$ be a compact complex surface endowed with an exact Bott-Chern algebroid $Q$. If a positive Aeppli class $\sigma \in \Sigma_Q(\RR)$ admits a solution $\omega_0$ of \eqref{eq:stromabstractexact} then it is unique. Furthermore, in this case the dilaton functional \eqref{eq:Dilfunctional} is bounded from above on $B_\sigma$.
\end{theorem}
\begin{proof}
Let $\omega \in B_\sigma$ be another solution of \eqref{eq:stromabstractexact}. Using $[\omega] = [\omega_0] = \sigma$, we have
$$
\omega = \omega_0 + (d\xi)^{1,1}
$$
for some $\xi \in \Omega^1$. Taking the solution of \eqref{eq:concavepath} given by $\omega_t = \omega_0 + t (d\xi)^{1,1}$ (see Remark \ref{rem:concavepathsurface}), the uniqueness part of the statement follows from Proposition \ref{prop:uniqueness}. Furthermore, by hypothesis the functional $M$ is convave along straight segments on $B_\sigma$ and therefore it is concave. Boundedness now follows from the fact that any critical point of a concave function defined on a connected space is a global maximum.


\end{proof}

We expect that the boundedness of the dilaton functional is closely related to the existence of solutions of the Calabi system on arbitrary dimensions, and for arbitrary Bott-Chern algebroid $Q$. It is interesting to notice that, when $Q$ is exact, the dilaton functional can be formulated on arbitrary SKT manifolds and that the only obstruction to the existence of critical points is the failure of $X$ to be K\"ahler.

To give evidence in this direction, we consider the example of a primary Hopf surface $X$ endowed with an exact Courant algebroid $Q$ (see Example \ref{example:HopfQBC}). $X$ being non-K\"ahler, $Q$ does not admit critical points of the functional \eqref{eq:Dilfunctional} for any choice of volume form $\mu$. Nonetheless, as we will see next, the asymptotics of the dilaton functional carry interesting information. 
For simplicity, we restrict to the case of a homogeneous Hopf surface
$$
X = X_{\alpha,\alpha} =  (\CC^2 \backslash \{0\})/\langle \gamma \rangle
$$
with $\alpha = \beta$ (see \eqref{eq:cyclicgroup}), via the identification
$$
S^3 \times S^1 \cong \SU(2) \times \U(1).
$$
By a result of Sasaki \cite{Sasaki}, the homogeneous integrable complex structures on $\SU(2) \times \U(1)$ are parametrized by
$$
w = x + i y \in \mathbb{C},
$$
and any such complex structure determines a diagonal Hopf surface $X_w = X_{e^w,e^w}$ induced by a holomorphic map \cite{Hasegawa}
\begin{align*}
\Phi_w \colon \mathbb{R} \times \SU(2) &\to \CC^2 \backslash \{0\} \\
(t,z_1,z_2) &\mapsto (e^{tw}z_1,e^{tw}z_2).
\end{align*}

Here, we identify $\SU(2)$ with $S^3 = \{(z_1,z_2) \in \CC^2 \; | \; |z_1|^2 + |z_2|^2 = 1\}$ by the correspondence
\begin{equation*}\label{eq:DolbQ}
\left(
\begin{array}{ccc}
z_1 & - \overline{z}_2 \\
z_2 & \overline{z}_1
\end{array}\right) \leftrightarrow (z_1,z_2).
\end{equation*}
More explicitly, consider generators for the Lie algebra
$$
\mathfrak{su}(2) \oplus \RR = \langle e_1, e_2, e_3, e_4 \rangle, \qquad
$$
such that
$$
de^1 = e^{23}, \quad de^2 = e^{31}, \quad de^3 = e^{12}, \quad de^4 = 0,
$$
for $\{e^j\}$ the dual basis, satisfying $e^j(e_k) = \delta_{jk}$, and the notation $e^{ij}=e^i\wedge e^j$ and similarly. Then, for $x \neq 0$, the complex structure corresponding to $w = x + i y$ is
$$
J_w (e_4 - y e_1) = x e_1, \qquad J_w e_2 = e_3,
$$
which in terms of the dual basis reads
$$
J_w (xe^4) = e^1 + y e^4, \qquad J_w e^2 = e^3.
$$
We define a basis of $(1,0)$-forms for $J_w$ varying analytically on $w$ by
$$
\eta^1_w = ie^1 + we^4, \qquad \eta^2 = e^2 + ie^3,
$$
and consider the $(2,0)$-form
$$
\Psi_w = \eta^1_w \wedge \eta^2.
$$
Note that
$$
d\Psi_w = -z e^4 \wedge \Psi_w,
$$
and therefore $J_w$ is integrable, since $de^4 =0$.

To fix an exact Courant algebroid over $X_w$, we use the identification \cite{AngellaBC}
\begin{equation*}\label{eq:11AeppliHopf}
H^{1,1}_A(X_w) \cong \CC \langle e^{41} \rangle,
\end{equation*}
and that any Aeppli class $a[e^{41}] \in H^{1,1}_A(X_w)$ determines an exact Courant algebroid $Q_{w,a}$ (see Example \ref{example:HopfQBC}) by
$$
[Q_{w,a}] = a[2i\partial_w(e^{41})] = \frac{a}{2x} [\eta^1_w \wedge \eta^2 \wedge \overline{\eta}^2] \in H^{2,1}_{\dbar}(X_w).
$$
The algebroid $Q_{w,a}$ is Bott-Chern when $a \in \RR$ and, for $x > 0$, it is positive precisely when $a > 0$. In the sequel we assume that $x > 0$ and that $a$ is real and positive. In this setup, a solution of the equation \eqref{eq:concavepath-geodlike} on the (unique) Aeppli class in $Q_{w,a}$ is given by (see Example \ref{example:HopfAeppli} and Remark \ref{rem:concavepathsurface})
$$
\omega_t = a e^{41} + t e^{23}, \textrm{ for } t \in (0,+\infty).
$$
In order to evaluate the functional \eqref{eq:Dilfunctional} along $\omega_t$, we fix a volume form on $X_w$
by
$$
\mu_w = \Psi_w \wedge \overline{\Psi}_w = 4x e^{4123}.
$$
The corresponding one-parameter family of dilaton functions is given by
$$
f_{t} = \frac{1}{2}\log(at/4x)
$$
and therefore the functional \eqref{eq:Dilfunctional} along $\omega_t$ is
$$
M(t) = \int_{X_w} e^{f_{\omega_t}}\mu_w = (2axt)^{\frac{1}{2}} V,
$$
where $V$ is the volume of $S^3 \times S^1$ with respect to the invariant volume element $e^{4123}$. We observe that $M(t)$ is concave, and it is not bounded from above.

We want to give now a cohomological interpretation of the asymptotics of $M(t)$ as $t \to +\infty$. 
For this, it is convenient to choose a different normalization of the parameter $t$. Note that
$$
\theta_{\omega_t} = - \frac{a}{t} e^4.
$$
and therefore when $y = 0$ and $t = a/x$, the pair $(\Psi_w,\omega_{a/x})$ corresponds to the unique solution of the twisted Calabi-Yau equation \eqref{eq:twistedStromexact} in Aeppli class $a$ obtained in Theorem \ref{th:Hopf}. Then, setting
$$
l = \frac{x}{a}t
$$
and
\begin{align*}
[H] & := [Q_{w,a}] + \overline{[Q_{w,a}]} = -\frac{a}{x}[e^{123}] \in H^3(S^3 \times S^1,\RR),\\
\ell_{w,a} & := [\theta_{\omega_{a/x}}] = -x[e^4] \in H^1(S^3 \times S^1,\RR).
\end{align*}
we have
$$
M(l) = (\ell_{w,a} \cdot [H]) l^{\frac{1}{2}},
$$
where $\cdot$ denotes the intersection product in cohomology. The degree three cohomology class $[H]$ corresponds to the \v Severa class of the smooth (real) exact Courant algebroid associated to $Q_{w,a}$ \cite{G2}. We observe that the two obstructions to the existence of critical points of $M$, given by $[Q_{w,a}]$ and the Lee form class $\ell_{w,a}$ , appear in the asymptotics. Furthermore, $\ell_{w,a}$ seems to be the way that the functional $M$ measures how far is $X_w$ from being K\"ahler.

\begin{remark}
It would be interesting to study the ODE corresponding to \eqref{eq:concavepath} in higher dimensions in a homogeneous setup. For the case of exact Courant algebroids over SKT manifolds, \eqref{eq:concavepath} reduces to
\begin{equation}\label{eq:concavepathexact}
\begin{split}
\Lambda_{\omega_t}(\partial \ddot \xi_t^{0,1} + \overline{\partial \ddot \xi_t^{0,1}}) & = \frac{2-n}{2n}|\Lambda_{\omega_t}(\partial \dot \xi_t^{0,1} + \overline{\partial \dot \xi_t^{0,1}})|^2.
\end{split}
\end{equation}

\end{remark}



\section{Linear theory and deformations}\label{sec:linear}

\subsection{Fredholm alternative}\label{sec:Fredholm differential}

In this section we study the linear theory for the twisted Hull-Strominger system on a Bott-Chern algebroid $Q$ \eqref{eq:twistedStromholredux}, showing that the linearization of the equations restricted to an Aeppli class \eqref{eq:twistedStromholredexp} induces a Fredholm operator. For the case $[\theta_\omega] = 0$, we further prove that this operator has index zero and provide a Fredholm alternative: either it has a non-trivial finite-dimensional kernel, or it is invertible.


Let $Q$ be a positive Bott-Chern algebroid over a compact complex manifold $X$ of dimension $n$ with $c_1(X) = 0$. Let $P$ be the holomorphic principal $G$-bundle underlying $Q$. We assume that $Q$ admits a solution $(\Psi,\omega,h)$ of the twisted Hull-Strominger system on $Q$ \eqref{eq:twistedStromholredux}.  If $(\tilde \Psi,\tilde \omega, \tilde h)$ is another tuple as in Definition \ref{def:twistedStromQ}, with $\tilde \Psi = e^\phi \Psi$ for some smooth function $\phi$, the condition of being a solution of \eqref{eq:twistedStromholredux} is equivalent to
\begin{equation}\label{eq:twistedStromholredux2}
\begin{split}
   F_{\tilde h}\wedge \tilde \omega^{n-1} & = 0  ,\\
      d\phi - \theta_{\tilde \omega} + \theta_\omega & = 0,\\
      \|\Psi\|_{\tilde \omega} & = e^{-\phi}.
  \end{split}
\end{equation}
By \eqref{eq:Leeformdef}, the second equation can be written alternatively as
$$
d \tilde \omega^{n-1} = (\theta_\omega + d\phi) \wedge \tilde \omega^{n-1}
$$
and therefore the linearization of \eqref{eq:twistedStromholredux2} at $(\Psi,\omega,h)$ is
\begin{equation}\label{eq:twistedStromholreduxlinear}
\begin{split}
   -2 \dbar \partial^h (\delta h\, h^{-1}) \wedge \omega^{n-1} + (n-1) F_h \wedge \delta \omega \wedge \omega^{n-2} & = 0  ,\\
      (d - \theta_\omega)\Big{(}(n-1)  \delta \omega \wedge \omega^{n-2} - \delta \phi \omega^{n-1}\Big{)} & = 0,\\
     \Lambda_\omega(\delta \omega) - 2\delta \phi & = 0,
  \end{split}
\end{equation}
for $\delta \omega \in \Omega^{1,1}$, $\delta h\, h^{-1} \in \Omega^{0}(i\ad P_h)$ and $\delta \phi \in C^\infty(X)$, where $P_h$ denotes the principal $K$-bundle corresponding to the reduction $h$ and
$$
(d - \theta_\omega)\alpha = d \alpha - \theta_\omega \wedge \alpha, \qquad \textrm{ for } \alpha \in \Omega^k.
$$
As expected from Lemma \ref{lemma:holSUn}, we observe that the parameter $\delta \phi$ is redundant.

Being variations of a hermitian metric on $Q$, the parameters $\delta \omega$ and $\delta h$ are related by the linearization of \eqref{eq:anomaly} (see Definition \ref{def:metricQ}). 
We want to study now the linearization \eqref{eq:twistedStromholreduxlinear} when $(\omega,h)$ varies in the Aeppli class $\sigma \in \Sigma_Q^+$ of the fixed solution, that is, for
$$
\delta \omega = -2ic(\delta h\, h^{-1},F_h) + 2\partial \xi^{0,1} + 2\overline{\partial \xi^{0,1}},
$$
where $\xi^{0,1} \in \Omega^{0,1}$ (see Lemma \ref{lem:tauhparamet}). It will be useful to use a real parametrization of $\delta \omega$. For this, notice that $d \xi + J d\xi = 2 (d \xi)^{1,1} = 2 (\partial \xi^{0,1} + \overline{\partial \xi^{0,1}})$ for any real $\xi \in \Omega^1$, and therefore from now on we write
$$
\delta \omega = d \xi + J d \xi - 2 ic(\delta h\, h^{-1},F_h).
$$
Motivated by the previous discussion, we define the operator
\begin{equation}\label{eq:lin full parameters}
 \begin{array}{cccc}
  \cL : &\Om^1 \times \Om^0(\ad P_h)  & \rightarrow & \Om^{2n-1}\times \Om^{2n}(\ad P_h) \\
         & ( \xi, s ) & \mapsto &  (\cL_1(\xi, s), \cL_2(\xi, s))\\
         \end{array}
 \end{equation}
  with
\begin{equation*}
  \begin{split}
  \cL_1(\xi,s)\: = &\: (d-\theta_\omega) \left( \left( T(d\xi + Jd \xi) +2 c(s, F_h) \right) \wedge \om^{n-2}\right),\\
  \cL_2(\xi, s) \: = &\: -2 i \dbar \partial^h(s)\wedge \om^{n-1} +(n-1)  F_h \wedge (d\xi +Jd\xi + 2 c(s,F_h) ) \wedge \om^{n-2} ,
 \end{split}
\end{equation*}
and where, for any $\alpha\in \Om^2$, we set
\begin{equation*}
T(\alpha)=\alpha  -\frac{1}{2(n-1)} \Lambda_\om(\alpha) \om.
\end{equation*}

Consider the complex of differential operators
\begin{equation}
 \label{eq:complex L}
 \Om^0 \lra{\iota_1 \circ d } \Omega^1 \times \Omega^0(\ad P_h) \lra{\cL} \Om^{2n-1} \times \Omega^{2n}(\ad P_h) \lra{(d-\theta_\omega) \circ p_1} \Om^{2n},
\end{equation}
where $\iota_1 \colon \Omega^1 \to \Omega^1 \times \Omega^0(\ad P_h)$ and $p_1 \colon \Om^{2n-1} \times \Omega^0(\ad P_h) \to \Om^{2n-1}$ denote the inclusion and the projection, respectively.

\begin{lemma}\label{lem:Lelliptic}
The complex (\ref{eq:complex L}) is elliptic.
\end{lemma}

For the proof, we decompose
\begin{equation}
\label{eq:L=UK}
\cL=\cU + \cK,
\end{equation}
with
\begin{equation*}
 \cU,\,\cK :  \Om^1\times \Om^0(\ad P_h)  \rightarrow   \Om^{2n-1} \times\Om^{2n}(\ad P_h),
\end{equation*}
where $\cU = \cU_1 \times \cU_2$ for
\begin{equation}
\label{eq:lin elliptic term}
\begin{split}
\cU_1 (\xi,s) & = (d-\theta_\omega)\left( \left(T(d\xi + J d \xi)\right) \wedge \om^{n-2}\right),\\
 \cU_2 (\xi,s) & =  -2 i \dbar \partial^h(s)\wedge \om^{n-1},
\end{split}
\end{equation}
and
\begin{equation*}
 \label{eq:lin compact term}
\cK(\xi, \delta h )    = ((d-\theta_\om)\left(2c(s, F_h) \wedge \om^{n-2}\right), \: (n-1)  F_h \wedge (d\xi +Jd\xi + 2c(s,F_h) ) \wedge \om^{n-2}  ) .
\end{equation*}
Note that while $\cU$ is of order $2$, the operator $\cK$ is only of order $1$, and hence the leading symbol of $\cL$ equals the leading symbol of $\cU$. The operator
\begin{equation*}
 \cU_2 \colon \Om^0(\ad P_h)  \rightarrow  \Om^{2n}(\ad P_h)
\end{equation*}
is elliptic \cite[Lemma 7.2.3]{lt}, and therefore the proof of Lemma \ref{lem:Lelliptic} is a direct consequence of our next result.

\begin{lemma}\label{lem:U1elliptic}
The following complex is elliptic
\begin{equation}
\label{eq:complex L om}
\xymatrix@R-2pc{
\Om^0 \ar[r]^{d} & \Omega^1 \ar[r]^{\cU_1} & \Omega^{2n-1} \ar[r]^{ d-\theta_\omega } & \Omega^{2n}.
}
\end{equation}
\end{lemma}
\begin{proof}
 Let $x\in X$ and $v\in T^*_x$, $v\neq 0$. We set $\sigma_{\cU_1}(v)$ to be the symbol of $\cU_1$ evaluated at $v$. Given $\xi\in \Lambda^1T^*_x$, we have
\begin{equation*}
\label{eq:symbol of L om}
 \sigma_{\cU_1}(v)(\xi)= v \wedge T(v \wedge \xi + J v\wedge J\xi) \wedge \omega^{n-2}.
\end{equation*}
Without loss of generality, we assume that $v$ is orthonormal and complete $\lbrace v, Jv \rbrace$ to an orthonormal basis $\lbrace v, Jv, e_2, J e_2, \ldots, e_{n}, Je_{n} \rbrace$ for $\om$ so that
\begin{equation*}
 \label{eq:basis decomposition hat om}
\om = v\wedge Jv +\sum_i e_i\wedge Je_i.
\end{equation*}
Then, in particular
\begin{equation}
\label{eq:first term symbol L om}
 v\wedge Jv\wedge \om^{n-2} = (n-2)! v\wedge Jv \wedge \sum_i \prod_{k\neq i } e_k\wedge J e_k
\end{equation}
and
\begin{equation}
 \label{eq:second term symbol L om}
 v \wedge \om^{n-1} = (n-1)! v\wedge \prod_i e_i\wedge Je_i.
\end{equation}
Assume that $\sigma_{\cU_1}(v)(\xi) = 0$ and decompose $J\xi$ in the chosen basis
\begin{equation*}
 J\xi = c_v v + c_v' Jv + \sum_i c_i e_i + c_i' J e_i.
\end{equation*}
From formulas (\ref{eq:first term symbol L om}) and (\ref{eq:second term symbol L om}) we deduce
\begin{equation*}
 \sum_i (c_i  e_i + c_i'J e_i) \wedge v\wedge Jv \wedge \prod_{k\neq i} e_k \wedge Je_k = \frac{1}{2} \Lambda_\om(v\wedge \xi+Jv\wedge J\xi) v\wedge \prod_i e_i\wedge J e_i.
\end{equation*}
The last equation gives a decomposition in a basis of $\Lambda^{2n-1}T^*_x$, and therefore $c_i=c_i'=0$ for any $i$, and also
$$\Lambda_\om (v\wedge \xi+Jv\wedge J\xi)=0 .
$$
From the last equation $c_v=0$ and hence
$$
\xi = c_v' v=\sigma_d(c_v')
$$
and the sequence (\ref{eq:complex L om}) is elliptic at $\Om^1$. To finish, note that $\sigma_{d-\theta_\omega}(v) = \sigma_d(v)$ is surjective, and hence the proof follows by dimension count.

\end{proof}

In the sequel, we will omit the injections and projections in the complex \eqref{eq:complex L}, and regard $\Omega^1$ and $\Omega^{2n-1}$ as subspaces of the domain and codomain of $\cL$, respectively. Our next goal is to show that $\cL$ induces a Fredholm operator between suitable Hilbert spaces. Using $\om$ and a choice of bi-invariant positive-definite bilinear form on $\mathfrak{g}$, we endow the spaces of differential forms and $\ad P_h$-valued differential forms with $L^2$ norms. Consider the orthogonal decompositions induced by Lemma \ref{lem:Lelliptic}
\begin{equation}\label{eq:orthogonaldecomp}
\begin{split}
\Omega^1 \times \Omega^0(\ad P_h) & =  \Im \; d \oplus \Im \; \cL^* \oplus \cH^1,\\
\Omega^{2n-1} \times \Omega^{2n}(\ad P_h) & = \Im \; \cL \oplus \Im \; (d-\theta_\omega)^* \oplus \cH^{2n-1},
\end{split}
\end{equation}
where
\begin{equation*}\label{eq:definition harmonics L}
 \begin{split}
 \cH^1  & := \ker \; d^*  \cap \ker \; \cL,\\
 \cH^{2n-1} & := \ker \; (d-\theta_\omega) \cap \ker \; \cL^*
 \end{split}
\end{equation*}
are finite dimensional. For the proof of our next result, we need the orthogonal decomposition of the space of $p$-forms (with respect to the $L^2$ inner product given by $\om$) induced by the de Rham complex
\begin{equation*}
\label{eq:complex L om2}
\xymatrix@R-2pc{
\ldots \ar[r]^{d} & \Omega^p \ar[r]^{d} & \Omega^{p+1} \ar[r]^{d} & \ldots,
}
\end{equation*}
and the Morse-Novikov complex
\begin{equation*}
\label{eq:complex L om3}
\xymatrix@R-2pc{
\ldots \ar[r]^{d - \theta_\omega} & \Omega^p \ar[r]^{d- \theta_\omega} & \Omega^{p+1} \ar[r]^{d- \theta_\omega} & \ldots,
}
\end{equation*}
given by
\begin{equation}\label{eq:orthogonaldeRham}
\begin{split}
\Om^p & = \Im \,d \oplus \Im\,d^*\oplus \cH^p_d,\\
\Om^p & = \Im \, (d - \theta_\omega) \oplus \Im\,(d - \theta_\omega)^* \oplus \cH^p_{d-\theta_\omega},\\
\end{split}
\end{equation}
where we use the notation
\begin{equation}\label{eq:harmonicsdeRham}
\begin{split}
\cH_d^p &:=\Om^p\cap\ker d\cap\ker d^*,\\
 \cH^p_{d-\theta_\omega} &:= \Om^p\cap\ker (d - \theta_\omega)\cap\ker (d- \theta_\omega)^*,
\end{split}
\end{equation}
for the spaces of $\omega$-harmonic and twisted $\omega$-harmonic $p$-forms, respectively. We consider the $L^{2,k}$ Sobolev completions of the spaces above, which we denote with a subscript $k$, and use the same notation for the unique extension of the differential operators to the completed spaces.

\begin{proposition}\label{prop:LFredholm}
The restriction of $\cL$ to $(\Im \; d^*)_k \times \Omega^0(\ad P_h)_k$ induces a Fredholm operator
\begin{equation}\label{eq:FredholmL}
  \cL \colon (\Im \; d^*)_k \times \Omega^0(\ad P_h)_k \rightarrow   (\Im \; d-\theta_\omega)_{k-2} \times \Omega^{2n}(\ad P_h)_{k-2}.
\end{equation}
\end{proposition}

\begin{proof}
Using the orthogonal decompositions of $\Omega^1$ given by \eqref{eq:orthogonaldecomp} and the first equation in \eqref{eq:orthogonaldeRham}, we obtain
$$
\Im \; \cL^* \oplus \cH^1 = \Im \; d^* \oplus \cH^1_d \oplus \Omega^0(\ad P_h).
$$
From this, the kernel of the restriction of $\cL$ to $(\Im \; d^*)_k \times \Omega^0(\ad P_h)_k$ is the intersection $\cH^1 \cap  ((\Im \; d^*)_{k} \oplus \Omega^0(\ad P_h)_k)$, which is finite-dimensional. On the other hand, using the orthogonal decomposition of $\Omega^{2n-1}$ induced by the Morse-Novikov complex we have
$$
\Im \; \cL \oplus \cH^{2n-1} = \Im \; (d-\theta_\omega) \oplus \cH^{2n-1}_{d - \theta_\omega}  \oplus \Omega^{2n}(\ad P_h).
$$
Thus, the cokernel of the restriction of $\cL$ to $(\Im \; d^*)_k \times \Omega^0(\ad P_h)_k$ is the intersection $\cH^{2n-1} \cap  (\Im \; (d-\theta_\omega)_{k-2} \oplus \Omega^{2n}(\ad P_h)_{k-2})$, which is finite-dimensional.
\end{proof}

To finish this section, we study the case $[\theta_\omega] = 0$. In this situation, one has $\theta_\omega = d f_\omega$ and \eqref{eq:FredholmL} induces a Fredholm operator (modifying $\mathcal{L}_1$ in \eqref{eq:lin full parameters} by $e^{-f_\omega} \cdot \mathcal{L}_1$)
\begin{equation}\label{eq:FredholmLsigma0}
  \cL \colon (\Im \; d^*)_k \times \Omega^0(\ad P_h)_k \rightarrow   (\Im \; d)_{k-2} \times \Omega^{2n}(\ad P_h)_{k-2},
\end{equation}
which we denote also by $\mathcal{L}$. We will prove that \eqref{eq:FredholmLsigma0} has index zero. Similarly as before, we decompose
\begin{equation*}
\label{eq:L=UKprime}
\cL=\cU' + \cK',
\end{equation*}
where $\cU' = \cU_1' \times \cU_2$ for
$$
\cU_1' (\xi,s) = d\left( e^{-f_\omega}\left(T(d\xi + J d \xi)\right) \wedge \om^{n-2}\right),
$$
and $\cK'$ is of order $1$. The proof follows by a detailed study of the operator $\cU_1'$. By the proof of Lemma \ref{lem:U1elliptic}, the complex
\begin{equation*}
\xymatrix@R-2pc{
\Om^0 \ar[r]^{d} & \Omega^1 \ar[r]^{\cU_1'} & \Omega^{2n-1} \ar[r]^{d} & \Omega^{2n}
}
\end{equation*}
is elliptic, and therefore we obtain finite-dimensional spaces
\begin{equation*}
 \label{eq:definition harmonics xi terms}
 \begin{array}{ccc}
 \cH_{\cU_1'}^1 & := & \ker d^* \cap \ker \cU_1',\\
 \cH_{\cU_1'}^{2n-1} & := & \ker d \cap \ker \cU_1^{'*}.
 \end{array}
\end{equation*}
By definition of $\cU_1'$, there are natural inclusions $\cH^1_d\subset \cH_{\cU_1'}^1$, $\cH_{d}^{2n-1} \subset \cH_{\cU_1'}^{2n-1}$ (see \eqref{eq:harmonicsdeRham}), and we consider the orthogonal decompositions
\begin{equation}
 \label{eq:orthogonalharmonicsdR}
\cH_{\cU_1'}^1 = \cH^1_d \oplus \cV, \qquad \cH_{\cU_1'}^{2n-1} = \cH_{d}^{2n-1} \oplus \cW.
\end{equation}

\begin{lemma}\label{lem:duality harmonics}
Assuming $[\theta_\omega] = 0$, there is an equality $\cU_1^{'*} = * \cU_1^{'} *$. Consequently, the Hodge $*$-operator induces an isomorphism between $\cV$ and $\cW$ and the restricted operator
$$
\cU_1' \colon (\Im \; d^*)_k \rightarrow   (\Im \; d)_{k-2}
$$
is Fredholm with index zero.
\end{lemma}
\begin{proof}
We decompose
\begin{equation*}
\cU_1= d\circ B \circ \tilde T \circ (1+J) \circ d,
\end{equation*}
by setting for any $\alpha\in\Om^2$,
\begin{align*}
B(\alpha) & =  e^{-f_\omega}\alpha \wedge \omega^{n-2},\\
\tilde T(\alpha) & = \alpha_0 + \frac{n-2}{2n(n-1)} \Lambda_\om\alpha\, \om,
\end{align*}
with $\alpha_0$ the trace-free part of $\alpha$. Thus
\begin{equation*}
 \cU_1^{'*} = d^* \circ (1+J)^* \circ \tilde T^*\circ B^* \circ d^*,
\end{equation*}
where we have $d^*= -*d*$ and, on two-forms, $J^*=*J*$. Note that $ \tilde T$ is self-adjoint and $B^*=e^{-f_\omega} \Lambda_\om^{n-2}$. Now, for any $\alpha\in\Om^2$, Lefschetz and type decompositions give
\begin{equation*}
(n-2)! *\alpha = -\alpha_0^{1,1}\wedge \om^{n-2} +\frac{\Lambda_\om\alpha}{n(n-1)} \om^{n-1} +\alpha_0^{2,0+0,2}\wedge \om^{n-2}.
\end{equation*}
Together with $(1+J)\alpha = 2\alpha^{1,1}$ on two-forms and $[\Lambda_\om , L_\om] = (n-p) \Id$ on $p$-forms
we obtain by a direct computation that $\cU_1^{'*}= *\cU_1' *$. To conclude the proof, note that $\xi\in \cH_d^1$ if and only if $*\xi\in \cH_{d}^{2n-1}$. Thus, the result follows from the decompositions \eqref{eq:orthogonalharmonicsdR}.
\end{proof}

\begin{corollary}
\label{cor:index zero linearization}
Assuming $[\theta_\omega] = 0$, the operator \eqref{eq:FredholmLsigma0} is Fredholm and has index zero.
\end{corollary}
\begin{proof}
From Lemma \ref{lem:duality harmonics}, and from the fact that $\cU_2$ in \eqref{eq:lin elliptic term} is elliptic of index zero (see \cite[Lemma 7.2.3]{lt}), we deduce that $\cU' = \cL - \cK'$ is Fredholm with index zero. Since $\cK'$ is of order one, it induces a compact operator by Rellich's Lemma, and hence the result follows.
\end{proof}

\subsection{Stability under deformations}\label{sec:local moduli}

Building on Corollary \ref{cor:index zero linearization}, in this section  we show that, under natural assumptions,
the existence of solutions to the Hull-Strominger system \eqref{eq:Stromhol} is stable under deformations of the Aeppli class, the Bott-Chern algebroid $Q$ and the complex manifold $X$. Note that the operator \eqref{eq:FredholmLsigma0} corresponds to the linearization of the more general Calabi system \eqref{eq:stromabstractredux} in a fixed Aeppli class, regardless of the choice of volume form $\mu$ on $X$. Thus, in this section we work in the generality of a compact complex manifold $X$ endowed with a smooth volume form $\mu$. 

We start by studying the deformations of pairs $(X,Q)$, given by a compact complex manifold $X$ and a Bott-Chern algebroid $Q$. We denote by $P$ the holomorphic principal bundle underlying $Q$. By a result of Ehresmann, a complex deformation of $X$ can be regarded as a smooth family of integrable almost complex structures on the smooth manifold underlying $X$. Furthermore, if we let $Q$ vary with the complex structure, we obtain a complex deformation of the holomorphic bundle $P$, which can be regarded as a smooth family of $(0,1)$-connections on the smooth $G$-bundle $\underline{P}$. Relying on \cite[Def. 4.1]{grt2} and Theorem \ref{lemma:deRhamC}, it is natural to give the following definition.

\begin{definition}
\label{def:deformation Q}
 A Bott-Chern deformation of $(X,Q)$ is a family $(X_t,Q_t)_{t\in B}$ of Bott-Chen algebroids $Q_t$ over complex manifolds $X_t$ parametrized by $B$, satisfying:
 \begin{itemize}
  \item[i)] $(X_0,Q_0)=(X,Q)$,
  \item[ii)] the map $t\mapsto J_t$ is smooth, where $J_t$ is the almost complex structure of $X_t$,
  \item[iii)] for each $t$ there exists a representative $(H_t, \theta_t)$ of the isomorphism class of $Q_t$ as in \eqref{eq:lescEind3}, such that the map $t\mapsto (H_t, \theta_t)$ is smooth.
\end{itemize}
\end{definition}

Given a Bott-Chern deformation of $(X,Q)$,
the Bott-Chern property for $Q_t$ implies that we can choose a more amenable representative of the isomorphism class of $Q_t$, after possibly twisting $Q_t$ by an Aeppli class.
We will denote by $h^{2,2}_{BC}(X)$ the dimension of $H^{2,2}_{BC}(X)$.

\begin{lemma}\label{lem:BCdeformation}
Let $(X_t,Q_t)_{t\in B}$ be a Bott-Chern deformation of $(X,Q)$ with $h^{2,2}_{BC}(X_t)$ constant, and $(H_t, \theta_t)_{t\in B}$ as in Definition \ref{def:deformation Q}.
Then there exists a smooth family $t \to (\tau_t,h_t)$ as in Definition \ref{def:BCtype}, such that $[(2i \partial_t \tau_t,\theta^{h_t})] = \tilde Q_t$ for all $t$, where $(X_t,\tilde Q_t)_{t\in B}$ is a Bott-Chern deformation of $(X,Q)$ related to $(X_t,Q_t)_{t\in B}$ by a family of twists. Furthermore, we can choose the family of reductions $h_t$ to be constant. If, in addition, $X$ satisfies the $\partial \dbar$-Lemma, then $\tilde Q_t \cong Q_t$ for all $t$.
\end{lemma}
\begin{proof}
We fix a reduction $h$ of $\underline{P}$, and denote $\theta^h_t$ the Chern connection of $h$ in the holomorphic principal bundle $P_t$ underlying $Q_t$. Then, $(\tilde H_t,\theta^h_t) \sim (H_t,\theta_t)$, where
\begin{equation}\label{eq:lem:BCdeformation}
\tilde H_t = H_t + CS(\theta_t) - CS(\theta^h_t) - dc(\theta_t \wedge \theta^h_t).
\end{equation}
By the Bott-Chern property, there exists $(\tau_t,h_t)$, not necessarily smooth in $t$, such that $(2i \partial_t \tau_t,\theta^{h_t}) \sim (\tilde H_t,\theta_t^h)$ for all $t$. By Lemma \ref{lem:CSRinvariant}, we have
$$
2i \partial_t \tau_t = \tilde H_t - 2i\partial_t \tilde R(h_t,h) + dB_t,
$$
for $B_t \in \Omega^{2,0}_t$, not necessarily smooth in $t$. Define $\tilde \tau_t = \tau_t + \tilde R(h_t,h)$. It follows from \eqref{eq:lem:BCdeformation} that $(2i \partial_t \tilde \tau_t,\theta^{h}_t) \sim (H_t,\theta_t)$ for all $t$, and furthermore
$$
dd^c_t \tilde \tau_t = c(F_{\theta^{h}_t} \wedge F_{\theta^{h}_t}).
$$
Considering now for each $t$ a hermitian metric $g_t$ on $X_t$, such that $t\mapsto g_t$ is smooth,
we use the decomposition induced by the Aeppli laplacian \cite{Sch}:
$$
\Om^{1,1}_t=\cH^{1,1}_{\Delta^{Ae}_t}\oplus \Im (\del_t \oplus \delb_t) \oplus \Im(\del_t\delb_t)^*.
$$
From $\del_t(\Im (\del_t \oplus \delb_t))\subset d\Om^{2,0}_t$ ,
we can assume that $\tilde\tau_t\in \cH^{1,1}_{\Delta^{Ae}_t} \oplus \Im(\del_t\delb_t)^*$.
We write $\tilde\tau_t = \tau_t^{harm}+(\del_t\delb_t)^* x_t$, for $\tau_t^{harm}\in\cH^{1,1}_{\Delta^{Ae}_t}$ and $x_t\in \Om^{2,2}_t$. Using now the decomposition induced
by the Bott-Chern laplacian \cite{Sch},
$$
\Om^{2,2}_t=\cH^{2,2}_{\Delta^{BC}_t}\oplus \Im (\del_t \delb_t) \oplus \Im(\del_t^* \oplus \delb_t^*),
$$
we can assume $x_t\in \Im (\del_t\delb_t)$. In particular, $\del_t x_t = \delb_t x_t=0$ and thus
$\del_t\delb_t \circ (\del_t\delb_t)^* x_t = \Delta^{BC}_t x_t$.
Then, $c(F_{\theta^{h}_t} \wedge F_{\theta^{h}_t})=2i\del_t\delb_t \tilde\tau_t = 2i\Delta^{BC}_t x_t$,
and $\Delta^{BC}_t x_t$ is smooth. As $h^{2,2}_{BC}(X_t)$ is constant, $G_{\Delta^{BC}_t}\circ \Delta^{BC}_t x_t = x_t$ is smooth, where $G_{\Delta^{BC}_t}$ is the associated Green operator. The final part of the statement follows from Lemma \ref{lem:ddbarisoQ}.
\end{proof}

We are now ready to prove the main result of this section.  We assume that $Q$ is positive, and fix a positive Aeppli class $\sigma\in\Sigma_Q(\RR)$.  We identify the space of Aeppli classes $\Sigma_Q(\RR)$ in $Q$ with a subspace of $H^{1,1}_A(X,\RR)$. We will denote $h^{1,1}_A(X) = \dim H^{1,1}_{A}(X)$.
We fix a volume form $\mu$ on $X$ and consider a solution of the Calabi system \eqref{eq:stromabstractredux} with Aeppli class $\sigma$. Recall that there is an inclusion  $\Ker d \subset \Ker \cL$, where $d \colon \Omega^1 \to \Omega^2$ is the exterior differential acting on forms.

\begin{theorem}\label{theo:stability under deformations}
Assume that $(X,Q)$ admits a solution of the Calabi system with Aeppli class $\sigma$, such that $\Ker d = \Ker \cL$. Let $(X_t,Q_t)_{t\in B}$ be a Bott-Chern deformation of $(X,Q)$ such that $h^{1,1}_A(X_t)$ and $h^{2,2}_{BC}(X_t)$ are constant. Then, for any $t$ small enough, $(X_t,Q_t)$ admits a differentiable family of solutions, parametrized by an open set in $\Sigma_{Q_t}(\RR)$.
\end{theorem}

\begin{proof}
Let $(X_t,Q_t)_{t\in B}$ be a Bott-Chern deformation of $(X,Q)$ and consider $(\tau_t,h)$ as in Lemma \ref{lem:BCdeformation}. We denote by $\sigma_t$ the smooth one-parameter family of compatible Aeppli classes $\sigma_t = [(\tau_t,h)]$ deforming $\sigma$. We can assume that $(\tau_0,h)$ is the given solution of \eqref{eq:stromabstractredux}, and therefore $\omega_t = \operatorname{Re}\tau_t$ is positive for $t$ small enough. Denote
$$
\Pi_h : \Om^{2n}(\ad \underline{P}) \rightarrow \Om^{2n}(\ad P_h)
$$
the natural projection. Consider $\cH^{1,1}_A(X_t,\omega_t)$ the space of $\omega_t$-harmonic Aeppli $(1,1)$ classes on $X_t$ \cite{Sch}, and denote by $\Delta_A^t$ the Aeppli laplacian. As $h^{1,1}_A(X_t)$ is constant, there is a differentiable family of isomorphisms:
$$
\Pi_{A,t} : \cH^{1,1}_{\Delta_{A}}(X) \rightarrow \cH^{1,1}_{\Delta_{A}^t}(X_t).
$$
Taking $k \gg 1$, we define the operator
 \begin{equation*}
  \label{eq:operator IFT strom}
  \begin{array}{cccc}
 \cS : & B \times \cH^{1,1}_{\Delta_{A}}(X)\times  (\Im \, d^*)_k\times \Om^0(\ad P_h)_k & \to & (\Im \, d)_{k-2}\times \Om^{2n}(\ad P_h)_{k-2}\\
       & (t , \gamma, \xi, s) & \mapsto & (\cS_1(t,\gamma,\xi,s),\cS_2(t,\gamma,\xi,s)),
  \end{array}
 \end{equation*}
by
\begin{equation*}
 \begin{array}{ccc}
  \cS_1(t,\gamma,\xi,s) & = &  d (e^{-f_{\tilde\omega_{\gamma,t}}}\tilde\omega_{\gamma,t}^{n-1}),\\
  \cS_2(t,\gamma,\xi,s) & = & \Pi_h(F_{\theta^{h'}_t}\wedge \tilde\om_{\gamma,t}^{n-1}),
 \end{array}
\end{equation*}
where $h'=\exp(is)\cdot h$ and $\theta^{h'}_t$ is the Chern connection of $h'$ on $P_t^c$. Further,
$$
\tilde\om_{\gamma,t}=\om_t+\Pi_{A,t}\gamma + (1+J_t)d\xi +\tilde R(h',h).
$$
Note that for $t$ and $\gamma$ small enough, zeros of $\cS(t,\gamma,\cdot,\cdot)$ are solutions of the system \eqref{eq:stromabstractredux} on $(X_t,\tilde Q_t)$ in the class $\sigma_t+\Pi_{A,t}\gamma$, for $\tilde Q_t = Q_t \otimes \Pi_{A,t}\gamma$.
The differential of $\cS$ at zero with respect to $(\xi, s)$ is \eqref{eq:FredholmLsigma0}. Using that $\Ker d \subset \Ker \cL$, the hypothesis $\Ker d = \Ker \cL$ is equivalent to 
$$
  (\Im \, d^*)_k\times \Om^0(\ad P_h)_k \cap \Ker \cL = \{0\}.
$$
Hence, from Corollary \ref{cor:index zero linearization} the operator \eqref{eq:FredholmLsigma0} is invertible and the implicit function theorem applies. By elliptic regularity the zeros of $\cS$ are smooth, and hence the result follows from Lemma \ref{lem:ddbarisoQ}.
\end{proof}

\begin{remark}\label{rem:LieinKer}
A holomorphic global sections of $Q$ can be realized as an infinitesimal automorphism of $Q$ (see Appendix \ref{sec:AeppliAut}), with action given by the Dorfman bracket. In the model provided by Example \ref{ex:algebroidH}, we can consider the subspace
$$
V_h = \{r + \xi  \; | \; r \in \Omega^0(\ad P_h), \; \xi \in \Omega^{1,0}, \;  \dbar_{\underline{Q}}(r + \xi) = 0  \} \subset H^0(X,Q)
$$
and define a Lie subalgebra (see \cite[Lemma 2.10]{grt2})
\begin{equation}\label{eq:Lie*}
(\Lie \Aut Q)^* = \{r + \partial \xi  \; | \; r + \xi \in V_h \} \subset \Lie \Aut Q.
\end{equation}
Then, 
there is  an inclusion
\begin{equation}\label{eq:LieinKer}
V_h \hookrightarrow \Ker \cL: r + \xi \to \frac{1}{2} (\xi + \overline \xi) + 2 r.
\end{equation}
We expect that, under natural assumptions, $V_h$ provides an orthogonal complement of $\Ker d \subset \Ker \cL$. A confirmation of this expectation would reduce the problem of deformation of solutions of the Calabi system considered in Theorem \ref{theo:stability under deformations} to algebraic geometry.
\end{remark}

\begin{remark}
In complex dimension $3$ 
the spaces $H^{1,1}_A(X)$ and $H^{2,2}_{BC}(X)$ are dual, and thus have the same dimension.
In that case, the statement of Theorem \ref{theo:stability under deformations} is slightly simplified.
\end{remark}

\begin{remark}\label{rem:uniquenessdef}
The proof of Theorem \ref{theo:stability under deformations} relies on the implicit function theorem and thus implies a local uniqueness result for \eqref{eq:stromabstractredux}. More precisely, there is a small neighbourhood $V$ of the solution $(\om,h)$:
\begin{equation*}
V=\lbrace (\om',e^{is}h) \; \vert \; \omega' \in \Omega^2, \;  s \in \Om^0(\ad P_h), \; \vert\vert (\om,h)-(\om',h')\vert\vert_{2,k}< \epsilon \rbrace,
 \end{equation*}
such that any solution of \eqref{eq:stromabstractredux} in $V$ lying on a small Bott-Chern deformation of $(X,Q)$ in a nearby Aeppli class $\sigma'$ comes from a unique differentiable family of solutions $(\om_t,h_t)$
induced by the deformation of $(X,Q,\sigma)$ as in Theorem \ref{theo:stability under deformations}. In particular, when the deformation is trivial, nearby solutions are parametrized by a small neighbourhood $U \subset \Sigma_Q(\RR)$ of $\sigma$.
\end{remark}

\begin{remark}\label{rem:prediction}
If we fix $(X,P)$ in Theorem \ref{theo:stability under deformations} and let $Q$ and $\sigma$ vary, the expected overall dimension of the space of deformed solutions is
\begin{equation}\label{eq:prediction2}
\dim \; \Im \; \partial + \dim \ker \partial = \dim H^{1,1}_A(X,\RR),
\end{equation}
where $\partial$ is as in \eqref{eq:partialmapintro}. The first contribution has to be understood as the number of deformations of $Q$, while the latter corresponds to $\dim \Sigma_Q(\RR)$. Remarkably, combining \cite[Table 1]{Latorre} with \cite[Prop. 2.3]{Ugarte}, equation \eqref{eq:prediction2} matches the number of solutions of the Hull-Strominger system in the nilmanifold $\mathfrak{h}_3$ (see Example \ref{example:Ugarte}) found in \cite{FIVU}, after the normalization \eqref{eq:normalizationnorm}. We thank L. Ugarte for this observation.
\end{remark}

To finish this section, we analyze the consequences of Theorem \ref{theo:stability under deformations} when the $\partial \dbar$-Lemma is satisfied. We need the following technical lemma, which is a consequence of Lemma \ref{lem:BCdeformation}. This result shall be compared with \cite[Prop. 4.2]{grt2}.

\begin{lemma}\label{lem:BCdeformationddbar}
Assume that $X$ is a $\partial \dbar$-manifold.
Then any small complex deformation $(X_t,P_t)$ of $(X,P)$ induces a unique Bott-Chern deformation $(X_t,Q_t)$ of $(X,Q)$
such that the underlying principal bundle of $Q_t$ is $P_t$ for all $t$.
\end{lemma}
\begin{proof}
First, note that $X$ satisfying the $\partial\dbar$-lemma, any small complex deformation $X_t$ of $X$ is a $\partial \dbar$-manifold.
For such deformations, $h^{2,2}_{BC}(X_t)$ is constant, and the natural maps $H^{2,2}_{BC}(X_t) \to H^4(X,\C)$ are injective.

Let $(X_t,P_t)$ be a complex deformation of the pair $(X,P)$, which we regard as a smooth family of $(0,1)_t$-connections on the smooth $G$-bundle $\underline{P}$,
together with a smooth family $(J_t)$ of almost complex structures on $X$.
By Chern-Weyl theory, $p_1(P_t)=0$ in De Rham cohomology for all $t$,
and thus by the previous remark $p_1(P_t)=0$ in $H^{2,2}_{BC}(X_t)$ for all $t$.
Taking a reduction $h$ of $\underline{P}$, we then have
$$
dd_t^c \tau_t = c(F_{\theta^h_t} \wedge F_{\theta^h_t}),
$$
where $\tau_t \in \Omega^{1,1}(X_t)$ and $\theta^h_t$ is the Chern connection of $h$ on $P_t$.
Taking now a smooth family of hermitian metrics on $X$ with respect to $(J_t)$ and considering the Bott-Chern laplacians,
it follows as in the proof of Lemma \ref{lem:BCdeformation} that we can choose $\tau_t$ varying smoothly with $t$.
The family $(2i\partial_t \tau_t, \theta^h_t)$ provides a required Bott-Chern deformation of $(X,Q)$.
Uniqueness for such deformations follows by Proposition \ref{prop:BCclassification}.
\end{proof}

The following specialization of Theorem \ref{theo:stability under deformations} follows by a direct application of Lemma \ref{lem:BCdeformationddbar}. Recall that on a $\partial\dbar$-manifold the Aeppli and Bott-Chern cohomologies are isomorphic to the Dolbeault cohomology $H^{p,q}(X)$. Furthermore, $h^{p,q}(X)$ is constant along any complex deformation.

\begin{corollary}\label{cor:ddbardef}
Assume that $X$ is a $\partial\dbar$-manifold, and that $(X,Q)$ admits a solution of the Calabi system such that $\Ker d = \Ker \cL$.
Then any small complex deformation of $(X,P)$ induces a unique Bott-Chern deformation of $(X,Q)$ admitting a family of solutions of the Calabi system of real dimension $h^{1,1}(X)$.
\end{corollary}

\subsection{Deformation of standard embedding solutions}\label{subsec:deformstemb}

In this section we address the deformation problem for special solutions of the Calabi system \eqref{eq:stromabstractredux}, that we will call \emph{standard embedding solutions} by analogy with a similar situation considered in the physics literature (see Remark \ref{rem:stembedding}).

Let $X$ be a compact complex manifold of dimension $n$ with compatible smooth volume form $\mu$. Let $P$ be a holomorphic principal $G$-bundle with $p_1(P) = 0 \in H^{2,2}_{BC}(X)$. For the next definition, we take the point of view of the principal bundle $P$, and therefore we consider the system \eqref{eq:stromabstract}.

\begin{definition}\label{def:stembsol}
A solution $(\omega,h)$ of the system \eqref{eq:stromabstract} is called a \emph{standard embedding solution} if $G = G' \times G'$ for a reductive Lie group $G'$, $c$ is (a multiple) of the difference of the Killing form $-\tr_{\mathfrak{g}'}$ on the two copies of $\mathfrak{g}'$, that is,
$$
c = \alpha (-\tr_{\mathfrak{g}'} + \tr_{\mathfrak{g}'})
$$
for $\alpha \in \RR$, $P = P_{G'} \times_X P_{G'}$ for a holomorphic $G'$-bundle $P_{G'}$ over $X$ and $h = h_K \times h_K$ for a reduction $h_K$ of $P_{G'}$ to a maximal compact $K \subset G'$.
\end{definition}

For a standard embedding solution $c(F_h \wedge F_h) = 0$ and hence $\omega$ satisfies the system \eqref{eq:stromabstractexact}. Thus, by Proposition \ref{prop:extremaexactcase}, $\omega$ is K\"ahler.
Conversely, given a K\"ahler metric on $X$ satisfying \eqref{eq:volume} and a Hermite-Einstein reduction $h_K$ on $P_{G'}$, we can construct a standard embedding solution by setting $h = h_K \times h_K$. Furthermore, when $X$ is Calabi-Yau and $\mu$ is given by \eqref{eq:muOmega}, the metric has holonomy contained in $\SU(n)$.

In order to address the deformation problem for standard embedding solutions, we need first to understand the operator $\cU_1'$ in Lemma \ref{lem:duality harmonics} when $\omega$ is K\"ahler.

\begin{proposition}\label{prop:vanishing V W Kahler}
Assume that $\om$ is K\"ahler. Then $\cV$ and $\cW$ in \eqref{eq:orthogonalharmonicsdR} vanish, and therefore $\cU_1'$ in Lemma \ref{lem:duality harmonics} is invertible.
\end{proposition}
\begin{proof}
From Lemma \ref{lem:duality harmonics}, it is enough to show that $\cV=0$.
Let $\xi\in\cV$. To simplify notations, set $\gamma=(1+J)d\xi=2(d\xi)^{1,1}$.
By definition of $\cU_1'$, and as $\om$ is K\"ahler (note that $e^{-f_\om}$ is constant), we have
\begin{equation*}
 \label{eq:equation d xi Kahler}
 d\gamma \wedge \om^{n-2} = \frac{1}{2(n-1)} (d \Lambda_\om \gamma) \wedge \om^{n-1}.
\end{equation*}
The left hand side in the above identity satisfies
\begin{equation*}
 d^c d\gamma\, \wedge \om^{n-2}= 0.
\end{equation*}
Thus the right hand side is also $d^c$-closed. Applying $d^c$, we obtain
\begin{eqnarray*}
 \Delta_\om (\Lambda_\om \gamma)=0,
\end{eqnarray*}
and the function $\Lambda_\om \gamma$ is constant so, using that
$\Lambda_\om \alpha=0$ for any $\alpha\in\Om^{2,0 + 0,2}$, we conclude that
$$
\Lambda_\om \gamma = \Lambda_\om(1+J)d\xi = 2 \Lambda_\om (d\xi).
$$
From the K\"ahler identities,
$$
\Lambda_\om d \xi = d\Lambda_\om \xi - (d^c)^*\xi=-(d^c)^* \xi
$$
and, as $\Lambda_\om d \xi$ is constant, it must vanish identically.
Thus, $d\xi$ is primitive. Using now
$$
d^c \gamma \wedge \om^{n-2}=0
$$
and $\gamma= (1+J)d\xi= d\xi + d^c(J\xi)$,
we have that $d\xi$ also satisfies
$$
d^cd\xi \wedge \om^{n-2}=0.
$$
We deduce by Stokes' theorem that
\begin{equation}
\label{eq:integral norm gamma is zero}
 \int_X  d^c(J\xi) \wedge  d\xi \wedge \om^{n-2} =\int_X J(d\xi)\wedge d\xi\wedge \om^{n-2}= 0.
\end{equation}
Note that $Jd\xi= (d\xi)^{1,1} - (d\xi)^{2,0+0,2}$.
Recall also that $d\xi$ being primitive implies that
$$
*(d\xi)^{1,1} = - (d\xi)^{1,1} \wedge \frac{\om^{n-2}}{(n-2)!},
$$
whereas
$$
*(d\xi)^{2,0+0,2} = (d\xi)^{2,0+0,2}\wedge \frac{\om^{n-2}}{(n-2)!}.
$$
From (\ref{eq:integral norm gamma is zero}), we deduce that the $L^2$-norm $\vert\vert d\xi \vert\vert_\om$ is zero, and $d\xi= 0$.
By definition of $\cV$ we conclude that 
$\xi$ must be zero, and the proof is complete.
\end{proof}


We fix a standard embedding solution $(\omega,h)$ of \eqref{eq:stromabstract} as in Definition \ref{def:stembsol}.
Note that for any constant $r>0$, the pair $(r \om,h)$ is also a standard embedding solution, and it defines a one-parameter family of hermitian metrics on the Bott-Chern algebroid $Q$ with holomorphic string class $[(0,\theta^h)]$. Denote by $\cL_{r\omega,h}$ the linearization of the Calabi system at $(r \om,h)$ with fixed Aeppli class $\sigma = [(r\omega,h)]$. Following an idea by Li and Yau \cite{LiYau}, we have the following.

\begin{lemma}\label{lem:rescaling trick}
Assume that $G'$ is semisimple and that the Chern connection of $h_K$ is irreducible. Then, there is a constant $r_0>0$ such that for any $r\geq r_0$, the operator \eqref{eq:FredholmLsigma0} induced by $\cL_{r\om,h}$  is invertible.
\end{lemma}
\begin{proof}
As in \cite{LiYau}, we evaluate $\cL_{r\om,h}$ at $(r\xi,s)$ :
\begin{equation*}
 \cL_{r\om,h}(r\xi, s) = r^{n-1} \tilde \cU(\xi, s) + r^{n-2} \tilde \cK( \xi,s)
\end{equation*}
with
\begin{eqnarray*}
  \tilde \cU(\xi, s) & = &  (\cU_1'(\xi), \: -2i \dbar \partial^h s \wedge \om^{n-1}+(n-1)F_h\wedge (1+J)d\xi\wedge \om^{n-2} ) \\
   \tilde \cK(\xi, s) & = & (d(2c(s, F_h)\wedge \om^{n-2}), \; (n-1) F_h\wedge c(2s, F_h)\wedge \om^{n-2} ).
\end{eqnarray*}
As $\om$ is K\"ahler, we have $2i \dbar \partial^h s \wedge \om^{n-1} = (\Delta_h s) \omega^n/n$, where $\Delta_h s$ is the laplacian of the Chern connection of $h$. Furthermore, as the Chern connection of $h_K$ is irreducible $\operatorname{Ker} \Delta_h s$ is identified with the centre of $\mathfrak{g}' \oplus \mathfrak{g}'$, which vanishes because $G'$ is semisimple, and therefore by Proposition \ref{prop:vanishing V W Kahler} the operator $\tilde \cU$ is invertible.
As
$$
r^{1-n}\cL_{r\om,h}(r\xi, s)= \tilde \cU (\xi, s) + r^{-1} \tilde \cK(\xi, s),
$$
for $r$ large enough, $r^{1-n}\cL_{r\om,h}$ is invertible.
\end{proof}

Combining Corollary \ref{cor:ddbardef} with Lemma \ref{lem:rescaling trick} we obtain the following result.
Recall that on a K\"ahler manifold the Aeppli cohomology is isomorphic to the Dolbeault cohomology $H^{p,q}(X)$.
Furthermore, $h^{p,q}(X)$ is constant along any complex deformation.

\begin{corollary}\label{cor:stability under deformations}
Assume that $(\omega,h)$  is a standard embedding solution of \eqref{eq:stromabstract}, with $G'$ semisimple and $h_K$ irreducible.
Then, up to scaling of $\om$,
any small complex deformation of $(X,P)$ admits an $h^{1,1}(X)$-dimensional differentiable family of solutions of
\eqref{eq:stromabstractredux}.
\end{corollary}

\begin{example}
Let $X$ be a projective complete intersection Calabi-Yau threefold.
Consider the standard embedding solution of the Hull-Strominger system in Remark \ref{rem:stembedding},
given by a Calabi-Yau metric $g$ on $X$ with K\"ahler form $\omega$ and induced hermitian metric $h$ on $TX \oplus TX$.  By \cite[Cor. 2.2]{Huyb}, the tangent bundle $TX$ has unobstructed deformations, parametrized by $ H^1(\End TX)$. By Lemma \ref{lem:BCdeformationddbar} the space
$$
 H^1(\End TX) \oplus H^1(\End TX).
$$
parametrizes the unobstructed Bott-Chern deformations of $Q$, with fixed $X$.
From Corollary \ref{cor:stability under deformations}, we obtain a family of solutions of the Hull-Strominger system obtained by deformation of the standard embedding solution of real dimension
$$
h^{1,1}(X) + 4 h^1(\End TX).
$$
Among these deformations there is a codimension $2 h^1(\End TX)$ family which corresponds to standard embedding solutions. Away from this locus, the deformed solutions are non-K\"ahler.
For the quintic hypersurface $h^1(\End TX) = 224$ and $h^{1,1}(X) = 1$,
and we obtain a family of non-K\"ahler solutions of real dimension $897$.
\end{example}

\appendix

\section{Aeppli classes and automorphisms}\label{sec:AeppliAut}

In this section we define the action of the group of automorphisms of a Bott-Chern algebroid $Q$ on its space of hermitian metrics $B_Q^+$, and study its effect on Aeppli classes. 

Let $Q$ be a Bott-Chern algebroid over a compact complex manifold $X$ with underlying principal bundle $P$. Denote by $\cG_P$ the group of automorphisms of $P$ projecting to the identity on $X$. Recall that an automorphism of $Q$ is given by a pair $(\varphi,g)$ as in \eqref{eq:defstringiso} with $g \in \cG_P$. By \cite[Prop. 2.12]{grt2}, there is a short exact sequence of groups
\begin{equation}\label{eq:AutW}
\xymatrix{
	0 \ar[r] & \Omega^{2,0}_{cl} \ar[r] & \Aut Q \ar[r] & \Ker \sigma_P \ar[r] & 1
}
\end{equation}
where $\sigma_P \colon \cG_P \to H^1(\Omega^{\leqslant \bullet})$ is the homomorphism \eqref{eq:sigmaP}, $\Aut  Q$ is the group of automorphisms of $Q$, and $\Omega^{2,0}_{cl}$ denotes the space of closed $(2,0)$-forms on $X$. Here, the map from $\Aut Q$ to $\Ker \sigma_P$ corresponds to the projection $(\varphi,g) \mapsto g$. Our next result defines an $\Aut Q$-action on $B_Q$ preserving the set of hermitian metrics $B_Q^+ \subset B_Q$ (see Section \ref{subsec:holCourhermit}).

\begin{lemma}\label{lem:action}
Let $Q$ be a Bott-Chern algebroid. Then, there is a well-defined action of $\Aut Q$ on $B_Q$ given by
$$
(\varphi,g) \cdot (\tau,h) = (\tau,gh)
$$
for any $(\varphi,g) \in \Aut Q$ and $(\tau,h) \in B_Q$.
\end{lemma}

\begin{proof}
Given $g \in \cG_P$ one has that $(\tau,h)$ and $(\tau,gh)$ define the same holomorphic string class if and only if $g \in \Ker \sigma_P$. The statement follows now from the definition of $B_Q$ and the exact sequence \eqref{eq:AutW}. 
\end{proof}

\begin{remark}
The $\Aut Q$-action in Lemma \ref{lem:action} can be recovered from \cite[Lem. A.2]{grt3}. Essentially, it corresponds to the restriction of the natural action of the group of automorphisms of a smooth complex string algebroid on its space of \emph{compact forms}. 
\end{remark}

Our next goal is to clarify the effect of the $\Aut Q$-action on Aeppli classes (see Definition \ref{def:ApQ}). For this, we first lift the homomorphism $\sigma_P$ to the Aeppli cohomology of $X$ using the Bott-Chern secondary characteristic class $R$ defined in Proposition \ref{prop:Donaldson}.

\begin{lemma}\label{lemma:tildesigmaP}
Let $P$ be a holomorphic principal $G$-bundle over $X$. Then, there is a homomorphism of groups
\begin{equation}\label{eq:tildesigmaP}
\widetilde \sigma_P \colon \cG_P \to H^{1,1}_A(X,\RR)
\end{equation}
defined by 
$$
\widetilde \sigma_P(g) = [R(h,gh)] \in H^{1,1}_A(X,\RR)
$$
for any choice of reduction $h \in \Omega^0(P/K)$, inducing a commutative diagram 
	\begin{equation}\label{eq:tildesigmaPdiagram}
	\xymatrix{
		& H^{1,1}_A(X,\RR) \ar[d]^{2i\partial}\\
		\cG_P \ar[r]^{\sigma_P \;\; } \ar[ru]^{\widetilde \sigma_P \;\; }  & H^1(\Omega^{\leqslant \bullet}) .
	}
	\end{equation}
\end{lemma}

\begin{proof}
Property $(3)$ in Proposition \ref{prop:Donaldson} implies that $R(h,gh)$ induces a well-defined class $[R(h,gh)] \in H^{1,1}_A(X,\RR)$, for any choice of reduction $h$ and $g \in \cG_P$. Property $(2)$ in Proposition \ref{prop:Donaldson} implies that
\begin{equation}\label{eq:RautP}
R(g h_1,g h_0) = R(h_1,h_0),
\end{equation}
and therefore, given another reduction $h'$, property $(1)$ in Proposition \ref{prop:Donaldson} implies
	\begin{align*}
	[R(h',gh')] & =  [R(h',gh)] + [R(gh,gh')]\\
	& = [R(h',h)] + [R(h,gh)] + [R(h,h')] = [R(h,gh)],
\end{align*}
and hence \eqref{eq:tildesigmaP} is independent of choices. Using Lemma \ref{lem:CSRinvariant}, we also have
\begin{align*}
[2i \partial R(h,gh)] & = [CS(\theta^{gh}) - CS(\theta^h) - d \la \theta^{gh} \wedge \theta^{h} \ra]\\
& = [CS(g\theta^{h}) - CS(\theta^h) - d \la g\theta^{h} \wedge \theta^{h} \ra] \in H^1(\Omega^{\leqslant \bullet}),
\end{align*}
which proves the commutativity of the diagram \eqref{eq:tildesigmaPdiagram}.  Finally, using again \eqref{eq:RautP} and Proposition \ref{prop:Donaldson}
\begin{align*}
[R(h,gg'h)] & =  [R(h,gh)] + [R(gh,gg'h)] = [R(h,gh)] + [R(h,g'h)],
\end{align*}
which proves that \eqref{eq:tildesigmaP} is a homomorphism, as claimed.
\end{proof}

Observe that, from the commutative diagram \eqref{eq:tildesigmaPdiagram}, we have an equality 
$$
\widetilde{\sigma}(\Ker \sigma_P) = \Im \; \widetilde \sigma_P \cap \Ker \partial.
$$
We are ready to give our formula for the change of Aeppli class under the action of an element in $\Aut Q$.

\begin{proposition}\label{Prop:AeppliAut}
Let $Q$ be a Bott-Chern algebroid. Then, for any $(\varphi,g) \in \Aut Q$ and $(\tau,h) \in B_Q$, we have
$$
[(\varphi,g) \cdot (\tau,h)] = [(\tau,h)] + \widetilde \sigma_P(g) \in \Sigma_Q(\RR).
$$
\end{proposition}

\begin{proof}
The proof follows from property $(1)$ in Proposition \ref{prop:Donaldson} and the definition of $\widetilde \sigma_P$. For instance,
$$
[(\varphi,g) \cdot (\tau,h)] - [(\tau,h)]  = Ap(\tau,gh,\tau,h) = - [R(gh,h)] = \widetilde \sigma_P(g) \in H^{1,1}_A(X,\RR).
$$
\end{proof}

The previous result creates a potential issue regarding the proposed picture for the uniqueness problem for the twisted Hull-Strominger system \eqref{eq:twistedStromholredux}, since the $\Aut Q$-action on $B_Q$ does not necessarily preserve Aeppli classes. A possible solution to this caveat is to consider a notion of \emph{big Aeppli classes} on $Q$ replacing $H^{1,1}_A(X,\RR)$ in Proposition \ref{propo:Aepplicone} by the natural quotient
$$
H^{1,1}_A(X,\RR)/\Im \; \widetilde \sigma_P.
$$
In our final result we identify constant directions in $\Sigma_Q(\RR)$ under the $\Aut Q$-action, via the duality isomorphism $H^{1,1}_A(X) \cong H^{n-1,n-1}_{BC}(X)^*$. This is based on a remarkable relation between $\widetilde \sigma_P$ and the \emph{Futaki Invariant} of the holomorphic principal bundle $P$. Similarly as in the theory of K\"ahler-Einstein metrics \cite{Futaki1}, this invariant is given by a character 
\begin{equation}\label{eq:Futaki}
    \xymatrix{
 \mathcal{F} \colon \Lie \; \cG_P \ar[r] & H^{1,1}_A(X,\CC).
 }
\end{equation}
defined by
$$
 \mathcal{F} (s) = [c(s,F_h)] \in H^{1,1}_A(X,\CC)
$$
for any choice of reduction $h$ on $P$. 
We can regard $\mathcal{F}$ as an obstruction to the existence of solutions of the Hermite-Einstein equation for a given choice of balanced metric $\tilde \omega$. Indeed, considering the balanced class $\mathfrak{b} = [\tilde \omega^{n-1}] \in H^{n-1,n-1}_{BC}(X,\RR)$, the Lie algebra character
$$
\langle \mathcal{F},\mathfrak{b} \rangle \colon \Lie \; \cG_P \to \mathbb{C}
$$
vanishes provided that there exists a solution $h$ to $F_h \wedge \tilde \omega^{n-1} = 0$. As we will see next, $\widetilde \sigma_P$ can be regarded as an integral (real) version of $\cF$.

\begin{proposition}\label{prop:Futaki}
Let $P$ be a holomorphic principal $G$-bundle over $X$. Let $(\cG_P)_0 \subset \cG_P$ denote the identity component of $\cG_P$. Let $\mathfrak{b} \in H^{n-1,n-1}_{BC}(X,\RR)$ such that there exists a balanced metric $\tilde \omega$ with $\mathfrak{b} = [\tilde \omega^{n-1}]$ and a reduction $h$ in $P$ satisfying $F_h \wedge \tilde \omega^{n-1}  = 0$. Then, for any $g \in (\cG_P)_0$ one has
$$
\widetilde \sigma_P(g) \cdot \mathfrak{b} = 0.
$$
\end{proposition}

\begin{proof}%
Let $g_t$ be a curve joining $g \in (\cG_P)_0$ to $1 \in \cG_P$. Setting $h_t = g_t h$ and using property $(2)$ in Proposition \ref{prop:Donaldson}, we calculate
$$
\frac{d}{dt}\widetilde \sigma_P (g_t) \cdot \mathfrak{b} = \frac{d}{dt} \int_X R(h,h_t) \wedge \tilde \omega^{n-1}= 2i \int_X c(\dot h_t h_t^{-1},F_{h_t}) \wedge \tilde \omega^{n-1} = 0,
$$
where we have used that $F_{h_t} = g_t F_h$ and the invariance of the pairing $c$. The claim follows from $\widetilde \sigma_P (1) = 0$.
\end{proof}

\vspace{-.6cm}

\end{document}